\documentclass[12pt]{article}
\usepackage{geometry}
\usepackage{amssymb}
\usepackage{tikz}
 \usepackage{xypic}
\usepackage[all,2cell]{xy}
\usepackage{graphicx}
\usepackage{amsmath}
\usepackage{fullpage}
\usepackage{amsthm}
\usepackage{mathtools}
\usepackage{leftidx}
\def\cP{\mathcal P}
\usepackage{xcolor, soul}
\sethlcolor{yellow}
\allowdisplaybreaks
\usepackage{mathrsfs}
\usepackage{calligra}
\usepackage[all,2cell]{xy}
\UseTwocells

\newtheorem{theorem}{Theorem}[section]
\newtheorem{corollary}[theorem]{Corollary}
\newtheorem{lemma}[theorem]{Lemma}
\newtheorem{proposition}[theorem]{Proposition}

\newtheorem{definition}[theorem]{Definition}
\newtheorem{example}[theorem]{Example}
\newtheorem{remark}[theorem]{Remark}

\newtheorem{notation}[theorem]{Notation}

\usepackage{hyperref}
\begin{document}

\title{The Squire complex and its extensions revisited}
\author{Elton Pasku \\
Universiteti i Tiran\"es \\
Fakulteti i Shkencave Natyrore \\
Departamenti i Matematik\"es\\ 
Tiran\"e, Albania \\
elton.pasku@fshn.edu.al}

\date{}

\maketitle

\begin{abstract}
The Squier complex $\mathcal{D}(\cP)$ has been defined in \cite{LDHM1} as a combinatorial 2-complex. Also several extensions of $\mathcal{D}(\cP)$ are defined and used extensively in \cite{LDHM2}, \cite{SES} and \cite{EP}. Since the results obtained in those papers rely on the assumption that $\mathcal{D}(\cP)$ and its extensions in \cite{LDHM2} and \cite{SES} are CW-complexes of dimensions 2 and 3, respectively, and no proofs are given to support these preconditions, we provide the respective proofs in the present paper.
\end{abstract}

\maketitle

\section{Some auxiliary notions and results on cell and CW-complexes}

\begin{definition}
\label{CW-complex}A \textit{CW-complex} is an ordered triple $(X,E,\Phi )$, where $X$ is a Hausdorff space, $E$ is a family of cells in $X$, and $\Phi =\{\Phi _{e}\mid e\in E\}$ is a family of characteristic maps such that

\begin{enumerate}
\item  $X=\bigcup \{e\mid e\in E\}$ (disjoint union);

\item  for each $k$-cell $e\in E$, the map $\Phi_{e}:(D^{k},S^{k-1})\longrightarrow (e\cup X^{(k-1)},X^{(k-1)})$ is a relative homeomorphism;

\item  if $e\in E$, then its closure is contained in a finite union of cells in $E$,

\item  $X$ has the weak topology determined by $\{\overline{e}\mid e\in E\}$: A set $A\subset X$ is open (or closed) if and only if $A\cap
\overline{e}$ is open (or closed) in $\overline{e}$ for every $e\in E$.
\end{enumerate}
\end{definition}

\bigskip

An alternative way to describe the topology on $X$ is to say that a set $A \subseteq X$ is open (or closed) iﬀ $\Phi^{-1}_{e} (A)$ is open (or closed) in $D^{n}$ for each characteristic map $\Phi_{e}$ (see \cite{AH} p. 519). It is proved in Lemma 8.15, p.200 of \cite{Rotman-T} or in p.193 of \cite{Schubert} that, for every $k$-cell $e\in E$, $\Phi_{e}(D^{k})=\overline{e}$. We call $\Phi _{e}(D^{k})$ a \textit{closed} $k$-cell. The restriction of $\Phi _{e}$ on the boundary $S^{k-1}$ of $D^{k}$ is called the \textit{attaching map} of $e$. 

A \textit{subcomplex} of a CW-complex $X$ is a subspace $A\subset X$ which is a union of cells of $X$, such that the closure of each cell in $A$ is contained in $A$. 

\begin{definition}
If $X$ is a CW-complex, hence a disjoint union of cells: $X=\cup\{e\mid e\in E\}$, then define, for each $k\geq 0$, the $k$-\textit{skeleton} $X^{(k)}$ of $X$ by
\begin{equation*}
X^{(k)}=\cup \{e\in E\mid \text{dim}(e)\leq k\}\text{.}
\end{equation*}
We then have, $X^{(k)}\subset X^{(k+1)}$ for every $k\geq 0$, and
$X=\cup _{k\geq 0}X^{(k)}$.
\end{definition}
In practice the construction of CW-complexes is done in an inductive way as follows:

\begin{enumerate}
\item[(1)]  Start with a discrete set $X^{(0)}$, the 0-cells of $X.$

\item[(2)]  Inductively, form the $n$-skeleton $X^{(n)}$ from $X^{(n-1)}$ by attaching $n$-cells $e_{n}^{\alpha }$ via maps $\phi _{\alpha }:S^{n-1}\longrightarrow X^{(n-1)}$. This means that $X^{(n)}$ is the quotient space $X^{(n-1)}\coprod_{\phi_{\alpha }}D_{\alpha }^{n}$.

\item[(3)]  $X=\cup _{n\geq 0}X^{(n)}$.
\end{enumerate}
It is a consequence of 4.5.3, p. 120 of \cite{Brown} that the construction of $X^{(n)}$ from $X^{(n-1)}$ can be represented by the following pushout diagram
\begin{equation*}
\xymatrix@R=30pt @C=60pt{ \sqcup_{\alpha}S^{n-1}_{\alpha} \ar[r]^{j} \ar[d]^{\phi} &  \sqcup_{\alpha}D^{n}_{\alpha} \ar[d]^{\nu j_{2}}\\
X^{(n-1)} \ar[r]^{\nu j_{1}} & X^{(n-1)} \sqcup_{\phi_{\alpha}}D^{n}_{\alpha}}
\end{equation*}
where $j$ is the obvious inclusion map, $\phi$ is the ''collection'' of the attaching maps of the $n$-cells, $j_{1}: X^{(n-1)} \rightarrow X^{(n-1)} \sqcup_{\alpha} D^{n}_{\alpha}$ and $j_{2}: \sqcup_{\alpha} D^{n}_{\alpha}\rightarrow X^{(n-1)} \sqcup_{\alpha} D^{n}_{\alpha}$ are the inclusion maps, $X^{(n-1)} \sqcup_{\phi_{\alpha}} D^{n}_{\alpha}$ is the quotient space of $X^{(n-1)} \sqcup_{\alpha} D^{n}_{\alpha}$ by the equivalence relation generated by all pairs $(u, \phi_{\alpha}(u))$ for all $\alpha$ and $u \in S_{\alpha}^{n-1}$, and $\nu: X^{(n-1)} \sqcup_{\alpha} D^{n}_{\alpha} \rightarrow X^{(n-1)} \sqcup_{\phi_{\alpha}} D^{n}_{\alpha}$ is the quotient map.

\begin{theorem} \label{numcw}(Theorem 7.3.16 of \cite{Maunder})
\label{KxL}If $K$ and $L$ are CW-complexes, so is the topological product $K\times L$ provided that

\begin{enumerate}
\item[(a)]  one of $K$, $L$ is locally compact, or

\item[(b)]  both $K$ and $L$ have a countable number of cells.
\end{enumerate}
\end{theorem}

There is also a more general notion than that of a CW-complex, the one of a\textit{\ cell complex}.

\begin{definition} (Definition 1.1 p. 6 of \cite{LW})
Let $X$ be a set. A \textit{cell structure} on $X$ is a pair $(X,\mathbf{\Phi })$ where $\Phi $ is a collection of maps of closed disks into $X$ satisfying the
following conditions.

\begin{enumerate}
\item[(i)]  If $\Phi \in \mathbf{\Phi }$ and $\Phi $ has domain $D^{n}$, then $\Phi $ is injective on $D^{n}-S^{n-1}$.

\item[(ii)]  The images $\{\Phi (D^{n}-S^{n-1})\mid \Phi \in \mathbf{\Phi }\} $ form a partition of $X$, i.e. they are disjoint and have union $X$.

\item[(iii)]  If $\Phi \in \mathbf{\Phi }$ and $\Phi $ has domain
$D^{n}$, then
\begin{equation*}
\Phi (S^{n-1})\subset \bigcup \{\Psi (D^{k}-S^{k-1})\mid \Psi \in \mathbf{\Phi }\text{ has domain }D^{k}\text{ and }k\leq n-1\}.
\end{equation*}
\end{enumerate}
\end{definition}

\bigskip

Note that in this definition we do not have any topology specified on $X$. If $\Phi \in \mathbf{\Phi }$ and $\Phi $ has domain $D^{n}$, we call $\sigma ^{n}=\Phi (D^{n})$ a closed $n$-cell and $\Phi $ its characteristic function. Also we denote $\Phi(S^{n-1})=(\sigma ^{n})^{\circ }$ and by abuse of language we call $\Phi (D^{n} - S^{n-1})$ an \textit{open }$n$ \textit{-cell}. It follows from Lemma 1.2, p. 7 of \cite{LW} that $\Phi(D^{n} - S^{n-1})=\sigma_{n} - (\sigma_{n})^{\circ}$ \newline

Two cell structures $(X,\mathbf{\Phi })$ and $(X,\mathbf{\Phi}^{\prime })$ are \textit{strictly equivalent} if there is a
one-to-one correspondence between $\mathbf{\Phi }$ and $\mathbf{\Phi }^{\prime }$ such that for a characteristic function
$\Phi \in \mathbf{\Phi }$ with domain $D^{n}$ there corresponds a characteristic function $\Phi ^{\prime }\in \mathbf{\Phi }^{\prime }$ with domain $D^{n}$ again, and there is a
homeomorphism of pairs $h:(D^{n},S^{n-1})\longrightarrow (D^{n},S^{n-1})$ such that $\Phi ^{\prime }=\Phi \circ h$. Strict
equivalence is an equivalence relation and for some cell structure $(X,\mathbf{\Phi })$ we denote by $\mathbb{S}_{\mathbf{\Phi }} $ the set of all pairs $(\sigma ^{n},[\Phi ])$, where $\sigma^{n}=\Phi(D^{n})$ and $[\Phi ]$ is the strict equivalence class of $\Phi \in \mathbf{\Phi }$. Obviously, if $(X,\mathbf{\Phi })$ and $(X,\mathbf{\Phi}^{\prime}) $ are strictly equivalent, then $\mathbb{S}_{\mathbf{\Phi }}=\mathbb{S}_{\mathbf{\Phi }^{\prime }}$.

\begin{definition} (Definition 1.3 p. 8 of \cite{LW})
A \textit{cell complex} on a set $X$ is an equivalence class of cell structures $(X,\mathbf{\Phi })$ under the equivalence relation
of strict equivalence. A cell complex on $X$ will be denoted by a pair $(X,\mathbb{S})$ where $\mathbb{S}=\mathbb{S}_{\mathbf{\Phi }}$ for some representative cell structure $(X,\mathbf{\Phi })$. The set $\mathbb{S}$ is called the set of \textit{closed cells} of $(X,\mathbb{S})$.
\end{definition}

\begin{definition} (Definition 1.4 p. 8 of \cite{LW})
A subcomplex $(A,\mathbb{J})$ of a cell complex $(X,\mathbb{S})$ is a cell complex such that $A\subseteq X$ and $\mathbb{J}\subseteq \mathbb{S}$.
\end{definition}

It is easy to see that every CW-complex is a cell complex and that every CW-subcomplex of a CW-complex is a cell subcomplex of it. 

\begin{definition} (Definition 6.1 p. 32 of \cite{LW})
Let $(X,\mathbb{S})$ be a cell complex and $\mathbb{R}$ an equivalence relation on $X$. Denote by $p$ the quotient map. Then $\mathbb{R}$ is a \textit{cellular equivalence} relation provided the following conditions are satisfied.

\begin{enumerate}
\item[$($1$)$]  If $\sigma \in \mathbb{S}$, then $p^{-1}p(\sigma-\sigma ^{\circ })$ is a union of open cells $\sigma _{i}-\sigma_{i}^{\circ }$ of the cellular partition of $X$.

\item[$($2$)$]  If $\sigma _{0}-\sigma _{0}^{\circ }\in p^{-1}p(\sigma -\sigma ^{\circ })$ is of minimal dimension among all such open cells in the
union, then $p\mid (\sigma _{0}-\sigma _{0}^{\circ })$ is a bijection onto $p(\sigma -\sigma ^{\circ })$ and $p(\sigma )=p(\sigma _{0})$. Such a cell $\sigma _{0}$ is called $\mathbb{R}$-\textit{minimal} for the cell $\sigma $.

\item[$($3$)$]  If $\sigma ^{\prime }$ and $\sigma ^{\prime \prime}$ are both $\mathbb{R}$-minimal for the cell $\sigma $ and if
$\Phi ^{\prime }$ and $\Phi ^{\prime \prime }$ are the respective characteristic functions, then there is a homeomorphism $h:D_{\sigma ^{\prime }}\longrightarrow D_{\sigma ^{\prime \prime}}$ such that $p\Phi ^{\prime }=p\Phi ^{\prime \prime }h$.
\end{enumerate}
\end{definition}

\begin{theorem} \label{cc} (Theorem 6.2 p. 32 of \cite{LW})
Let $(X,\mathbb{S})$ be a cell complex and $\mathbb{R}$ a cellular equivalence relation on $X$. Define $\mathbb{S}/\mathbb{R}=\{p(\sigma )\mid \sigma \in \mathbb{S}$ and $\sigma $ is $\mathbb{R}$-minimal$\}$. Then $(X/\mathbb{R},\mathbb{S}/\mathbb{R})$ is a cell complex.
\end{theorem}

\begin{definition} (Definition 6.3 p. 33 of \cite{LW})
If $(X,\mathbb{S})$ is a cell complex and $\mathbb{R}$ a cellular equivalence relation on $X$, the complex $(X/\mathbb{R},\mathbb{S}/\mathbb{R})$ is the \textit{quotient} or \textit{identification} complex of $(X,\mathbb{S})$ with respect to $\mathbb{R}$.
\end{definition}

\begin{definition} (Definition 4.1 p 27 of \cite{LW})
Let $(X,\mathbb{S})$ and $(Y,\mathbb{J})$ be cell complexes. A map $f: (X,\mathbb{S}) \rightarrow (Y,\mathbb{J})$ is called cellular provided that for each $n$, $f(X^{(n)}) \subseteq Y^{(n)}$. The map $f$ is called regular if it is cellular and for each $\sigma^{n} \in X^{(n)}$, if $f(\sigma^{n})=\tau^{m} \in Y^{(m)}$, then $f(\sigma^{n}-(\sigma^{n})^{\circ})=\tau^{m}-(\tau^{m})^{\circ}$.
\end{definition}

Suppose that $X$ is a space and for each $\lambda \in \Lambda $ there is given an attaching map\newline $f_{\lambda }:$ $\partial D_{\lambda }\longrightarrow X$, with $\partial D_{\lambda }$ the boundary of $D_{\lambda }$. We let
\begin{equation*}
\mathcal{B}=\bigcup \{D_{\lambda }\mid \lambda \in \Lambda \},\text{ }\partial \mathcal{B}=\bigcup \{\partial D_{\lambda }\mid \lambda \in \Lambda
\}
\end{equation*}
and
\begin{equation*}
F=\bigcup_{\lambda }f_{\lambda }:\partial \mathcal{B}\longrightarrow X
\end{equation*}
be the union map. With these notations we have the following two propositions.

\begin{proposition} \label{hd}(Proposition 2.1, p. 45 of \cite{LW})
Let $X$ be a Hausdorff space, and suppose that $Y=\mathcal{B}\coprod_{F}X$ is obtained by attaching the cells $\{D_{\lambda}\mid \lambda \in \Lambda \}$ to $X$. Then $Y$ is a Hausdorff space.
\end{proposition}

\begin{proposition} \label{c2cw}(Proposition 5.7, p. 59 of \cite{LW})
\label{cw-equi}Let $X$ be a CW-complex and $\mathbb{R}$ a cellular equivalence relation such that the space $X/\mathbb{R}$ is
Hausdorff. Then with the quotient structure on $X/\mathbb{R}$, the quotient space is a CW-complex and the quotient map $X \rightarrow X/\mathbb{R}$ is a regular cellular map.
\end{proposition}

\begin{proposition} \label{idsc}(Proposition 2.1.1, p. 56 of \cite{FP})
Let $Y, X$ be CW-complexes and let $f: Y \rightarrow X$ be a regular map. Then,
\begin{itemize}
\item[(i)] the image by $f$ of a closed cell of $Y$ is a closed cell of $X$;
\item[(ii)] $f(Y)$ is a subcomplex of $X$;
\item[(iii)] the induced map $Y \rightarrow f(Y)$ is an identification.
\end{itemize}
\end{proposition}

\begin{definition} \label{R} (Definition 3.7 p. 165 of \cite{LW})
Let $(X,\mathbb{S})$ and $(Y,\mathbb{K})$ be oriented CW-complexes and let $f:X\longrightarrow Y$ be a cellular map. The cells of $\mathbb{S}$ are denoted by $\sigma_{\lambda }^{n}$ and those of $\mathbb{K}$ by $\tau _{\mu }^{n}$. If $f_{\#}:C_{n}(X)\longrightarrow C_{n}(Y)$ is the map induced by $f$, we write
\begin{equation*}
f_{\#}(\sigma ^{n})=\sum_{\mu \in M_{n}}[f:\sigma ^{n}:\tau _{\mu}^{n}]\tau _{\mu }^{n},
\end{equation*}
where $M_{n}$ a set of indices in a 1-1 correspondence with the set of distinct $n$-cells of $Y$, $[f:\sigma ^{n}:\tau _{\mu }^{n}]\in \mathbb{Z}$ and $[f:\sigma^{n}:\tau _{\mu }^{n}]=0$ for all but finitely many $\mu $. The integer $[f:\sigma ^{n}:\tau _{\mu }^{n}]$ is called the degree (or sometimes the incidence number) with which $\sigma ^{n}$ is mapped on $\tau _{\mu }^{n}$ by $f$.
\end{definition}
 
 \begin{proposition} \label{R1}
 (Proposition 3.8 p. 166 of \cite{LW}) Let $X$ and $Y$ be oriented CW-complexes, let $f: X \rightarrow Y$ be a cellular map, let $\sigma$ and $\tau$ be $n$-cells of $X$ and $Y$ respectively, and let $F(\sigma: \tau)$ be the composite map
\begin{equation} \label{Rr1}
\xymatrix{F(\sigma: \tau): D^{n}/ S^{n-1} \ar[r]^(0.60){\bar{\Phi}_{\sigma}} & X^{n}/X^{n-1} \ar[r]^{\bar{f}} & Y^{n}/Y^{n-1} \ar[r]^(0.55){j_{\tau}} & \tau/\partial \tau,}
\end{equation}
where $\bar{\Phi}_{\sigma}$ and $\bar{f}$ are the induced maps of the characteristic map $\Phi_{\sigma}$ and $f$ respectively, and $j_{\tau}$ is the map which collapses all other spheres of $Y^{n}/Y^{n-1}$ to the base point. Let $F(\sigma: \tau)_{\ast}$ be the endomorphism of $H_{n}(D^{n}/S^{n-1})$ induced by $F(\sigma: \tau)$. If $u_{\sigma}\bar{\xi}_{n}$ is the generator of $H_{n}(D^{n}/S^{n-1})$ corresponding to the orientation of $\sigma$ and $\tau$ is the generator of $H(\tau/\partial \tau)$, then $F(\sigma: \tau)_{\ast}(u_{\sigma}\bar{\xi}_{n})=[f:\sigma:\tau]\tau$. 
 \end{proposition}

 \begin{corollary} \label{R2}(Corollary 3.9 p. 166 of \cite{LW})
 If $f,\sigma$ and $\tau$ are as above, and if $\tau$ is not a subset of $f(\sigma)$, then $[f:\sigma:\tau]=0$.
 \end{corollary}

\section{The graph of derivations of a monoid presentation}

A rewriting system is a pair $\cP=( \mathbf{x}, \mathbf{r})$ where $\mathbf{x}$ is a non empty set and $\mathbf{r}$ is a set of rules $r=(r_{+1},r_{-1}) \in F \times F$ where $F$ is the free monoid on $\mathbf{x}$ whose elements are ''words'' with letters from $\mathbf{x}$ and its unit element is the empty word $\lambda$. We will assume throughout that for every $r \in \mathbf{r}$, $r_{+1}\neq r_{-1}$. Related with $\mathbf{r}$ there is the so called the one single step reduction of words
$$\rightarrow_{\mathbf{r}}=\{(ur_{+1}v,ur_{-1}r)|r \in \mathbf{r} \text{ and } u, v \in F\}.$$
The reflexive and transitive closure of $\rightarrow_{\mathbf{r}}$ is denoted by $\rightarrow_{\mathbf{r}}^{\ast}$, and the reflexive, transitive and symmetric closure is denoted by $\leftrightarrow_{\mathbf{r}}^{\ast}$ and is also known as the Thue congruence generated by $\mathbf{r}$. The quotient $F/\leftrightarrow_{\mathbf{r}}^{\ast}$ forms a monoid $S$ whose elements are the congruence classes $\bar{u}$ of words $u \in F$, and the multiplication is given by $\bar{u}\cdot \bar{v}= \overline{uv}$. We say that the monoid $S$ is given by $\cP$, or that $\cP$ is a presentation for $S$.

For every rewriting system $\cP=( \mathbf{x}, \mathbf{r})$ we can define its graph of derivations $\Gamma(\cP)$ whose vertices are the elements of $F$, and the edges are all quadruples 
$$e=(w,r, \varepsilon, w') \text{ where } w,w' \in F, \varepsilon=\pm 1, r \in \mathbf{r},$$
with initial, terminal and inverse functions
$$\iota (e)= wr_{\varepsilon}w', \tau (e)= wr_{-\varepsilon}w' \text{ and } e^{-1}=(w,r, -\varepsilon, w').$$
The edge $e$ is called positive if $\varepsilon=1$. We can think of $\Gamma(\cP)$ as a one dimensional cw-complex with 0-cells all the elements of $F$ and with 1-cells all positive edges. We note here that $e^{-1}=(w,r, -1, w')$ is not a new edge attached to the complex, but is defined to mean the topological inverse of the characteristic map of $e=(w,r, 1, w')$. By definition, $\iota(e^{-1})=\tau(e)$ and $\tau(e^{-1})=\iota(e)$. 

A path $P$ of length $n$ in $\Gamma(\cP)$ is a sequence of edges $P=e^{\varepsilon_{1}}_{1} \circ \dots \circ e^{\varepsilon_{i}}_{i} \circ e^{\varepsilon_{i+1}}_{i+1}\circ \dots \circ e^{\varepsilon_{n}}_{n}$ where $\tau (e^{\varepsilon_{i}}_{i})=\iota (e^{\varepsilon_{i+1}}_{i+1})$ for $1 \leq i \leq n-1$. It is called positive if the edges are positive, that is each $\varepsilon_{i}=1$, and is called closed if $\iota (e_{1})=\tau (e_{n})$. The inverse of a path $P=e^{\varepsilon_{1}}_{1} \circ \dots \circ e^{\varepsilon_{i}}_{i} \circ e^{\varepsilon_{i+1}}_{i+1}\circ \dots \circ e^{\varepsilon_{n}}_{n}$ is the path $P^{-1}=e^{-\varepsilon_{n}}_{n} \circ \dots \circ e^{-\varepsilon_{i+1}}_{i+1} \circ e^{-\varepsilon_{i}}_{i}\circ \dots \circ e^{-\varepsilon_{1}}_{1}$.

Note that there is a 1-1 correspondence between the elements of $S$ given by $\cP$ and the connected components of $\Gamma(\cP)$ since  $u \leftrightarrow_{\mathbf{r}}^{\ast} v$ if and only if there is a path in $\Gamma(\cP)$ connecting $u$ with $v$. 

There is a natural two-sided action of $F$ on $\Gamma(\cP)$. The action on vertices is given by the multiplication of $F$, and the action of $z,z' \in F$ on edges $e=(w,r, \varepsilon, w')$ is given by
$$z.e.z'=(zw,r, \varepsilon, w'z'),$$
and sometimes is called translation. This action extends to paths in the obvious way. We note here that
$$(z.e.z')^{-1}=(zw,r, \varepsilon, w'z')^{-1}=(zw, r, -\varepsilon, w'z')=z. (w, r, -\varepsilon, w').z'= z. e^{-1}. z'.$$
We say that a positive edge $e$ is left reduced (resp. right reduced) if it cannot be written in the form $u.f$ (resp. $f. u$) for some non empty word $u \in F$ and a positive edge $f$. 

It is necessary to make a topological interpretation for all the above. We will identify a positive edge with the characteristic map of a certain 1-cell which arises from the edge itself. This will be done in the following fashion. Consider the set $F \times \mathbf{r} \times F$. For each $\kappa= (w, r, w') \in F \times \mathbf{r} \times F$ we let $D_{\kappa}^{1}$ be a copy of the unit segment $I$ and let $S_{\kappa}^{0}=\partial D_{\kappa}^{1}=\{0,1\}$. For each $\kappa= (w, r, w') \in F \times \mathbf{r} \times F$ we define the attaching map of $E=(w,r,1,w')$
$$\phi_{\kappa}: \{0,1\} \rightarrow F \text{ such that } \phi_{\kappa}(0)=wr_{+1}w', \phi_{\kappa}(1)=wr_{-1}w'.$$
Now we let $\Gamma(\cP)$ be the quotient space
$$F \sqcup_{\phi_{\kappa}} D^{1}_{\kappa}= \left(F \underset{\kappa \in F \times \mathbf{r} \times F}{\sqcup} D^{1}_{\kappa}\right)/\sim$$
where $\sim$ is the equivalence relation generated by all pairs
$$(0, \phi_{\kappa}(0)) \text{ and } (1,\phi_{\kappa}(1)),$$
whenever $\kappa \in F \times \mathbf{r} \times F$. Alternatively $\Gamma(\cP)$ is given by the following pushout diagram
\begin{equation*}
\xymatrix@R=30pt @C=60pt{ \underset{\kappa}{\sqcup}S^{0}_{\kappa} \ar[r]^{\iota} \ar[d]^{\phi} &  \underset{\kappa}{\sqcup}D^{1}_{\kappa} \ar[d]^{\nu_{\Gamma} \iota_{2}}\\
F \ar[r]^{\nu_{\Gamma}\iota_{1}} & F \sqcup_{\phi_{\kappa}} D^{1}_{\kappa}}
\end{equation*}
where 
\begin{align*}
& \iota && \text{is the inclusion of the boundary,}\\
&\phi=\underset{\kappa}{\sqcup} \phi_{\kappa} && \text{is the direct sum map,}\\
&\iota_{1}: F \rightarrow F \underset{\kappa \in F \times \mathbf{r} \times F}{\sqcup} D^{1}_{\kappa},\\
& \iota_{2}: \underset{\kappa \in F \times \mathbf{r} \times F}{\sqcup} D^{1}_{\kappa} \rightarrow F \underset{\kappa \in F \times \mathbf{r} \times F}{\sqcup} D^{1}_{\kappa} && \text{are the inclusion maps, and}\\
&\nu_{\Gamma}: F \underset{\kappa \in F \times \mathbf{r} \times F}{\sqcup} D^{1}_{\kappa} \rightarrow \left(F \underset{\kappa \in F \times \mathbf{r} \times F}{\sqcup} D^{1}_{\kappa}\right)/\sim && \text{is the quotient map}
\end{align*}
The fact the above diagram is a push out diagram will be used later. In this way each 1-cell attached to the discrete space $F$ is in 1-1 correspondence with a 4-tuple $\kappa=(w, r, w')$ (here $r=(r_{+1}, r_{-1})$ is itself a pair). It is this correspondence which enables us to interpret the previously defined positive edge $e=(w, r, 1, w')$ as the characteristic map $\Phi_{\kappa}: D^{1}_{\kappa} \rightarrow \Gamma(\cP)$ with label $\kappa=(w, r, w')$ which is given by
\begin{equation} \label{defEk}
\Phi_{\kappa}= \nu_{\Gamma} \circ \iota_{2} \circ \iota_{\kappa},
\end{equation}
where 
$$\iota_{\kappa}: D^{1}_{\kappa} \rightarrow  \underset{\kappa \in F \times \mathbf{r} \times F}{\sqcup} D^{1}_{\kappa}$$
is the inclusion map. Latter in the paper we will change this identification to $\Phi_{e}$ without causing any confusion since any positive edge $e=(w, r, 1, w')$ is uniquely determined by $\kappa=(w, r, w')$. Note that $\Phi_{\kappa}$ is at the same time a path in $\Gamma(\cP)$ such that
$$\Phi_{\kappa}(0)=\nu_{\Gamma} (\iota_{2}(\iota_{\kappa}(0)))=\nu_{\Gamma}(\iota_{2}(\phi_{\kappa}(0)))= wr_{+1}w'=\iota(e),$$
and
$$\Phi_{\kappa}(1)=\nu_{\Gamma}(\iota_{2}(\iota_{\kappa}(1)))=\nu_{\Gamma}(\iota_{2}(\phi_{\kappa}(1)))= wr_{-1}w'=\tau(e).$$
\begin{remark} \label{hom1}
For each positive edge $E=(u, r, 1, v)$, the characteristic map $\Phi_{\kappa}$ with $\kappa=(u, r,v)$, maps $I$ homeomorphically onto $E$ since $ur_{+1}v\neq ur_{-1}v$. 
\end{remark}
The inverse path $\Phi^{-1}_{\kappa}$ of $\Phi_{\kappa}$ is given by
\begin{equation} \label{defEk-1}
\Phi^{-1}_{\kappa}(t)=\Phi_{\kappa}(1-t),
\end{equation}
hence
$$\Phi^{-1}_{\kappa}(0)=\Phi_{\kappa}(1)=wr_{-1}w'=\tau(e)$$
and
$$\Phi^{-1}_{\kappa}(1)=\Phi_{\kappa}(0)=wr_{+1}w'=\iota(e),$$
which justify the definition of $e^{-1}=(w, r, -1, w')$.

Further, a path $p$ defined before as a concatenation of edges 
$$p=e_{1}^{\varepsilon_{1}} \circ \dots \circ e_{n}^{\varepsilon_{n}},$$
where each $e_{i}$ is a positive edge and $\varepsilon_{i}=\pm1$, is in this new topological setting the concatenation 
\begin{equation} \label{tildep}
\tilde{p}=\Phi^{\varepsilon_{1}}_{\kappa_{1}} \ast \dots \ast \Phi^{\varepsilon_{n}}_{\kappa_{n}}
\end{equation}
of the corresponding paths $\Phi_{\kappa_{i}}^{\varepsilon_{i}}$ where $\Phi_{\kappa_{i}}$ is the characteristic map of $e_{i}$. 

Finally we interpret topologically the bi action of $F$ on the quotient space $\Gamma(\cP)$. For every fixed $z,z' \in F$, we let 
$$\alpha_{(z,z')}: F \rightarrow \Gamma(\cP) \text{ such that } u \mapsto \nu_{\Gamma} (\iota_{1} (zuz')).$$
This mapping is continuous since it is a composite of mappings $\nu_{\Gamma}$, $\iota_{1}$ and the mapping $F\rightarrow F$ given by $u \mapsto zuz'$, all being continuous. Before we define a second mapping, we let for $z,z' \in F$ and $\kappa=(w,r,w') \in F \times \mathbf{r} \times F$, $z.\kappa.z'=(zw,r, w'z')$. Now we define
$$\beta=\underset{\kappa}{\sqcup}\beta_{\kappa}: \underset{\kappa}{\sqcup}D^{1}_{\kappa} \rightarrow \Gamma(\cP)$$
where each
$$\beta_{\kappa}: D_{\kappa}^{1} \rightarrow \Gamma(\cP) \text{ is given by } \beta_{\kappa}(t)=\nu_{\Gamma}( \iota_{2} (\iota_{z.\kappa.z'}(t))).$$
Also this mapping is continuous as a composite of continuous maps, consequently so is $\beta$. Now we see that for each $\kappa=(w,r,w') \in F \times \mathbf{r} \times F$ we have that for $0 \in S^{0}_{\kappa}$, 
\begin{align*}
(\alpha_{(z,z')} \circ \phi)(0)&= \alpha_{(z,z')}(\phi_{\kappa}(0))\\
&= \alpha_{(z,z')} (wr_{+1}w')\\
&= \nu_{\Gamma} (\iota_{1}(zwr_{+1}w'z'))\\
&= \nu_{\Gamma} (\iota_{1}(\phi_{z.\kappa.z'}(0)))\\
&= \nu_{\Gamma} (\iota_{2}(\iota_{z.\kappa.z'}(0))) \tag{definition of $\sim$}\\
&= (\beta \circ \iota)(0).
\end{align*}
and similarly, for $1 \in S^{0}_{\kappa}$, 
\begin{align*}
(\alpha_{(z,z')} \circ \phi)(1)&= \alpha_{(z,z')}(\phi_{\kappa}(1))\\
&= \alpha_{(z,z')} (wr_{-1}w')\\
&= \nu_{\Gamma} (\iota_{1}(zwr_{-1}w'z'))\\
&= \nu_{\Gamma} (\iota_{1}(\phi_{z.\kappa.z'}(1)))\\
&= \nu_{\Gamma} (\iota_{2}(\iota_{z.\kappa.z'}(1))) \tag{definition of $\sim$}\\
&= (\beta \circ \iota)(1),
\end{align*}
Therefore we have the existence of a mapping $\gamma_{(z,z')}$ which makes the diagram commute 
\begin{equation} \label{po1}
\xymatrix@R=1.2cm@C=2.0cm{ \underset{\kappa}{\sqcup}S^{0}_{\kappa} \ar[r]^{\iota} \ar[d]^{\phi} &  \underset{\kappa}{\sqcup}D^{1}_{\kappa} \ar[d]^{\nu_{\Gamma} \iota_{2}} 
\ar@/^1pc/[rdd]^{\beta}   \\
F  \ar[r]^{\nu_{\Gamma}\iota_{1}} \ar@/_1pc/[rrd]_{\alpha_{(z,z')}}   & F \sqcup_{\phi_{\kappa}} D^{1}_{\kappa} \ar@{-->}[rd]|{\gamma_{(z,z')}} \\
&          & F \sqcup_{\phi_{\kappa}} D^{1}_{\kappa} }
\end{equation}
The commutativity of the lower triangle of the diagram (\ref{po1}) implies that for every $u \in F$, 
$$\gamma_{(z,z')} (\nu_{\Gamma}(\iota_{1}(u)))=\alpha_{(z,z')}(u)= \nu_{\Gamma}(\iota_{1}(zuz')),$$
which matches with the combinatorial definition of the bi action on the set of vertices. Now we use the commutativity of the upper triangle to see how $\gamma_{(z,z')}$ behaves on the set of edges of $\Gamma(\cP)$. For each $\kappa \in F \times \mathbf{r} \times F$, and for each $t \in D_{\kappa}^{1}$,
\begin{align*}
 \gamma_{(z,z')}(\Phi_{\kappa}(t))&=\gamma \nu_{\Gamma} (\iota_{2}(\iota_{\kappa}(t)))\\
 &=\beta_{\kappa}(t)\\
 &=\nu_{\Gamma}(\iota_{2}(\iota_{z.\kappa.z'}(t)))\\
 &= \Phi_{z.\kappa.z'}(t),
\end{align*}
which means that we have the equality of paths
\begin{equation} \label{eqp}
\Phi_{z.\kappa.z'}=\gamma_{(z,z')} \circ \Phi_{\kappa}.
\end{equation}
In particular we have for a positive edge $e=(w,r,1,w')$ (with $\kappa=(w,r,w')$) that
\begin{align*}
\gamma_{(z,z')} (\iota(e))&=\gamma_{(z,z')}(\Phi_{\kappa}(0))\\
&=\Phi_{z.\kappa.z'}(0)\\
&= zwr_{+1}w'z'\\
&= z \iota(e) z',
\end{align*}
and
\begin{align*}
\gamma_{(z,z')} (\tau(e))&=\gamma_{(z,z')}(\Phi_{\kappa}(1))\\
&=\Phi_{z.\kappa.z'}(1)\\
&= zwr_{-1}w'z'\\
&= z \tau(e) z'.
\end{align*}
In a similar fashion we can show that
\begin{equation} \label{eqp-1}
\Phi^{-1}_{z.\kappa.z'}=\gamma_{(z,z')} \circ \Phi^{-1}_{\kappa},
\end{equation}
from which it follows that
$$\gamma (\iota(e^{-1}))= z \tau(e) z' ,$$
and
$$\gamma_{(z,z')} (\tau(e^{-1}))= z \iota(e) z'.$$
These justify the combinatorial definition of the bi action of $F$ on the set of positive edges and their inverses. Now we see how this definition extends to paths in $\Gamma(\cP)$. Let $p=e^{\varepsilon_{1}}_{1} \circ \dots \circ e_{n}^{\varepsilon_{n}}$ be any path in $\Gamma(\cP)$ and let for $1\leq i\leq n$, $\Phi_{\kappa_{i}}$ be the characteristic map of the positive edge $e_{i}$. Then, 
\begin{align*}
\gamma_{(z,z')} \circ (\Phi_{\kappa_{1}}^{\varepsilon_{1}} \ast \dots \ast \Phi_{\kappa_{n}}^{\varepsilon_{n}})&=(\gamma_{(z,z')} \circ \Phi_{\kappa_{1}}^{\varepsilon_{1}}) \ast \dots \ast (\gamma_{(z,z')} \circ \Phi_{\kappa_{n}}^{\varepsilon_{n}}) \tag{holds true in general}\\
&= \Phi_{z.\kappa_{1}.z'}^{\varepsilon_{1}} \ast \dots \ast \Phi_{z.\kappa_{n}.z'}^{\varepsilon_{n}}, \tag{by (\ref{eqp}) and (\ref{eqp-1})}
\end{align*}
where the image in $\Gamma(\cP)$ of $\Phi_{z.\kappa_{1}.z'}^{\varepsilon_{1}} \ast \dots \ast \Phi_{z.\kappa_{n}.z'}^{\varepsilon_{n}}$ is $z.p.z'$. With the notations established in (\ref{tildep}) we have that
\begin{equation} \label{gtp}
\gamma_{(z,z')} \circ \tilde{p}=\widetilde{z.p.z'},
\end{equation}
which justifies the previous combinatorial definition. 

To conclude this section we show that the bi action of $F$ extends naturally to any topological path on $\Gamma(\cP)$. Indeed, let $p: I \rightarrow \Gamma(\cP)$ be a continuous map (a topological path on $\Gamma(\cP)$), then for any pair $z,z' \in F$ letting $\gamma_{(z,z')}: \Gamma(\cP) \rightarrow \Gamma(\cP)$ be the mapping defined above for that pair, we have the path 
$$\gamma_{(z,z')} \circ p: I \rightarrow \Gamma(\cP) \text{ such that } (\gamma_{(z,z')} \circ p)(t)=\gamma_{(z,z')}(p(t)),$$
for every $t \in I$. If it happens that $p$ is the constant path $c_{w}$ at some vertex $w \in F$, then
$$(\gamma_{(z,z')} \circ c_{w})(t)= \gamma_{(z,z')} (c_{w}(t))=\gamma_{(z,z')} (w)=zwz',$$
consequently
\begin{equation} \label{gmc}
\gamma_{(z,z')} \circ c_{w}=c_{zwz'}.
\end{equation}

\section{Attaching 2-cells to $\Gamma(\cP)$}

The Squier complex $\mathcal{D}(\mathcal{P})$ associated with $\cP$ is a combinatorial 2-complex with 1-skeleton $\Gamma(\cP)$, to which, for each pair of positive edges $e,f$ a 2-cell $[e,f]$ is attached along the closed path
$$\delta[e,f]=e. \iota (f) \circ \tau (e). f \circ (e.\tau (f))^{-1}\circ (\iota (e). f)^{-1}.$$
Sometimes we refer the 2-cell $[e,f]$ as a square 2-cell, though the name for it is not well-chosen as the following shows.
\begin{example}
Assume that $\cP=(\mathbf{x}, \mathbf{r})$ and there is some rewriting rule $r=(r_{+1}, \lambda)$, so $r_{-1}=\lambda$. If $e=f=(\lambda,r,1,\lambda)$, then 
$$e. \iota (f)=(\lambda, r, 1, r) \text{ and } \iota(f).e= (r, r, 1, \lambda)$$
are parallel in the sense that they have the same endpoints, and 
$$\tau(e).f=(\lambda,r,1,\lambda)=e.\tau{f}$$
are the same edge. In this case the closed path $\delta[e,f]$ is not even a simple path which passes once in every of its vertices except for $rr$.
\end{example}
If we let 
 $$\mathfrak{S}=\{\delta[e,f]| e, f \text{ are positive edges}\},$$ 
 then the set of 2-cells attached to the complex is in 1-1 correspondence with the distinct loops in $\mathfrak{S}$. In other words, every square 2-cell $[e,f]$ is uniquely determined by the loop $\delta[e,f]$ which will be referred in what follows as the combinatorial boundary of $[e,f]$. It may happen though that different pairs of positive edges give rise to the same square 2-cell which means that the corresponding loops defined above are the same. The following lemma characterises those pairs.
\begin{lemma} \label{unamb}
Two square cells $[e,f]$ and $[e',f']$ coincide if and only if there is some $w \in F$ such that $e=e'.w$ and $f'=w.f$, or $e'=e.w$ and $f=w.f'$.
\end{lemma} 
\begin{proof}
Let us prove first that $[e.w, f]=[e,w.f]$. Since a square 2-cell is uniquely determined by its own combinatorial boundary, we should prove that the 2-cells have the same combinatorial boundaries. Indeed,
\begin{align*}
\partial [e.w, f]&= ((e. w). \iota f) \circ ((\tau e \cdot w). f) \circ ((e^{-1}.w).\tau f) \circ ((\iota e \cdot w). f^{-1})\\
&= (e .(w \cdot \iota f)) \circ (\tau e. (w.f)) \circ (e^{-1}. (w \cdot \tau f)) \circ (\iota e . (w.f^{-1})\\
&= (e . (\iota(w.f))) \circ (\tau e . (w.f)) \circ (e^{-1}. (\tau(w.f))) \circ (\iota e . (w.f)^{-1}) \tag{by (\ref{eqp}) and (\ref{eqp-1})}\\
&= \partial [e,w.f].
\end{align*}
Conversely, assume that $[e,f]=[e',f']$ and let $e=(u, r,1, v)$, $e'=(u', r',1, v')$, $f=(\bar{u}, \bar{r}, 1, \bar{v})$ and $f'=(\bar{u}', \bar{r}', 1, \bar{v}')$. Let $r=(a,b)$, $r'=(a',b')$, $\bar{r}=(\bar{a}, \bar{b})$ and $\bar{r}'=(\bar{a}',\bar{b}')$. Equating the combinatorial boundaries we have the following
\begin{align*}
&(u, r, 1, v \cdot \bar{u}\bar{a}\bar{v})=e . \iota f=e' . \iota f'= (u', r', 1, v' \cdot \bar{u}' \bar{a}' \bar{v'})\\
&(ubv \cdot \bar{u}, \bar{r}, 1, \bar{v})=\tau e. f=\tau e' . f'= (u' b' v' \cdot \bar{u}', \bar{r}', 1, \bar{v}')\\
& (u av \cdot \bar{u}, \bar{r}, 1, \bar{v})= \iota e. f=\iota e' . f'= (u a' v' \cdot \bar{u}', \bar{r}', 1, \bar{v}' )\\
& (u, r, 1, v \cdot \bar{u} \bar{b} \bar{v})= e . \tau f=e'.\tau f'=(u', r', 1, v' \cdot \bar{u}' \bar{b}' \bar{v'}).
\end{align*}
From the first displayed equation we have that $u=u'$ and $(a,b)=r=r'=(a',b')$. It follows from this and the second displayed equation that $v\bar{u}=v'\bar{u'}$. If $v'$ is a prefix of $v$, then there is some $w$ such that $v'w=v$. In this case we have that
$$\iota e=u a v =uav'w =u'a' v' w=\iota e' w$$
and
$$\tau e=ubv=ubv'w=u'b'v'w=\tau e' w,$$
which shows that $e=e'.w$. From the third displayed equation we have that $(\bar{a}, \bar{b})=\bar{r}=\bar{r}'=(\bar{a}',\bar{b}')$ and that $\bar{v}=\bar{v}'$. This and the fourth displayed equation imply that $v \bar{u}=v'\bar{u}'$, hence $v'w \bar{u}=v'\bar{u}'$ and $w \bar{u}=\bar{u'}$. It follows now that 
$$\iota f'=\bar{u}'\bar{a}' \bar{v'}=w \bar{u}\bar{a}\bar{v}=w . \iota f$$
and
$$\tau f'= \bar{u}'\bar{b}' \bar{v}'=w \bar{u} \bar{b}\bar{v}=w . \tau f,$$
which prove that $f'=w. f$. 

In the case when $v$ is a prefix of $v'$ it follows in a similar way as above that there is some $w$ such that $e'=e.w$ and $f=w.f'$.
\end{proof}

\begin{corollary} \label{cunamb}
For every two positive edges $e,f$, the square 2-cell $[e,f]$ is uniquely determined by the pair of edges $(e.\iota f, \iota e. f)$.
\end{corollary}
\begin{proof}
If there is some 2-cell $[e',f']$ such that $e'.\iota f' =e. \iota f$, $\iota e' . f'=\iota e. f$, then as in Lemma \ref{unamb}, letting $e=(u, r,1, v)$, $e'=(u', r',1, v')$, $f=(\bar{u}, \bar{r}, 1, \bar{v})$  and $f'=(\bar{u}', \bar{r}', 1, \bar{v}')$ with $r=(a,b)$, $r'=(a',b')$, $\bar{r}=(\bar{a}, \bar{b})$ and $\bar{r}'=(\bar{a}',\bar{b}')$, we see that
$$e'.\iota f' =(u', r', 1, v'.\bar{u}' \bar{a}'\bar{v}')= e. \iota f= (u, r, 1, v. \bar{u}\bar{a}\bar{v})$$
and
$$\iota e' . f'=(u'a'v'.\bar{u}', \bar{r}', 1, \bar{v}')= \iota e. f=(uav.\bar{u}, \bar{r}, 1, \bar{v}),$$
which imply that $u=u'$, $r=r'$, $\bar{r}=\bar{r}'$, $\bar{v}=\bar{v}'$ and $v' \bar{u}'=v \bar{u}$.  If $v'$ is a prefix of $v$, then for some $w$ we have that $v=v' w$ and $w \bar{u}=\bar{u}'$. It follows in this case that
$$e=(u, r, 1, v)=(u', r', 1, v'w)=e'. w,$$
and
$$f'=(\bar{u}', \bar{r}', 1, \bar{v}')=(w \bar{u}, \bar{r}, 1, \bar{v})=w. f$$
Now Lemma \ref{unamb} implies that $[e,f]=[e',f']$. The case when $v$ is a prefix of $v'$ is analogous to the above.
\end{proof}
\begin{notation} \label{ns}
For every two positive edges $e,f$, the square 2-cell $[e,f]$ is represented in \cite{SES} in the new form $[\iota e.\iota f, (e.\iota f, \iota e. f)]$. Corollary above implies that this notation is unambiguous.
\end{notation}

\begin{lemma} \label{bmsq}
Let $e,e'$ be left reduced positive edges and $f,f'$ be right reduced positive edges. Then for every $\alpha, \beta, \alpha', \beta' \in F$, $\alpha.[e,f]. \beta=\alpha' . [e',f']. \beta'$ if and only if $\alpha=\alpha'$, $\beta=\beta'$, $[e,f]=[e',f']$.
\end{lemma}
\begin{proof}
Let $e=(\lambda,r,1, v)$, $e'=(\lambda,r',1, v')$, $f=(\bar{u}, \bar{r}, 1, \lambda)$ and $f'=(\bar{u}', \bar{r}', 1, \lambda)$. From the assumption we have the equality of the 2-cells $[\alpha.e, f.\beta]=[\alpha'.e', f'.\beta']$, and then from Lemma \ref{unamb}, there is some $w \in F$ such that either $\alpha. e=\alpha'. e'. w$ and $w.f.\beta=f'. \beta'$, or $\alpha. e. w=\alpha'. e'$ and $f.\beta=w.f'. \beta'$. If any of the above holds true, then the lemma follows. We give here the proof when the first equality holds true since the other case is dual to the first. From the first given equality we have
$$(\alpha, r, 1, v)=\alpha. e=\alpha'. e'. w=(\alpha', r', 1, v'w),$$
which implies that $\alpha=\alpha'$, $r=r'$ and $v=v'w$.  Therefore
$$e=(\lambda,r,1, v)=(\lambda,r', 1, v'w)=e'.w.$$
From the second given equality we have
$$(w. \bar{u}, \bar{r}, 1, \beta)=w.f.\beta=f'. \beta'=(\bar{u}', \bar{r}',1, \beta'),$$
from which we get that $\beta=\beta'$, $\bar{r}=\bar{r}'$ and $w \bar{u}=\bar{u}'$. Therefore
$$w.f=(w. \bar{u}, \bar{r}, 1, \lambda)=(\bar{u}', \bar{r}',1, \lambda)=f'.$$
Now Lemma \ref{unamb} implies that $[e,f]=[e',f']$. 
\end{proof}
\begin{corollary}
Every 2-cell $[g,h]$ can be written uniquely as $[g,h]=u. [e, f].v$ where $e$ is a left reduced positive edge and $f$ a right reduced positive edge.
\end{corollary}
\begin{proof}
Indeed, there is some $u \in F$ such that $g=u.e$ with $e$ left reduced, and similarly there is some $v\in F$ such that $h=f.v$ and $f$ is right reduced, hence 
$$[g,h]=[u.e, f.v]=u.[e,f].v.$$
Lemma \ref{bmsq} implies that the above representation of $[g,h]$ is unique.
\end{proof}
The following will be useful in the last section.
\begin{lemma} \label{uf}
Let $e_{1},e_{2},e_{1}',e_{2}'$ be left and right reduced edges and $u_{1},v_{1},u_{2},v_{2}, u'_{1},v'_{1}, u'_{2},v'_{2}\in F$ such that $[u_{1}.e_{1}.v_{1}, u_{2}.e_{2}.v_{2}]=[u'_{1}.e'_{1}.v'_{1}, u'_{2}.e'_{2}.v'_{2}]$, then $u_{1}=u'_{1}$, $v_{2}=v'_{2}$, $e_{1}=e'_{1}$, $e_{2}=e'_{2}$ and $v_{1}u_{2}=v'_{1}u'_{2}$.
\end{lemma}
\begin{proof}
From Corollary \ref{cunamb} we have that 
$$u_{1}.e_{1}.v_{1} \iota(u_{2}.e_{2}.v_{2})=u'_{1}.e'_{1}.v'_{1} \iota(u'_{2}.e'_{2}.v'_{2})$$
and
$$\iota(u_{1}.e_{1}.v_{1})u_{2}.e_{2}.v_{2}= \iota(u'_{1}.e'_{1}.v'_{1})u'_{2}.e'_{2}.v'_{2}.$$
The assumption on edges implies that 
$$u_{1}=u'_{1}, v_{2}=v'_{2}, e_{1}=e'_{1}, e_{2}=e'_{2},$$
and 
$$v_{1}u_{2} \iota(e_{2})v_{2}=v'_{1}u'_{2}\iota(e'_{2})v'_{2}.$$
After cancellations we obtain the remaining equality $v_{1}u_{2}=v'_{1}u'_{2}$.
\end{proof}

The attachment of all 2-cells in one gulp can be done topologically as follows. For each $s \in \mathfrak{S}$ we let $D_{s}^{2}$ identified with $I \times I$ be a copy of the unit disk $D^{2}$ labeled by $s$, and let $\partial D_{s}^{2}$ be its boundary denoted as the unit circle $S_{s}^{1}$ labeled by $s$. For each $s=\partial [e,f] \in \mathfrak{S}$ we define the attaching map
$$\phi_{s}: S_{s}^{1} \rightarrow \Gamma(\cP)$$
of $[e,f]$ as follows. Consider edges
$$
e.\iota f, \tau e. f, \iota e. f, e. \tau f.
$$
All these edges do not depend on the choice of the pair $(e,f)$ with $\partial [e,f]=s$. This is an easy implication of Lemma \ref{unamb}. The four segments of $\partial(I \times I)$ are labeled as depicted in the figure below

\begin{equation}\label{amsq}
\begin{tikzpicture}[x=4cm, y=4cm, thick]  
\def\a{0.3333}  
\def\b{0.6667}   

\draw (0,0) -- (1,0) node[pos=0.5, below, font=\small] {$e.\iota f$};   
\node[above=5pt, font=\small] at (0.5,0) {$L_1$};

\draw (0,0) -- (0,1) node[pos=0.5, left, font=\small] {$(\iota e. f)^{-1}$};   
\node[right=5pt, font=\small] at (0,0.5) {$L_3$};

\draw (0,1) -- (1,1) node[pos=0.5, above, font=\small] {$(e. \tau f)^{-1}$};  

\node[below=5pt, font=\small] at (0.5,1) {$L_4$};
\draw (\b,1) -- (1,1) node[pos=0.5, above, font=\small] {};   

\draw (1,1) -- (1,0) node[pos=0.5, right, font=\small] {$\tau e. f$};  

\node[left=5pt, font=\small] at (1,0.5) {$L_2$};

\filldraw (0,0) circle (1.5pt) node[below left] {$(0,0)$};  
\filldraw (1,0) circle (1.5pt) node[below right] {$(1,0)$};  
\filldraw (1,1) circle (1.5pt) node[above right] {$(1,1)$};  
\filldraw (0,1) circle (1.5pt) node[above left] {$(0,1)$};  
\end{tikzpicture}
\end{equation}

We define the following four homeomorphisms. 
$$h_{1}: L_{1} \rightarrow I \text{ by } ( t, 0) \mapsto t,$$ 
$$h_{2}: L_{2} \rightarrow I \text{ by } (1,t) \mapsto t,$$
$$h_{3}: L_{3} \rightarrow I \text{ by } (0, t) \mapsto 1-t,$$
$$h_{4}: L_{4} \rightarrow I \text{ by } (t,1) \mapsto 1-t.$$
Consider now the mappings
$$\Phi_{e.\iota f}\circ h_{1}, \Phi_{\tau e. f}\circ h_{2}, \Phi^{-1}_{e. \tau f}\circ h_{4}, \Phi^{-1}_{\iota e. f}\circ h_{3}.$$
Observe that
\begin{align*}
(\Phi_{e.\iota f}\circ h_{1})((1,0))&=\Phi_{e.\iota f}(1)\\
&= \tau e. \iota f\\
&=\Phi_{\tau e. f}(0)\\
&=(\Phi_{\tau e. f}\circ h_{2})((1,0)),
\end{align*}
therefore from the gluing lemma (the finite closed cover version) we can glue together $\Phi_{e.\iota f}\circ h_{1}$ with $\Phi_{\tau e. f}\circ h_{2}$ to form a continuous mapping
$$\phi_{s,1}: L_{1} \cup L_{2} \rightarrow \mathcal{D}(\cP)$$
such that
$$\phi_{s,1}|_{L_{1}}=\Phi_{e.\iota f}\circ h_{1} \text{ and } \phi_{s,1}|_{L_{2}}=\Phi_{\tau e. f}\circ h_{2},$$
therefore 
$$\text{Im}\phi_{s,1}=e.\iota (f) \cup \tau(e).f$$
Also we observe that
\begin{align*}
(\Phi^{-1}_{e.\tau f}\circ h_{4})((0,1))&=\Phi^{-1}_{e.\tau f}(1)\\
&= \Phi_{e.\tau f}(0)\\
&= \iota e. \tau f\\
&=\Phi_{\iota e. f}(1)\\
&=(\Phi^{-1}_{\iota e. f}\circ h_{3})((0,1)),
\end{align*}
therefore we can glue as before $\Phi^{-1}_{e.\tau f}\circ h_{4}$ with $\Phi^{-1}_{\iota e. f}\circ h_{3}$ to form a continuous map
$$\phi_{s,2}: L_{3} \cup L_{4} \rightarrow \mathcal{D}(\cP)$$
such that
$$\phi_{s,2}|_{L_{4}}=\Phi^{-1}_{e.\tau f}\circ h_{4} \text{ and } \phi_{s,2}|_{L_{3}}=\Phi^{-1}_{\iota e. f}\circ h_{3},$$
therefore 
$$\text{Im}\phi_{s,2}=\iota (e).f \cup e.\tau(f).$$
We see that
\begin{align*}
\phi_{s,1}(0,0)&= (\Phi_{e.\iota f}\circ h_{1} )(0,0)\\
&= \Phi_{e. \iota f}(0)\\
&= \iota e. \iota f\\
&= \Phi_{\iota e. f}(0)\\
&=\Phi^{-1}_{\iota e. f}(1)\\
&= (\Phi^{-1}_{\iota e. f}\circ h_{3} )(0,0)\\
&= \phi_{s,2}(0,0),
\end{align*}
and that 
\begin{align*}
\phi_{s,1}(1,1)&= (\Phi_{\tau e. f}\circ h_{2} )(1,1)\\
&= \Phi_{\tau e. f}(1)\\
&= \tau e .\tau f\\
&= \Phi_{e.\tau f}(1)\\
&= \Phi^{-1}_{e.\tau f}(0)\\
&= (\Phi^{-1}_{e.\tau f}\circ h_{4} )(1,1)\\
&= \phi_{s,2}(1,1).
\end{align*}
From the gluing lemma again we can now form a continuous mapping $\phi_{s}: S^{1}_{s} \rightarrow \Gamma(\cP)$ such that 
\begin{align}
\phi_{s}|_{L_{1}}&=\Phi_{e.\iota f} \circ h_{1}, \phi_{s}|_{L_{2}}=\Phi_{\tau e. f} \circ h_{2}, \notag\\
\phi_{s}|_{L_{4}}&=\Phi^{-1}_{e.\tau f} \circ h_{4}, \phi_{s}|_{L_{3}}= \Phi^{-1}_{\iota e. f} \circ h_{3}, \label{rest}
\end{align}
with
$$\text{Im}(\phi_{s})=e.\iota(f) \cup \tau(e).f \cup \iota(e).f \cup e.\tau(f),$$
This is the attaching map of $[e,f]$ and $\text{Im}(\phi_{s})$ is the topological boundary of $[e,f]$. We remark that 
\begin{equation} \label{S1}
\phi_{s}(t,0)= \Phi_{e.\iota f}(t) \text{ and } \phi_{s}(1,t)= \Phi_{\tau e. f}(t)
\end{equation}
which explains the presence of $e.\iota f$ and $\tau e. f$ in the combinatorial boundary of $[e,f]$. Also we have that
\begin{equation} \label{S2}
\phi_{s}(0,t)=  \Phi^{-1}_{\iota e. f}(1-t) \text{ and } \phi_{s}(t,1)=\Phi^{-1}_{e.\tau f}(1-t)
\end{equation}
which explain the presence of $(\iota e. f)^{-1}$ and $(e. \tau f)^{-1}$ in the combinatorial boundary of $[e,f]$.

Now we let $\mathcal{D}(\cP)$ be the quotient space
$$\Gamma(\cP) \sqcup_{\phi_{s}} D^{2}_{s}= \left(\Gamma(\cP) \underset{s \in \mathfrak{S}}{\sqcup} D^{2}_{s}\right)/\sim$$
where $\sim$ is the equivalence relation generated by all pairs
\begin{equation} \label{rel2}
(u, \phi_{s}(u)) \text{ with } u \in S^{1}_{s} \text{ and } s \in \mathfrak{S}.
\end{equation}
Alternatively $\mathcal{D}(\mathcal{P})$ is given by the following push out diagram
\begin{equation*}
\xymatrix@R=30pt @C=60pt{ \underset{s}{\sqcup}S^{1}_{s} \ar[r]^{j} \ar[d]^{\phi} &  \underset{s}{\sqcup}D^{2}_{s} \ar[d]^{\nu_{\mathcal{D}} j_{2}}\\
\Gamma(\cP) \ar[r]^{\nu_{\mathcal{D}}j_{1}} & \Gamma(\cP) \sqcup_{\phi_{s}} D^{2}_{s}}
\end{equation*}
where 
\begin{align*}
& j && \text{is the inclusion of the boundary,}\\
&\phi=\underset{s}{\sqcup} \phi_{s} && \text{is the direct sum map,}\\
&j_{1}: \Gamma(\cP) \rightarrow \Gamma(\cP) \underset{s \in \mathfrak{S}}{\sqcup} D^{2}_{s} ,\\
& j_{2}: \underset{s \in \mathfrak{S}}{\sqcup} D^{2}_{s} \rightarrow \Gamma(\cP) \underset{s \in \mathfrak{S}}{\sqcup} D^{2}_{s} && \text{are the inclusion maps, and}\\
&\nu_{\mathcal{D}}: \Gamma(\cP) \underset{s \in \mathfrak{S}}{\sqcup} D^{2}_{s} \rightarrow \left(\Gamma(\cP) \underset{s \in \mathfrak{S}}{\sqcup} D^{2}_{s}\right)/\sim && \text{is the quotient map.}
\end{align*}
For every $s \in \mathfrak{S}$ we let
$$j_{s}: D^{2}_{s} \rightarrow \underset{s }{\sqcup} D^{2}_{s}$$
to be the inclusion map. Since the attaching map of a 2-cell $[e,f]$ with $\partial [e,f]=s$ is the mapping $\phi_{s}$, the characteristic map of $[e,f]$ will be the composite
$$\Phi_{s}=\nu_{\mathcal{D}} \circ  j_{2} \circ j_{s}.$$
Now we will describe in topological terms the bi action of $F$ on the 2-skeleton of $\mathcal{D}(\cP)$. In combinatorial terms this bi action is defined by setting for every $z,z' \in F$ and every square 2-cell $[e,f]$,
$$z.[e,f].z'=[z.e, f.z'].$$
Topologically this bi action can be described as follows. For the given pair of words $z,z'$ consider the continuous mapping $\gamma_{(z,z')}: \Gamma(\cP) \rightarrow \Gamma(\cP)$ of diagram (\ref{po1}) and then take the composite $\nu_{\mathcal{D}} \circ j_{1} \circ \gamma_{(z,z')}: \Gamma(\cP) \rightarrow \mathcal{D}(\cP)$ which is still continuous. Further, for every $s \in \mathfrak{S}$, we define
$$\delta_{s}: D^{2}_{s} \rightarrow \Gamma(\cP) \sqcup_{\phi_{s}} D^{2}_{s},$$
by
$$\delta_{s}(u)=\nu_{\mathcal{D}}(j_{2}(j_{z.s.z'}(u))),$$
where $u$ of the right hand side is considered a point in $D^{2}_{z.s.z'}$. Again this mapping is continuous and so is the extension
$$\delta=\underset{s}{\sqcup}\delta_{s}: \underset{s}{\sqcup} D^{2}_{s} \rightarrow \Gamma(\cP) \sqcup_{\phi_{s}} D^{2}_{s}$$
of all $\delta_{s}$. Now we check that $\delta \circ j=(\nu_{\mathcal{D}} \circ j_{1} \circ \gamma_{(z,z')}) \circ \phi$.
Let $u \in S^{1}_{s}$ for some particular $s \in \mathfrak{S}$, and let $[e,f]$ be the 2-cell such that $s =\partial [e,f]$. Assume that $u \in L_{1}$, then
\begin{align*}
((\nu_{\mathcal{D}} \circ j_{1} \circ \gamma_{(z,z')}) \circ \phi)(u)&= (\nu_{\mathcal{D}} \circ j_{1} \circ \gamma_{(z,z')})(\phi_{s}(u))\\
&= (\nu_{\mathcal{D}} \circ j_{1} \circ \gamma_{(z,z')})((\Phi_{e.\iota f} \circ h_{1})(u)) \tag{by (\ref{rest})}\\
&= \nu_{\mathcal{D}} (j_{1}((\gamma_{(z,z')} \circ \Phi_{e.\iota f})(h_{1}(u))))\\
&= \nu_{\mathcal{D}} (j_{1}(\Phi_{z.e.\iota f.z'}(h_{1}(u)))) \tag{by (\ref{eqp})}\\
&= \nu_{\mathcal{D}} (j_{1}((\Phi_{z.e.\iota f.z'} \circ h_{1})(u)))\\
&=\nu_{\mathcal{D}} (j_{1}( \phi_{z.s.z'}(u))) \tag{by (\ref{rest})}\\
&=\nu_{\mathcal{D}} (j_{1}( \phi_{z.s.z'}(j_{z.s.z'}(u)))) \tag{replace $u$ by $j_{z.s.z'}(u)$}\\
&= \nu_{\mathcal{D}} (j_{2}(j_{z.s.z'}(u))) \tag{by (\ref{rel2})}\\
&=(\delta \circ j)(u).
\end{align*}
There is a similar proof to the above when $u \in L_{2}$ and is omitted here. Now we check the case when $u \in L_{4}$ and leave the proof to the reader for its analogue case when $u \in L_{3}$. 
\begin{align*}
((\nu_{\mathcal{D}} \circ j_{1} \circ \gamma_{(z,z')}) \circ \phi)(u)&= (\nu_{\mathcal{D}} \circ j_{1} \circ \gamma_{(z,z')})(\phi_{s}(u))\\
&= (\nu_{\mathcal{D}} \circ j_{1} \circ \gamma_{(z,z')})((\Phi^{-1}_{e.\tau f} \circ h_{4})(u)) \tag{by (\ref{rest})}\\
&= \nu_{\mathcal{D}} (j_{1}((\gamma_{(z,z')} \circ \Phi^{-1}_{e.\tau f})(h_{4}(u))))\\
&= \nu_{\mathcal{D}} (j_{1}(\Phi^{-1}_{z.e.\tau f.z'}(h_{4}(u)))) \tag{by (\ref{eqp-1})}\\
&= \nu_{\mathcal{D}} (j_{1}((\Phi^{-1}_{z.e.\iota f.z'} \circ h_{4})(u)))\\
&=\nu_{\mathcal{D}} (j_{1}( \phi_{z.s.z'}(u))) \tag{by (\ref{rest})}\\
&=\nu_{\mathcal{D}} (j_{1}( \phi_{z.s.z'}(j_{z.s.z'}(u)))) \tag{replace $u$ by $j_{z.s.z'}(u)$}\\
&= \nu_{\mathcal{D}} (j_{2}(j_{z.s.z'}(u))) \tag{by (\ref{rel2})}\\
&=(\delta \circ j)(u).
\end{align*}
From the universal property of the push outs we have the existence of a mapping $\rho_{(z,z')}: \Gamma(\cP) \sqcup_{\phi_{s}} D^{2}_{s} \rightarrow \Gamma(\cP) \sqcup_{\phi_{s}} D^{2}_{s}$ such that the following diagram commutes
\begin{equation} \label{po2}
\xymatrix@R=1.5cm@C=2.5cm{ \underset{s}{\sqcup}S^{1}_{s}  \ar[r]^{j} \ar[d]^{\phi} &  \underset{s}{\sqcup}D^{2}_{s} \ar[d]^{\nu_{\mathcal{D}} j_{2}}
\ar@/^1pc/[rdd]^{\delta}   \\
\Gamma(\cP) \ar[r]^{\nu_{\mathcal{D}}j_{1}}  \ar[r]^{\nu_{\mathcal{D}}j_{1}} \ar@/_1pc/[rrd]_{\nu_{\mathcal{D}}j_{1}\gamma_{(z,z')}}   & \Gamma(\cP) \sqcup_{\phi_{s}} D^{2}_{s} \ar@{-->}[rd]|{\rho_{(z,z')}} \\
&          & \Gamma(\cP) \sqcup_{\phi_{s}} D^{2}_{s} }
\end{equation}
Now we see how the characteristic maps of $[e,f]$ and $z.[e,f].z'=[z.e, f.z']$ are related to each other. Assume that $\partial [e, f]=s$, hence $\partial [z.e,f.z']=z.s.z'$. It follows that for every $u \in D^{2}_{s}$,
\begin{align*}
(\rho_{(z,z')} \circ \Phi_{s})(u)&= (\rho_{(z,z')} \circ \nu_{\mathcal{D}} \circ j_{2}  \circ j_{s})(u)\\
&= \delta_{s}(u) \tag{from the commutativity of (\ref{po2})}\\
&= \nu_{\mathcal{D}}(j_{2}(j_{z.s.z'}(u)))\\
&= \Phi_{z.s.z'}(u),
\end{align*}
hence we have that
\begin{equation} \label{relch}
\rho_{(z,z')} \circ \Phi_{s}=\Phi_{z.s.z'}.
\end{equation}
Since for every 2-cell $[e,f]$ with $\delta[e,f]=s$ and for every $z,z' \in F$ we have that $\delta[z.e,f.z']=z. \delta[e,f].z'$, we can rewrite (\ref{relch}) in the form
$$\rho_{(z,z')} \circ \Phi_{\delta[e,f]}=\Phi_{\delta [z.e,f.z']},$$
which explains why the bi action on 2-cells gives 2-cells by the rule $z.[e,f].z'=[z.e, f.z']$. The following lemma will be used in the next section.
\begin{lemma} \label{ronu}
For every positive edge $e$ and every $z,z' \in F$ and every $\varepsilon=\pm 1$, $\rho_{(z,z')} \circ \nu_{\mathcal{D}}\circ j_{1} \circ \Phi^{\varepsilon}_{e}=\nu_{\mathcal{D}}\circ j_{1} \circ \Phi^{\varepsilon}_{z.e.z'}$.
\end{lemma}
\begin{proof}
For every $t \in I$ we have that 
\begin{align*}
(\rho_{(z,z')} \circ \nu_{\mathcal{D}}\circ j_{1} \circ \Phi^{\varepsilon}_{e})(t)&=(\rho_{(z,z')} \circ \nu_{\mathcal{D}}\circ j_{1}) (\Phi^{\varepsilon}_{e} (t))\\
&=(\nu_{\mathcal{D}}\circ j_{1} \circ \gamma_{(z,z')})(\Phi^{\varepsilon}_{e} (t)) \tag{by (\ref{po2})}\\
&=(\nu_{\mathcal{D}}\circ j_{1})( ( \gamma_{(z,z')} \circ \Phi^{\varepsilon}_{e} )(t))\\
&=(\nu_{\mathcal{D}}\circ j_{1})(\Phi^{\varepsilon}_{z.e.z'} (t) ) \tag{by (\ref{eqp}) and (\ref{eqp-1})}\\
&=(\nu_{\mathcal{D}}\circ j_{1} \circ \Phi^{\varepsilon}_{z.e.z'})(t).
\end{align*}
\end{proof}
We conclude this section by remarking that the commutativity of the lower triangle of (\ref{po2}) shows that the previously defined bi action of $F$ on $\Gamma(\cP)$ is preserved after attaching square two cells.

\section{Adding extra 2-cells to $\mathcal{D}(\cP)$-the general case} \label{gc}

We assume in this section that $\mathbf{p}$ is a set of closed paths in $\Gamma(\cP)$ none of which is the combinatorial boundary of a square 2-cell. We can attach extra 2-cells to $\mathcal{D}(\cP)$ in the following way. For every $w,w' \in F$ and every $p \in \mathbf{p}$ a 2-cell $[w,p,w']$ is attached along the closed path
$$\delta [w,p,w']=w.p.w',$$
which we call here the combinatorial boundary of $[w,p,w']$. We emphasise that the set of the new 2-cells is in a 1-1 correspondence with triples $[w,p,w']$ from $F\times \mathbf{p} \times F$. We have in particular that when $w=w'=\lambda$, $[\lambda, p,\lambda]$ is the 2-cell attached along the closed path $p$. In order to shorten the notations, the extended 2-complex will be written in what follows as $\mathcal{D}(\mathcal{P}) \sqcup F.\mathbf{p}.F$ instead of $\mathcal{D}(\mathcal{P}) \sqcup F \times \mathbf{p} \times F$. The bi action of $F$ on the cells of $\mathcal{D}(\cP)$ now extends to the 2-cells of $\mathcal{D}(\mathcal{P}) \sqcup F.\mathbf{p}.F$ by setting for each new 2-cell $[w,p,w']$ and each $v,v'\in F$
$$v.[w,p,w'].v'=[vw,p,w'v'].$$
In particular we have that 
\begin{equation} \label{unqexp}
[w,p,w'] =w.[\lambda, p, \lambda].w'
\end{equation}
and that this expression is unique. Sometimes, in order to abreviate the notations, we write the 2-cell $[\lambda, p, \lambda]$ simply as $p$, and $[w,p,w']$ as $w.p.w'$.

We will define the extension $\mathcal{D}(\mathcal{P}) \sqcup F.\mathbf{p}.F$ of $\mathcal{D}(\cP)$ topologically in terms of the attaching maps of 2-cells in $F.\mathbf{p}.F$. Also we will reveal a relationship between the characteristic map of a 2-cell and the characteristic map of a translate of that cell. Let 
$$\delta [w,p,w']=e_{1}^{\varepsilon_{1}} \circ \dots \circ e_{n}^{\varepsilon_{n}}$$
be a closed path at $w=\iota e_{1}^{\varepsilon_{1}}$. We distinguish between three possible cases for the length $n$ of $\delta [w,p,w']$.\\
(i) $n=2$, hence 
$$\delta [w,p,w']=e_{1}^{\varepsilon_{1}} \circ e_{2}^{\varepsilon_{2}}.$$
The boundary $S^{1}_{[w,p,w']}$ of the attaching disk $D^{2}_{[w,p,w']}$ is labeled as below.
\begin{equation} \label{amp2}
\begin{tikzpicture}[x=4cm, y=4cm, thick]  
\def\a{0.3333}  
\def\b{0.6667}   

\draw (0,0) -- (1,0) node[pos=0.5, below, font=\small] {$e^{\varepsilon_{1}}_{1}$};   
\node[above=5pt, font=\small] at (0.5,0) {$L_1$};

\draw (0,0) -- (0,1) node[pos=0.5, left, font=\small] {$e^{\varepsilon_{2}}_{2}$};   
\node[right=5pt, font=\small] at (0,0.5) {$L_3$};

\draw (0,1) -- (1,1) node[pos=0.5, above, font=\small] {$\tau e^{\varepsilon_{1}}_{1}$};  

\node[below=5pt, font=\small] at (0.5,1) {$L_4$};
\draw (\b,1) -- (1,1) node[pos=0.5, above, font=\small] {};   

\draw (1,1) -- (1,0) node[pos=0.5, right, font=\small] {$\tau e^{\varepsilon_{1}}_{1}$};  

\node[left=5pt, font=\small] at (1,0.5) {$L_2$};

\filldraw (0,0) circle (1.5pt) node[below left] {$(0,0)$};  
\filldraw (1,0) circle (1.5pt) node[below right] {$(1,0)$};  
\filldraw (1,1) circle (1.5pt) node[above right] {$(1,1)$};  
\filldraw (0,1) circle (1.5pt) node[above left] {$(0,1)$};  
\end{tikzpicture}
\end{equation}
Let $h_{1},h_{2},h_{3},h_{4}$ be the homeomorphisms defined earlier:
$$h_{1}: L_{1} \rightarrow I \text{ by } ( t, 0) \mapsto t,$$ 
$$h_{2}: L_{2} \rightarrow I \text{ by } (1,t) \mapsto t,$$
$$h_{3}: L_{3} \rightarrow I \text{ by } (0, t) \mapsto 1-t,$$
$$h_{4}: L_{4} \rightarrow I \text{ by } (t,1) \mapsto 1-t.$$ 
Consider now two maps
$$\nu_{\mathcal{D}}\circ j_{1} \circ \Phi^{\varepsilon_{1}}_{e_{1}}\circ h_{1} \text{ and } c_{(\nu_{\mathcal{D}} \circ j_{1})(\tau e^{\varepsilon_{1}}_{1})}\circ h_{2},$$
which are continuous and 
\begin{align*}
(\nu_{\mathcal{D}}\circ j_{1} \circ \Phi^{\varepsilon_{1}}_{e_{1}}\circ h_{1} )(1,0)&= (\nu_{\mathcal{D}}\circ j_{1} \circ \Phi^{\varepsilon_{1}}_{e_{1}})(1)\\
&= \tau e^{\varepsilon_{1}}_{1}\\
&= c_{(\nu_{\mathcal{D}} \circ j_{1})(\tau e^{\varepsilon_{1}}_{1})}(0)\\
&= (c_{(\nu_{\mathcal{D}} \circ j_{1})(\tau e^{\varepsilon_{1}}_{1})}\circ h_{2})(1,0).
\end{align*}
We can therefore glue these two maps together to form a continuous map
$$\phi^{(1)}_{[w,p,w']}: L_{1} \cup L_{2} \rightarrow \mathcal{D}(\cP)$$
such that
$$\phi^{(1)}_{[w,p,w']}|_{L_{1}}= \nu_{\mathcal{D}}\circ j_{1} \circ \Phi^{\varepsilon_{1}}_{e_{1}}\circ h_{1} \text{ and } \phi^{(1)}_{[w,p,w']}|_{L_{2}}=c_{(\nu_{\mathcal{D}} \circ j_{1})(\tau e^{\varepsilon_{1}}_{1})}\circ h_{2}.$$
We see that for every $(t,0) \in L_{1}$,
\begin{equation} \label{A1}
\phi^{(1)}_{[w,p,w']}(t,0)=\nu_{\mathcal{D}}( j_{1}(\Phi^{\varepsilon_{1}}_{e_{1}}(t)))=\Phi^{\varepsilon_{1}}_{e_{1}}(t),
\end{equation}
where the last equality follows from the fact that $\nu_{\mathcal{D}}\circ j_{1}$ is a topological embedding. Also we see that for every $(1,t) \in L_{2}$
\begin{equation} \label{A2}
\phi^{(1)}_{[w,p,w']}(1,t)=\tau e^{\varepsilon_{1}}_{1}=\Phi^{\varepsilon_{1}}_{e_{1}}(1).
\end{equation}
Observe now that
\begin{equation} \label{f1-r1}
\phi^{(1)}_{[w,p,w']}(0,0)= \iota e^{\varepsilon_{1}}_{1}=\tau e^{\varepsilon_{2}}_{2}
\end{equation}
and 
\begin{equation} \label{f1-r2}
\phi^{(1)}_{[w,p,w']}(1,1)=\tau e^{\varepsilon_{1}}_{1}=\iota e^{\varepsilon_{2}}_{2}.
\end{equation}
Further we consider mappings
$$\nu_{\mathcal{D}}\circ j_{1} \circ \Phi^{\varepsilon_{2}}_{e_{2}}\circ h_{3} \text{ and } c_{(\nu_{\mathcal{D}} \circ j_{1})(\tau e^{\varepsilon_{1}}_{1})}\circ h_{4},$$
which are continuous and 
\begin{align*}
(\nu_{\mathcal{D}}\circ j_{1} \circ \Phi^{\varepsilon_{2}}_{e_{2}}\circ h_{3} )(0,1)&= (\nu_{\mathcal{D}}\circ j_{1} \circ \Phi^{\varepsilon_{2}}_{e_{2}})(0)\\
&=  \iota e^{\varepsilon_{2}}_{2}\\
&=\tau e^{\varepsilon_{1}}_{1}\\
&= c_{(\nu_{\mathcal{D}} \circ j_{1})(\tau e^{\varepsilon_{1}}_{1})}(1)\\
&= (c_{(\nu_{\mathcal{D}} \circ j_{1})(\tau e_{1}^{\varepsilon_{1}})}\circ h_{4})(0,1).
\end{align*}
We can therefore glue these two maps together to form a continuous map
$$\phi^{(2)}_{[w,p,w']}: L_{3} \cup L_{4} \rightarrow \mathcal{D}(\cP)$$
such that
$$\phi^{(2)}_{[w,p,w']}|_{L_{3}}= \nu_{\mathcal{D}}\circ j_{1} \circ \Phi^{\varepsilon_{2}}_{e_{2}}\circ h_{3} \text{ and } \phi^{(2)}_{[w,p,w']}|_{L_{4}}=c_{(\nu_{\mathcal{D}} \circ j_{1})(\tau e^{\varepsilon_{1}}_{1})}\circ h_{4}.$$
We see that for every $(0,t) \in L_{3}$,
\begin{equation} \label{A3}
\phi^{(2)}_{[w,p,w']}(0,t)=\nu_{\mathcal{D}}(j_{1}(\Phi^{\varepsilon_{2}}_{e_{2}}(1-t)))=\Phi^{\varepsilon_{2}}_{e_{2}}(1-t),
\end{equation}
and for every $(t,1) \in L_{4}$
\begin{equation} \label{A4}
\phi^{(2)}_{[w,p,w']}(t,1)=\tau e^{\varepsilon_{1}}_{1}=\Phi^{\varepsilon_{2}}_{e_{2}}(0).
\end{equation}
Observe now that
\begin{equation} \label{f1-l1}
\phi^{(2)}_{[w,p,w']}(0,0)=\tau e^{\varepsilon_{2}}_{2}= \iota e^{\varepsilon_{1}}_{1}
\end{equation}
and 
\begin{equation} \label{f1-l2}
\phi^{(2)}_{[w,p,w']}(1,1)=\tau e^{\varepsilon_{1}}_{1}=\iota e^{\varepsilon_{2}}_{2}.
\end{equation}
It follows from (\ref{f1-r1}), (\ref{f1-l1}), (\ref{f1-r2}) and (\ref{f1-l2}) that we can glue $\phi^{(1)}_{[w,p,w']}$ with $\phi^{(2)}_{[w,p,w']}$ to obtain a continuous map (the attaching map of $[w,p, w']$)
$$\phi_{[w,p,w']}: S^{1}_{[w,p,w']} \rightarrow \mathcal{D}(\cP),$$
which maps $S^{1}_{[w,p,w']}$ onto 
$$\partial [w,p,w']= e_{1} \cup e_{2}.$$ 
which is the topological boundary of $[w,p,w']$. We remark here that (\ref{A1}) and (\ref{A3}) justify the presence of $e_{1}^{\varepsilon_{1}}$ and $e_{2}^{\varepsilon_{2}}$ in $\delta [w,p,w']$.\\
(ii) If $n=3$, then 
$$\delta [w,p,w']=e_{1}^{\varepsilon_{1}} \circ e_{2}^{\varepsilon_{2}} \circ e_{3}^{\varepsilon_{3}}.$$
The boundary $S^{1}_{[w,p,w']}$ of the attaching disk $D^{2}_{[w,p,w']}$ is labeled as below.
\begin{equation} \label{amp3}
\begin{tikzpicture}[x=4cm, y=4cm, thick]  
\def\a{0.3333}  
\def\b{0.6667}   

\draw (0,0) -- (1,0) node[pos=0.5, below, font=\small] {$e^{\varepsilon_{1}}_{1}$};   
\node[above=5pt, font=\small] at (0.5,0) {$L_1$};

\draw (0,0) -- (0,1) node[pos=0.5, left, font=\small] {$e^{\varepsilon_{3}}_{3}$};   
\node[right=5pt, font=\small] at (0,0.5) {$L_3$};

\draw (0,1) -- (1,1) node[pos=0.5, above, font=\small] {$\tau e^{\varepsilon_{2}}_{2} $};  

\node[below=5pt, font=\small] at (0.5,1) {$L_4$};
\draw (\b,1) -- (1,1) node[pos=0.5, above, font=\small] {};   

\draw (1,1) -- (1,0) node[pos=0.5, right, font=\small] {$e^{\varepsilon_{2}}_{2}$};  

\node[left=5pt, font=\small] at (1,0.5) {$L_2$};

\filldraw (0,0) circle (1.5pt) node[below left] {$(0,0)$};  
\filldraw (1,0) circle (1.5pt) node[below right] {$(1,0)$};  
\filldraw (1,1) circle (1.5pt) node[above right] {$(1,1)$};  
\filldraw (0,1) circle (1.5pt) node[above left] {$(0,1)$};  
\end{tikzpicture}
\end{equation}
Let $h_{1},h_{2},h_{3},h_{4}$ be the homeomorphisms defined earlier:
$$h_{1}: L_{1} \rightarrow I \text{ by } ( t, 0) \mapsto t,$$ 
$$h_{2}: L_{2} \rightarrow I \text{ by } (1,t) \mapsto t,$$
$$h_{3}: L_{3} \rightarrow I \text{ by } (0, t) \mapsto 1-t,$$
$$h_{4}: L_{4} \rightarrow I \text{ by } (t,1) \mapsto 1-t.$$ 
Consider now two maps
$$\nu_{\mathcal{D}}\circ j_{1} \circ \Phi^{\varepsilon_{1}}_{e_{1}}\circ h_{1} \text{ and } \nu_{\mathcal{D}}\circ j_{1} \circ \Phi^{\varepsilon_{2}}_{e_{2}}\circ h_{2},$$
which are continuous and
\begin{align*}
(\nu_{\mathcal{D}}\circ j_{1} \circ \Phi^{\varepsilon_{1}}_{e_{1}}\circ h_{1})(1,0)&= (\nu_{\mathcal{D}}\circ j_{1} \circ \Phi^{\varepsilon_{1}}_{e_{1}})(h_{1}(1,0))\\
&=(\nu_{\mathcal{D}}\circ j_{1})(\Phi^{\varepsilon_{1}}_{e_{1}}(1))\\
&=\tau e^{\varepsilon_{1}}_{1}\\
&= (\nu_{\mathcal{D}}\circ j_{1})(\Phi^{\varepsilon_{2}}_{e_{2}}(0))\\
&= (\nu_{\mathcal{D}}\circ j_{1} \circ \Phi^{\varepsilon_{2}}_{e_{2}}\circ h_{2})(1,0).
\end{align*}
This implies that we can glue together these functions to get a continuous map
$$\phi^{(1)}_{[w,p,w']}: L_{1} \cup L_{2} \rightarrow \mathcal{D}(\cP)$$
such that
$$\phi^{(1)}_{[w,p,w']}|_{L_{1}}=\nu_{\mathcal{D}}\circ j_{1} \circ \Phi^{\varepsilon_{1}}_{e_{1}}\circ h_{1} \text{ and } \phi^{(1)}_{[w,p,w']}|_{L_{2}}=\nu_{\mathcal{D}}\circ j_{1} \circ \Phi^{\varepsilon_{2}}_{e_{2}}\circ h_{2}.$$
Observe that for every $(t,0) \in L_{1}$,
\begin{equation} \label{B1}
\phi^{(1)}_{[w,p,w']}(t,0)= \nu_{\mathcal{D}}(j_{1} (\Phi^{\varepsilon_{1}}_{e_{1}}(t)))=\Phi^{\varepsilon_{1}}_{e_{1}}(t),
\end{equation}
and in particular we have
\begin{equation} \label{f2-r1}
\phi^{(1)}_{[w,p,w']}(0,0)= \Phi^{\varepsilon_{1}}_{e_{1}}(0)= \iota e_{1}^{\varepsilon_{1}}.
\end{equation}
Also for every $(1,t) \in L_{2}$ we have
\begin{equation} \label{B2}
\phi^{(1)}_{[w,p,w']}(1,t)=\nu_{\mathcal{D}}(j_{1}( \Phi^{\varepsilon_{2}}_{e_{2}}(t)))=\Phi^{\varepsilon_{2}}_{e_{2}}(t),
\end{equation}
from which we get that
\begin{equation} \label{f2-r2}
\phi^{(1)}_{[w,p,w']}(1,1)= \Phi^{\varepsilon_{2}}_{e_{2}}(1)= \tau e_{2}^{\varepsilon_{2}}.
\end{equation}
Further we consider two other mappings
$$\phi^{(1)}_{[w,p,w']}|_{L_{1}}=\nu_{\mathcal{D}}\circ j_{1} \circ \Phi^{\varepsilon_{3}}_{e_{3}}\circ h_{3} \text{ and } c_{(\nu_{\mathcal{D}} \circ j_{1})(\tau e^{\varepsilon_{2}}_{2})}\circ h_{4},$$
which are continuous and
\begin{align*}
(\nu_{\mathcal{D}}\circ j_{1} \circ \Phi^{\varepsilon_{3}}_{e_{3}}\circ h_{3})(0,1)&= (\nu_{\mathcal{D}}\circ j_{1} \circ \Phi^{\varepsilon_{3}}_{e_{3}})(h_{3}(0,1))\\
&= (\nu_{\mathcal{D}}\circ j_{1}) ( \Phi^{\varepsilon_{3}}_{e_{3}}(0))\\
&= \iota e_{3}^{\varepsilon_{3}}\\
&= \tau e_{2}^{\varepsilon_{2}}\\
&= (c_{(\nu_{\mathcal{D}} \circ j_{1})(\tau e^{\varepsilon_{2}}_{2})})(1)\\
&= (c_{(\nu_{\mathcal{D}} \circ j_{1})(\tau e^{\varepsilon_{2}}_{2})}\circ h_{4})(0,1).
\end{align*}
It follows that we can glue these two maps together to obtain a continuous map
$$\phi^{(2)}_{[w,p,w']}: L_{3} \cup L_{4} \rightarrow \mathcal{D}(\cP)$$
such that
$$\phi^{(2)}_{[w,p,w']}|_{L_{3}}=\nu_{\mathcal{D}}\circ j_{1} \circ \Phi^{\varepsilon_{3}}_{e_{3}}\circ h_{3} \text{ and } \phi^{(2)}_{[w,p,w']}|_{L_{4}}=c_{(\nu_{\mathcal{D}} \circ j_{1})(\tau e^{\varepsilon_{2}}_{2})}\circ h_{4}.$$
Observe now that for every $(0,t) \in L_{3}$ we have
\begin{equation} \label{B3}
\phi^{(2)}_{[w,p,w']}(0,t)= \nu_{\mathcal{D}}(j_{1}(\Phi^{\varepsilon_{3}}_{e_{3}}(1-t)))=\Phi^{\varepsilon_{3}}_{e_{3}}(1-t),
\end{equation}
and that in particular
\begin{equation} \label{f2-l1}
\phi^{(2)}_{[w,p,w']}(0,0)= \Phi^{\varepsilon_{3}}_{e_{3}}(1)= \tau e_{3}^{\varepsilon_{3}}=\iota e_{1}^{\varepsilon_{1}}.
\end{equation}
Also we have that for every $(t,1) \in L_{4}$,
\begin{equation} \label{B4}
\phi^{(2)}_{[w,p,w']}(t,1)=  (c_{(\nu_{\mathcal{D}} \circ j_{1})(\tau e^{\varepsilon_{2}}_{2})}\circ h_{4})(t,1)= \tau e_{2}^{\varepsilon_{2}},
\end{equation}
which implies that
\begin{equation} \label{f2-l2}
\phi^{(2)}_{[w,p,w']}(1,1)= \tau e_{2}^{\varepsilon_{2}}.
\end{equation}
We get from (\ref{f2-r1}), (\ref{f2-l1}), (\ref{f2-r2}) and (\ref{f2-l2}) that $\phi^{(1)}_{[w,p,w']}$ and $\phi^{(2)}_{[w,p,w']}$ can be glued together to obtain the attaching map
$$\phi_{[w,p,w']}:S_{[w,p,w']}^{1} \rightarrow \mathcal{D}(\cP)$$
of $[w,p,w']$ which maps $S_{[w,p,w']}^{1}$ onto the topological boundary $\partial [w,p,w']=e_{1} \cup e_{2} \cup e_{3}$ of $[w,p,w']$. The justification for the presence of $e_{i}^{\varepsilon_{i}}$ in $\delta [w,p,w']$ is done in the the same way as for cells of type (i). \\
(iii) If $n\geq 4$, then 
$$\delta [w,p,w']=e^{\varepsilon_{1}}_{1} \circ e^{\varepsilon_{2}}_{2} \circ \dots \circ e^{\varepsilon_{n-1}}_{n-1} \circ e^{\varepsilon_{n}}_{n}.$$
It follows that we can find $2 \leq s \leq n-1$ such that 
$$\delta [w,p,w']=e^{\varepsilon_{1}}_{1} \circ e^{\varepsilon_{2}}_{2} \circ \dots \circ e_{s}^{\varepsilon_{s}} \circ e^{\varepsilon_{s+1}}_{s+1} \circ \dots \circ e_{n-1}^{\varepsilon_{n-1}} \circ e^{\varepsilon_{n}}_{n}.$$
The boundary $S^{1}_{[w,p,w']}$ of the attaching disk $D^{2}_{[w,p,w']}$ will be labeled this time as below.
\begin{equation} \label{amp4}
\begin{tikzpicture}[x=4cm, y=4cm, thick]  
\def\a{0.3333}  
\def\b{0.6667}   

\draw (0,0) -- (1,0) node[pos=0.5, below, font=\small] {$e_{1}^{\varepsilon_{1}}$};   
\node[above=5pt, font=\small] at (0.5,0) {$L_1$};

\draw (0,0) -- (0,1) node[pos=0.5, left, font=\small] {$e_{n}^{\varepsilon_{n}}$};   
\node[right=5pt, font=\small] at (0,0.5) {$L_3$};

\draw (0,1) -- (\a,1) node[pos=0.5, above, font=\small] {$e^{\varepsilon_{n-1}}_{n-1}$};  
\draw[dotted] (\a,1) -- (\b,1);  
\node[below=5pt, font=\small] at (0.5,1) {$L_4$};
\draw (\b,1) -- (1,1) node[pos=0.5, above, font=\small] {$e^{\varepsilon_{s+1}}_{s+1}$};   

\draw (1,1) -- (1,\b) node[pos=0.5, right, font=\small] {$e^{\varepsilon_{s}}_{s}$};  
\draw[dotted] (1,\b) -- (1,\a);  
\node[left=5pt, font=\small] at (1,0.5) {$L_2$};
\draw (1,\a) -- (1,0) node[pos=0.5, right, font=\small] {$e^{\varepsilon_{2}}_{2}$};   

\filldraw (0,0) circle (1.5pt) node[below left] {$(0,0)$};  
\filldraw (1,0) circle (1.5pt) node[below right] {$(1,0)$};  
\filldraw (1,1) circle (1.5pt) node[above right] {$(1,1)$};  
\filldraw (0,1) circle (1.5pt) node[above left] {$(0,1)$};  
\end{tikzpicture}
\end{equation}
Let $h_{1},h_{2},h_{3},h_{4}$ be the homeomorphisms defined earlier:
$$h_{1}: L_{1} \rightarrow I \text{ by } ( t, 0) \mapsto t,$$ 
$$h_{2}: L_{2} \rightarrow I \text{ by } (1,t) \mapsto t,$$
$$h_{3}: L_{3} \rightarrow I \text{ by } (0, t) \mapsto 1-t,$$
$$h_{4}: L_{4} \rightarrow I \text{ by } (t,1) \mapsto 1-t.$$ 
Consider now two maps
$$\nu_{\mathcal{D}}\circ j_{1} \circ \Phi^{\varepsilon_{1}}_{e_{1}}\circ h_{1} \text{ and } \nu_{\mathcal{D}}\circ j_{1} \circ (\Phi^{\varepsilon_{2}}_{e_{2}} \ast \dots \ast \Phi^{\varepsilon_{s}}_{e_{s}})\circ h_{2},$$
which are continuous and
\begin{align*}
(\nu_{\mathcal{D}}\circ j_{1} \circ \Phi^{\varepsilon_{1}}_{e_{1}}\circ h_{1})(1,0)&=( \nu_{\mathcal{D}}\circ j_{1} )(\Phi^{\varepsilon_{1}}_{e_{1}}(1))\\
&=\tau e_{1}^{\varepsilon_{1}}\\
&= \iota e_{2}^{\varepsilon_{2}}\\
&= (\nu_{\mathcal{D}}\circ j_{1} \circ (\Phi^{\varepsilon_{2}}_{e_{2}} \ast \dots \ast \Phi^{\varepsilon_{s}}_{e_{s}}))(0)\\
&=(\nu_{\mathcal{D}}\circ j_{1} \circ (\Phi^{\varepsilon_{2}}_{e_{2}} \ast \dots \ast \Phi^{\varepsilon_{s}}_{e_{s}})\circ h_{2})(1,0).
\end{align*}
This shows that we can glue these two maps together to get in return a continuous map
$$\phi^{(1)}_{[w,p,w']}: L_{1} \cup L_{2} \rightarrow \mathcal{D}(\cP)$$
such that
$$\phi^{(1)}_{[w,p,w']}|_{L_{1}}=\nu_{\mathcal{D}}\circ j_{1} \circ \Phi^{\varepsilon_{1}}_{e_{1}}\circ h_{1} \text{ and } \phi^{(1)}_{[w,p,w']}|_{L_{2}}=\nu_{\mathcal{D}}\circ j_{1} \circ (\Phi^{\varepsilon_{2}}_{e_{2}} \ast \dots \ast \Phi^{\varepsilon_{s}}_{e_{s}})\circ h_{2}.$$
Observe now that for every $(t,0) \in L_{1}$,
\begin{equation} \label{C1}
\phi^{(1)}_{[w,p,w']}(t,0)=\nu_{\mathcal{D}}( j_{1} (\Phi^{\varepsilon_{1}}_{e_{1}}(t)))=\Phi^{\varepsilon_{1}}_{e_{1}}(t),
\end{equation}
and then
\begin{equation} \label{f3-r1}
\phi^{(1)}_{[w,p,w']}(0,0)= \Phi^{\varepsilon_{1}}_{e_{1}}(0)= \iota e_{1}^{\varepsilon_{1}}.
\end{equation}
Also we see that for every $(1,t) \in L_{2}$, there is some $i=2,...,s$ such that $t \in [\dfrac{i-2}{s-1}, \dfrac{i-1}{s-1}]$, and then
\begin{align}
\phi^{(1)}_{[w,p,w']}(1,t)&= (\nu_{\mathcal{D}}\circ j_{1} \circ (\Phi^{\varepsilon_{2}}_{e_{2}} \ast \dots \ast \Phi^{\varepsilon_{s}}_{e_{s}}) \circ h_{2})(1,t) \notag\\
&= \nu_{\mathcal{D}}(j_{1} (\Phi^{\varepsilon_{2}}_{e_{2}}  \ast \dots \ast \Phi^{\varepsilon_{s}}_{e_{s}} )(t)))\notag \\
&=\nu_{\mathcal{D}}(j_{1} (\Phi^{\varepsilon_{i}}_{e_{i}}((s-1)t-(i-2))) \notag\\
&= \Phi^{\varepsilon_{i}}_{e_{i}}((s-1)t-(i-2)). && \label{C2}
\end{align}
It follows from the above that
\begin{equation} \label{f3-r2}
\phi^{(1)}_{[w,p,w']}(1,1)= \Phi^{\varepsilon_{s}}_{e_{s}}(1)= \tau e_{s}^{\varepsilon_{s}}.
\end{equation}
Next we consider another pair of continuous maps
$$\nu_{\mathcal{D}}\circ j_{1} \circ \Phi^{\varepsilon_{n}}_{e_{n}}\circ h_{3} \text{ and } \nu_{\mathcal{D}}\circ j_{1} \circ (\Phi^{\varepsilon_{s+1}}_{e_{s+1}} \ast \dots \ast \Phi^{\varepsilon_{n-1}}_{e_{n-1}})\circ h_{4},$$
for which we see that
\begin{align*}
(\nu_{\mathcal{D}}\circ j_{1} \circ \Phi^{\varepsilon_{n}}_{e_{n}}\circ h_{3} )(0,1)&= (\nu_{\mathcal{D}}\circ j_{1}) (\Phi^{\varepsilon_{n}}_{e_{n}}(0))\\
&= \iota e_{n}^{\varepsilon_{n}}\\
&= \tau e_{n-1}^{\varepsilon_{n-1}}\\
&=(\nu_{\mathcal{D}}\circ j_{1} \circ (\Phi^{\varepsilon_{s+1}}_{e_{s+1}} \ast \dots \ast \Phi^{\varepsilon_{n-1}}_{e_{n-1}}))(1)\\
&=(\nu_{\mathcal{D}}\circ j_{1} \circ (\Phi^{\varepsilon_{s+1}}_{e_{s+1}} \ast \dots \ast \Phi^{\varepsilon_{n-1}}_{e_{n-1}})\circ h_{4})(0,1).
\end{align*}
It follows that we can glue these maps to get a continuous map
$$\phi^{(2)}_{[w,p,w']}: L_{3} \cup L_{4} \rightarrow \mathcal{D}(\cP)$$
such that
$$\phi^{(2)}_{[w,p,w']}|_{L_{3}}=\nu_{\mathcal{D}}\circ j_{1} \circ \Phi^{\varepsilon_{n}}_{e_{n}}\circ h_{3}  \text{ and } \phi^{(2)}_{[w,p,w']}|_{L_{4}}=\nu_{\mathcal{D}}\circ j_{1} \circ (\Phi^{\varepsilon_{s+1}}_{e_{s+1}} \ast \dots \ast \Phi^{\varepsilon_{n-1}}_{e_{n-1}})\circ h_{4}.$$
Observe that for every $(0,t) \in L_{3}$ we have that
\begin{equation} \label{C3}
\phi^{(2)}_{[w,p,w']}(0,t)=\nu_{\mathcal{D}}( j_{1} (\Phi^{\varepsilon_{n}}_{e_{n}}(1-t)))=\Phi^{\varepsilon_{n}}_{e_{n}}(1-t),
\end{equation}
and in particular we have 
\begin{equation} \label{f3-l1}
\phi^{(2)}_{[w,p,w']}(0,0)= \Phi^{\varepsilon_{n}}_{e_{n}}(1)= \tau e_{n}^{\varepsilon_{n}}=\iota e_{1}^{\varepsilon_{1}}.
\end{equation}
Letting now $r=n-s-1$, we observe that for every $(t,1) \in L_{4}$ there is some $i=0,...,r-1$ such that $1-t \in [\dfrac{i}{r}, \dfrac{i+1}{r}]$, and 
\begin{align}
\phi^{(2)}_{[w,p,w']}(t,1)&=(\Phi^{\varepsilon_{s+1}}_{e_{s+1}} \ast \dots \ast \Phi^{\varepsilon_{n-1}}_{e_{n-1}})(1-t) \notag\\
&= \Phi^{\varepsilon_{s+1+i}}_{e_{s+1+i}}(r(1-t)-i). && \label{C4}
\end{align}
It follows from this that
\begin{equation} \label{f3-l2}
\phi^{(2)}_{[w,p,w']}(1,1)=\Phi^{\varepsilon_{s+1}}_{e_{s+1}}(0)=\iota e_{s+1}^{\varepsilon_{s+1}}=\tau e_{s}^{\varepsilon_{s}}.
\end{equation}
Taking now into account (\ref{f3-r1}), (\ref{f3-l1}), (\ref{f3-r2}) and (\ref{f3-l2}) we conclude that $\phi^{(1)}_{[w,p,w']}$ can be glued with $\phi^{(2)}_{[w,p,w']}$ to obtain the attaching map 
$$\phi_{[w,p,w']}: S^{1}_{[w,p,w']} \rightarrow \mathcal{D}(\cP)$$
of $[w,p,w']$. We see that $\phi_{[w,p,w']}$ maps $S^{1}_{[w,p,w']}$ onto the set
$$\partial [w,p,w']=e_{1} \cup \dots \cup e_{n}$$
which is the topological boundary of $[w,p,w']$. The presence of $e^{\varepsilon_{1}}_{1}$ and $e^{\varepsilon_{n}}_{n}$ in $\delta [w,p,w']$ is justified in the same way as before. The presence of $e^{\varepsilon_{i}}_{i}$ for $2\leq i \leq s$ follows from (\ref{C2}). Finally, the presence of $e_{s+1+i}^{\varepsilon_{s+1+i}}$ is justified from (\ref{C4}).

In what follows we will write for short $F \times \mathbf{p} \times F=\mathfrak{C}$ and we let $\mathcal{D}(\cP) \sqcup F. \mathbf{p}.F$ be the quotient space
$$\mathcal{D}(\cP) \sqcup_{\phi_{\mathfrak{c}}} D^{2}_{\mathfrak{c}}= \left(\mathcal{D}(\cP) \underset{ \mathfrak{c} \in \mathfrak{C}}{\sqcup} D^{2}_{\mathfrak{c}}\right)/\sim$$
where $\sim$ is the equivalence relation generated by all pairs
\begin{equation} \label{rel22}
(u, \phi_{\mathfrak{c}}(u)) \text{ with } u \in S^{1}_{\mathfrak{c}} \text{ and } \mathfrak{c} \in \mathfrak{C}.
\end{equation}
Alternatively $\mathcal{D}(\cP) \sqcup F. \mathbf{p}.F$ is given by the following push out diagram
\begin{equation*}
\xymatrix@R=30pt @C=60pt{ \underset{\mathfrak{c}}{\sqcup}S^{1}_{\mathfrak{c}} \ar[r]^{\mu} \ar[d]^{\phi} &  \underset{\mathfrak{c}}{\sqcup}D^{2}_{\mathfrak{c}} \ar[d]^{\nu_{\mathfrak{D}} \mu_{2}}\\
\mathcal{D}(\cP) \ar[r]^{\nu_{\mathfrak{D}}\mu_{1}} & \mathcal{D}(\cP) \sqcup_{\phi_{\mathfrak{c}}} D^{2}_{\mathfrak{c}}}
\end{equation*}
where 
\begin{align*}
& \mu && \text{is the inclusion of the boundary,}\\
&\phi=\underset{\mathfrak{c}}{\sqcup} \phi_{\mathfrak{c}} && \text{is the direct sum map,}\\
&\mu_{1}: \mathcal{D}(\cP) \rightarrow \mathcal{D}(\cP) \underset{\mathfrak{c} \in \mathfrak{C}}{\sqcup} D^{2}_{\mathfrak{c}} ,\\
& \mu_{2}: \underset{\mathfrak{c} \in \mathfrak{C}}{\sqcup} D^{2}_{\mathfrak{c}} \rightarrow \mathcal{D}(\cP) \underset{\mathfrak{c} \in \mathfrak{C}}{\sqcup} D^{2}_{\mathfrak{c}} && \text{are the inclusion maps, and}\\
&\nu_{\mathfrak{D}}: \mathcal{D}(\cP) \underset{\mathfrak{c} \in \mathfrak{C}}{\sqcup} D^{2}_{\mathfrak{c}} \rightarrow \left(\mathcal{D}(\cP) \underset{\mathfrak{c} \in \mathfrak{C}}{\sqcup} D^{2}_{\mathfrak{c}}\right)/\sim && \text{is the quotient map.}
\end{align*}
For every $\mathfrak{c} \in \mathfrak{C}$ we let
$$\mu_{\mathfrak{c}}: D^{2}_{\mathfrak{c}} \rightarrow \underset{\mathfrak{c} }{\sqcup} D^{2}_{\mathfrak{c}}$$
to be the inclusion map. The characteristic map of a 2-cell $\mathfrak{c} \in \mathfrak{C}$ is the composite
$$\Phi_{\mathfrak{c}}=\nu_{\mathfrak{D}} \circ \mu_{2} \circ \mu_{\mathfrak{c}}.$$
Now we will describe in topological terms the bi action of $F$ on $F\times \mathbf{p} \times F$. For a fixed pair of words $z,z'$ consider the continuous mapping $\nu_{\mathfrak{D}} \circ \mu_{1} \circ \rho_{(z,z')}: \mathcal{D}(\cP) \rightarrow \mathcal{D}(\cP) {\sqcup_{\phi_\mathfrak{c}}} D^{2}_{\mathfrak{c}}$ where $\rho_{(z,z')}: \mathcal{D}(\cP) \rightarrow \mathcal{D}(\cP)$ is the mapping of diagram (\ref{po2}). Further, for every $\mathfrak{c} \in \mathfrak{C}$, we define
$$\eta_{\mathfrak{c}}: D^{2}_{\mathfrak{c}} \rightarrow \mathcal{D}(\cP) \sqcup_{\phi_{\mathfrak{c}}} D^{2}_{\mathfrak{c}},$$
by
$$\eta_{\mathfrak{c}}(u)=\nu_{\mathfrak{D}}(\mu_{2}(\mu_{z.\mathfrak{c}.z'}(u))).$$
Again this mapping is continuous and so is the extension
$$\eta=\underset{\mathfrak{c}}{\sqcup}\eta_{\mathfrak{c}}: \underset{\mathfrak{c}}{\sqcup} D^{2}_{\mathfrak{c}} \rightarrow \mathcal{D}(\cP) \sqcup_{\phi_{\mathfrak{c}}} D^{2}_{\mathfrak{c}}$$
of all $\eta_{\mathfrak{c}}$. Now we check that $\eta \circ \mu=(\nu_{\mathfrak{D}} \circ \mu_{1} \circ \rho_{(z,z')}) \circ \phi$.
Let $u \in S^{1}_{\mathfrak{c}}$ for some particular $\mathfrak{c} \in \mathfrak{C}$. The proof when $\mathfrak{c}$ is a 2-cell of type (iii) goes like this. Assume that $u \in L_{1} \cup L_{2}$. If $u \in L_{1}$ then
\begin{align*}
((\nu_{\mathfrak{D}} \circ \mu_{1}\circ \rho_{(z,z')}) \circ \phi)(u)&= (\nu_{\mathfrak{D}} \circ \mu_{1}\circ \rho_{(z,z')})(\phi_{\mathfrak{c}}(u))\\
&= (\nu_{\mathfrak{D}}\circ \mu_{1} \circ \rho_{(z,z')})((\nu_{\mathcal{D}} \circ j_{1}\circ \Phi^{\varepsilon_{1}}_{e_{1}} \circ h_{1})(u)) \tag{$u \in L_{1}$}\\
&= \nu_{\mathfrak{D}}(\mu_{1}(((\rho_{(z,z')} \circ \nu_{\mathcal{D}}\circ j_{1} \circ \Phi^{\varepsilon_{1}}_{e_{1}}) (h_{1}(u))))\\
&=\nu_{\mathfrak{D}}(\mu_{1}(((\nu_{\mathcal{D}} \circ j_{1}\circ \Phi^{\varepsilon_{1}}_{z.e_{1}.z'}) (h_{1}(u)))) \tag{by Lemma \ref{ronu}}\\
&=\nu_{\mathfrak{D}} (\mu_{1}(((\nu_{\mathcal{D}} \circ j_{1} \circ \Phi^{\varepsilon_{1}}_{z.e_{1}.z'} \circ h_{1})(u)))\\
&=\nu_{\mathfrak{D}} (\mu_{1}(\phi_{z.\mathfrak{c}.z'}(u)))\tag{$u \in S^{1}_{z.\mathfrak{c}.z'}$}\\
&=(\nu_{\mathfrak{D}} \circ \mu_{2} \circ \mu_{z.\mathfrak{c}.z'})(u) \tag{by the definition of $\sim$}\\
&= (\eta \circ \mu)(u). \tag{$u \in S^{1}_{\mathfrak{c}}$}
\end{align*}
If $u\in L_{2}$, then
{\small
\begin{align*}
((\nu_{\mathfrak{D}}\circ \mu_{1} \circ \rho_{(z,z')}) \circ \phi)(u)&= (\nu_{\mathfrak{D}}\circ \mu_{1} \circ \rho_{(z,z')})(\phi_{\mathfrak{c}}(u))\\
&= (\nu_{\mathfrak{D}}\circ \mu_{1} \circ \rho_{(z,z')}) ((\nu_{\mathcal{D}} \circ j_{1} \circ (\Phi^{\varepsilon_{2}}_{e_{2}} \ast \dots \ast \Phi^{\varepsilon_{s}}_{e_{s}}) \circ h_{2})(u)) \tag{$u \in L_{2}$}\\
&= (\nu_{\mathfrak{D}}\circ \mu_{1} ) ((\rho_{(z,z')} \circ \nu_{\mathcal{D}} \circ j_{1} \circ (\Phi^{\varepsilon_{2}}_{e_{2}} \ast \dots \ast \Phi^{\varepsilon_{s}}_{e_{s}}) \circ h_{2})(u))\\
&= (\nu_{\mathfrak{D}}\circ \mu_{1} ) (((\rho_{(z,z')} \circ \nu_{\mathcal{D}} \circ j_{1} \circ \Phi^{\varepsilon_{s}}_{e_{2}}) \ast \dots \ast (\rho_{(z,z')} \circ \nu_{\mathcal{D}} \circ j_{1} \circ \Phi^{\varepsilon_{s}}_{e_{s}})) \circ h_{2})(u))\\
&=(\nu_{\mathfrak{D}} \circ \mu_{1})((\rho_{(z,z')} \circ \nu_{\mathcal{D}}\circ j_{1} \circ \Phi^{\varepsilon_{2}}_{e_{2}}) \ast \dots \ast (\rho_{(z,z')} \circ \nu_{\mathcal{D}} \circ j_{1}\circ \Phi^{\varepsilon_{s}}_{e_{s}})(h_{2}(u))) \\
&=\nu_{\mathfrak{D}} (\mu_{1}(((\nu_{\mathcal{D}} \circ j_{1} \circ \Phi^{\varepsilon_{2}}_{z.e_{2}.z'}) \ast  \dots \ast (\nu_{\mathcal{D}} \circ j_{1} \circ \Phi^{\varepsilon_{s}}_{z.e_{s}.z'})) (h_{2}(u))) )\tag{by Lemma \ref{ronu}}\\
&= \nu_{\mathfrak{D}}(\mu_{1}(((\nu_{\mathcal{D}} \circ j_{1} \circ(\Phi^{\varepsilon_{2}}_{z.e_{2}.z'} \ast \dots \ast \Phi^{\varepsilon_{s}}_{z.e_{s}.z'})) \circ h_{2})(u)))\\
&= \nu_{\mathfrak{D}} (\mu_{1}(\phi_{z.\mathfrak{c}.z'} (u))) \tag{$u \in S^{1}_{z.\mathfrak{c}.z'}$}\\
&=  (\nu_{\mathfrak{D}} \circ \mu_{2} \circ \mu_{z.\mathfrak{c}.z'})(u) \tag{by the definition of $\sim$}\\
&= (\eta \circ \mu)(u). \tag{$u \in S^{1}_{\mathfrak{c}}$}
\end{align*}
}
The proof when $u \in L_{3} \cup L_{4}$ is done in a similar fashion. 

Now we consider the case when $\mathfrak{c}$ is a 2-cell of type (ii). The proof when $u \in L_{1} \cup L_{2}$ is identical with the corresponding one for (iii) type cells. We check the case when $u \in L_{3} \cup L_{4}$. If $u \in L_{3}$, then 
\begin{align*}
((\nu_{\mathfrak{D}} \circ \mu_{1}\circ \rho_{(z,z')}) \circ \phi)(u)&= (\nu_{\mathfrak{D}} \circ \mu_{1}\circ \rho_{(z,z')})(\phi_{\mathfrak{c}}(u))\\
&= (\nu_{\mathfrak{D}}\circ \mu_{1} \circ \rho_{(z,z')})((\nu_{\mathcal{D}} \circ j_{1}\circ \Phi^{\varepsilon_{3}}_{e_{3}} \circ h_{3})(u)) \tag{$u \in L_{3}$}\\
&= \nu_{\mathfrak{D}}(\mu_{1}(((\rho_{(z,z')} \circ \nu_{\mathcal{D}}\circ j_{1} \circ \Phi^{\varepsilon_{3}}_{e_{3}}) (h_{3}(u))))\\
&=\nu_{\mathfrak{D}}(\mu_{1}(((\nu_{\mathcal{D}} \circ j_{1}\circ \Phi^{\varepsilon_{3}}_{z.e_{3}.z'}) (h_{3}(u)))) \tag{by Lemma \ref{ronu}}\\
&=\nu_{\mathfrak{D}} (\mu_{1}(((\nu_{\mathcal{D}} \circ j_{1} \circ \Phi^{\varepsilon_{3}}_{z.e_{3}.z'} \circ h_{3})(u)))\\
&=\nu_{\mathfrak{D}} (\mu_{1}(\phi_{z.\mathfrak{c}.z'}(u)))\tag{$u \in S^{1}_{z.\mathfrak{c}.z'}$}\\
&=(\nu_{\mathfrak{D}} \circ \mu_{2} \circ \mu_{z.\mathfrak{c}.z'})(u) \tag{by the definition of $\sim$}\\
&= (\eta \circ \mu)(u). \tag{$u \in S^{1}_{\mathfrak{c}}$}
\end{align*}
If $u \in L_{4}$ then
\begin{align*}
((\nu_{\mathfrak{D}}\circ \mu_{1} \circ \rho_{(z,z')}) \circ \phi)(u)&= (\nu_{\mathfrak{D}}\circ \mu_{1} \circ \rho_{(z,z')})(\phi_{\mathfrak{c}}(u))\\
&= (\nu_{\mathfrak{D}}\circ \mu_{1} \circ \rho_{(z,z')})((c_{(\nu_{\mathcal{D}} \circ j_{1})(\tau e^{\varepsilon_{2}}_{2})}\circ h_{4})(u)) \tag{$u \in L_{4}$} \\
&=(\nu_{\mathfrak{D}}\circ \mu_{1} )(( \rho_{(z,z')} \circ c_{(\nu_{\mathcal{D}} \circ j_{1})(\tau e^{\varepsilon_{2}}_{2})} \circ h_{4})(u))\\
&=(\nu_{\mathfrak{D}}\circ \mu_{1} )(( c_{(\nu_{\mathcal{D}} \circ j_{1})(\tau(z.e^{\varepsilon_{2}}_{2}.z'))} \circ h_{4})(u)) \tag{by (\ref{gmc})}\\
&= \nu_{\mathfrak{D}} (\mu_{1}(\phi_{z.\mathfrak{c}.z'} (u))) \tag{$u \in S^{1}_{z.\mathfrak{c}.z'}$}\\
&=  (\nu_{\mathfrak{D}} \circ \mu_{2} \circ \mu_{z.\mathfrak{c}.z'})(u) \tag{by the definition of $\sim$}\\
&= (\eta \circ \mu)(u). \tag{$u \in S^{1}_{\mathfrak{c}}$}
\end{align*}
Finally we check the case when $\mathfrak{c}$ is a 2-cell of type (i). We will make the proof for the case when $u \in L_{1} \cup L_{2}$ since the proof when $u \in L_{3} \cup L_{4}$ is similar. The proof for the case when $u \in L_{1}$ is identical with the corresponding proof for 2-cells of type (iii) and is omitted. We check the case when $u \in L_{2}$. 
\begin{align*}
((\nu_{\mathfrak{D}}\circ \mu_{1} \circ \rho_{(z,z')}) \circ \phi)(u)&= (\nu_{\mathfrak{D}}\circ \mu_{1} \circ \rho_{(z,z')})(\phi_{\mathfrak{c}}(u))\\
&= (\nu_{\mathfrak{D}}\circ \mu_{1} \circ \rho_{(z,z')})((c_{(\nu_{\mathcal{D}} \circ j_{1})(\tau e^{\varepsilon_{1}}_{1})}\circ h_{2})(u)) \tag{$u \in L_{2}$} \\
&=(\nu_{\mathfrak{D}}\circ \mu_{1} )(( \rho_{(z,z')} \circ c_{(\nu_{\mathcal{D}} \circ j_{1})(\tau e^{\varepsilon_{1}}_{1})} \circ h_{2})(u))\\
&=(\nu_{\mathfrak{D}}\circ \mu_{1} )(( c_{(\nu_{\mathcal{D}} \circ j_{1})(z.\tau e^{\varepsilon_{1}}_{1}.z')} \circ h_{2})(u)) \tag{by (\ref{gmc})}\\
&= \nu_{\mathfrak{D}} (\mu_{1}(\phi_{z.\mathfrak{c}.z'} (u))) \tag{$u \in S^{1}_{z.\mathfrak{c}.z'}$}\\
&=  (\nu_{\mathfrak{D}} \circ \mu_{2} \circ \mu_{z.\mathfrak{c}.z'})(u) \tag{by the definition of $\sim$}\\
&= (\eta \circ \mu)(u). \tag{$u \in S^{1}_{\mathfrak{c}}$}
\end{align*}
Since $\eta \circ \mu=(\nu_{\mathfrak{D}} \circ \mu_{1} \circ \rho_{(z,z')}) \circ \phi$ it follows that there is a continuous map $\xi_{(z,z')}: \mathcal{D}(\cP) \sqcup_{\phi_{\mathfrak{c}}} D^{2}_{\mathfrak{c}} \rightarrow \mathcal{D}(\cP) \sqcup_{\phi_{\mathfrak{c}}} D^{2}_{\mathfrak{c}}$ such that the following diagram commutes
\begin{equation} \label{po3}
\xymatrix@R=30pt @C=60pt{ \underset{\mathfrak{c}}{\sqcup}S^{1}_{\mathfrak{c}} \ar[r]^{\mu} \ar[d]^{\phi} &  \underset{\mathfrak{c}}{\sqcup}D^{2}_{\mathfrak{c}} \ar[d]^{\nu_{\mathfrak{D}} \mu_{2}}  \ar@/^1pc/[rdd]^{\eta} \\
\mathcal{D}(\cP) \ar[r]^{\nu_{\mathfrak{D}}\mu_{1}} \ar@/_1pc/[rrd]_{\nu_{\mathfrak{D}}\mu_{1}\rho_{(z,z')}} & \mathcal{D}(\cP) \sqcup_{\phi_{\mathfrak{c}}} D^{2}_{\mathfrak{c}}\ar@{-->}[rd]|{\xi_{(z,z')}} \\
&          & \mathcal{D}(\cP) \sqcup_{\phi_{\mathfrak{c}}} D^{2}_{\mathfrak{c}}}
\end{equation}
Now we see how the characteristic maps of a 2-cell $\sigma' \in F.\mathbf{p}.F$ and that of $z.\sigma'.z'$ are related to each other. It follows that for every $u \in D^{2}_{\sigma'}$,
\begin{align*}
(\xi_{(z,z')} \circ \Phi_{\sigma'})(u)&= (\xi_{(z,z')} \circ \nu_{\mathfrak{D}} \circ \mu_{2}\circ \mu_{\sigma'})(u)\\
&= \eta_{\sigma'}(u) \tag{from the commutativity of (\ref{po3})}\\
&= \nu_{\mathfrak{D}}(\mu_{2}(\mu_{z.\sigma'.z'}(u)))\\
&= \Phi_{z.\sigma'.z'}(u),
\end{align*}
hence we have that
\begin{equation} \label{relch2}
\xi_{(z,z')} \circ \Phi_{\sigma'}=\Phi_{z.\sigma'.z'}.
\end{equation}
Assume that $\sigma'=[u,\sigma,v]$ with $\sigma \in \mathbf{p}$. Since
$$z.\sigma'.z'=z.[u,  \sigma ,v].z'= [zu,  \sigma, vz']=zu. \sigma .vz',$$
then we may replace in (\ref{relch2}) $\Phi_{\sigma'}$ by $\Phi_{[u, \sigma, v]}$ and $\Phi_{z.\sigma'.z'}$ by $\Phi_{[zu,\sigma,vz']}$ to obtain
\begin{equation} \label{relch21}
\xi_{(z,z')} \circ \Phi_{[u,\sigma, v]}=\Phi_{[zu,\sigma,vz']},
\end{equation}
which explains why we made the definition $z.[u,\sigma,v].z'=[zu ,\sigma, vz']$.

We conclude by remarking that the commutativity of the lower triangle of (\ref{po3}) shows that the bi action of $F$ on $\mathcal{D}(\cP)$ is preserved after attaching 2-cells from $F.\mathbf{p}.F$.

\section{Adding extra 2-cells to $\mathcal{D}(\cP)$-a special case}

The special case we deal with in this section is the case of 2-cells that arise by resolving critical pairs of positive edges with the same initial in $\mathcal{D}(\cP)$. We begin with the definition of such pairs.

\begin{definition} \label{critpairs}
A pair $(e, f)$ of positive edges with the same initial forms a critical pair if either
\begin{itemize}
	\item [(1)] One of the pair is both left and right reduced (a critical pair of inclusion type), or
	\item[(2)] One of the pair is left reduced but not right reduced, the other is right reduced but not left reduced and they are not disjoint (a critical pair of overlapping type).
\end{itemize}
In the first case $e=(u, r, 1, v)$, $f=(\lambda, s, 1, \lambda)$ and $s_{+}=ur_{+}v$. In the second case, $e=(\lambda, r, 1, v)$, $f=(u, s, 1, \lambda)$ and $r_{+}v=us_{+}$ where $r_{+}$ overlaps with $s_{+}$ and both $u,v \neq \lambda$, otherwise we would have a pair as in the first case. We say that a critical pair $(e,f)$ is resolvable if one of the following occurs.
\begin{itemize}
\item[(i)] $\tau e=\tau f$.
\item[(ii)] $\tau e \neq \tau f$ and
\begin{itemize}
\item[(ii-1)] There is a positive path $\mathfrak{l}_{(e,f)}$ from $\tau e$ to $\tau f$, or
\item[(ii-2)] There is is a positive path $\mathfrak{r}_{(f,e)}$ from $\tau f$ to $\tau e$, or
\item[(ii-3)] There are positive paths $\lambda_{(e,f)}$ and $\rho_{(e,f)}$ from $\tau e$ and $\tau f$ respectively to a common vertex $w$ which is distinct from $\tau f$ and $\tau e$.
\end{itemize}
\end{itemize}

\end{definition}
Let now $\mathfrak{p}$ be any nonempty set of resolvable critical pairs of positive edges. For every pair $(e,f) \in \mathfrak{p}$, we attach a 2-cell $\sigma$ in the following way. If the pair is of type (i), then the corresponding 2-cell is attached along the closed path
$$\delta \sigma= e \circ f^{-1}.$$
If the pair is of type (ii-1), then $\sigma$ is attached along the path
$$\delta \sigma= e \circ \mathfrak{l}_{(e,f)} \circ f^{-1}.$$
If the pair is of type (ii-2), then $\sigma$ is attached along the path
$$\delta \sigma= e \circ \mathfrak{r}_{(f,e)}^{-1} \circ f^{-1},$$
and if the pair is of type (ii-3), then $\sigma$ is attached along the path
$$\delta \sigma=e \circ \lambda_{(e,f)} \circ \rho_{(e,f)}^{-1} \circ f^{-1}.$$
In each case $\delta \sigma$ will be called the combinatorial boundary of $\sigma$. 

We let $\mathbf{p}$ be the set of 2-cells attached to $\Gamma(\cP)$ as above. It is obvious that $\mathbf{p}$ is in a 1-1 correspondence with $\mathfrak{p}$. Further, for every $\sigma \in \mathbf{p}$ which arises from some $(e,f) \in \mathfrak{p}$ and therefore is attached along the respective $\delta \sigma$ and for every $u,v \in F$ we attach a new 2-cell denoted by $u.\sigma.v$ along the closed path $u.\delta \sigma .v$ and define
$$\delta(u. \sigma. v)= u.\delta \sigma .v.$$
Recollecting, we have attached 2-cells $\sigma$ along closed paths
$$\delta \sigma= g_{1} \circ \dots \circ g_{s} \circ \tilde{g}_{r}^{-1} \circ \dots \circ \tilde{g}_{1}^{-1}.$$
If the 2-cell is of type (i), then $r=s=1$, if it is of type (ii-1), then $r=1$, if it is of type (ii-2) then $s=1$ and if it is of type (ii-3) then $r,s \geq 2$.

We let
$$F.\mathbf{p}.F=\{ u.\sigma.v| \sigma \in \mathbf{p}\}.$$
There is also a natural bi action of $F$ on the set of 2-cells from $F.\mathbf{p}.F$ defined by
$$z.(u.\sigma. v).z'=zu. \sigma. vz'$$
for every $u.\sigma.v \in F.\mathbf{p}.F$ and every $z,z' \in F$. This bi action sometimes is called translation. Unlike square 2-cells, here every 2-cell from $F.\mathbf{p}.F$ is uniquely determined by a 2-cell $\sigma \in \mathbf{p}$ and a pair of words $u,v \in F$. This is shown in the following lemma.
\begin{lemma} \label{unambp}
If $\sigma, \sigma' \in \mathbf{p}$ and $\alpha, \beta, \alpha', \beta' \in F$, then $\alpha. \sigma. \beta=\alpha'. \sigma'. \beta'$ if and only if $\alpha=\alpha', \beta=\beta'$ and $\sigma=\sigma'$.
\end{lemma}
\begin{proof}
We first exclude the case when $\sigma$ arises from the resolution of a critical pair of inclusion type, let say $e=(u, r, 1, v)$ and $f=(\lambda,s,1,\lambda)$ with $s_{+}=u r_{+}v$, while $\sigma'$ arises from an overlapping pair like $e'=(\lambda, r', 1, v')$ and $f'=(u',s',1,\lambda)$ where $r'_{+}v'=u's'_{+}$ and $r'_{+}$ overlap with $s'_{+}$ but $u', v' \neq \lambda$. Indeed, equating the combinatorial boundaries of $\alpha. \sigma. \beta$ and $\alpha'. \sigma. \beta'$ we have
$$(\alpha u, r, 1, v \beta)=\alpha. e. \beta=\alpha'. e'. \beta'= (\alpha', r' , 1, v'\beta'),$$
which implies $\alpha u=\alpha'$, and since
$$(\alpha, s, 1, \beta)=\alpha. f. \beta= \alpha'.f'. \beta'=(\alpha' u', s', 1, \beta'),$$
we obtain that $\alpha=\alpha'u'$. It follows that $\alpha=\alpha' u'=\alpha u u'$, which is impossible since $u' \neq 1$. 
There are now two possible cases remaining for $\sigma, \sigma'$. First, they both arise from resolving critical pairs of inclusion type. Let $e=(u, r, 1, v)$ and $f=(\lambda,s,1,\lambda)$ with $s_{+}=u r_{+}v$ be the edges that resolve to create $\sigma$, and $e'=(u', r', 1, v')$ and $f'=(\lambda,s',1,\lambda)$ with $s'_{+}=u' r'_{+}v'$ be the edges that resolve to create $\sigma'$. Since $\alpha. \sigma. \beta=\alpha'. \sigma. \beta'$, then
$$(\alpha, s, 1, \beta)=\alpha.f. \beta=\alpha'. f'. \beta'=(\alpha', s', 1, \beta'),$$
from which we have that $\alpha=\alpha', \beta=\beta', s=s'$. But we also have that
$$(\alpha u, r, 1, v \beta)=\alpha. e. \beta=\alpha'. e'. \beta'=(\alpha' u', r', 1, v' \beta'),$$
which imply that $r=r'$, $u=u'$ and $v=v'$. So we get that $e=e'$ and $f=f'$. The second case is when both 2-cells $\sigma, \sigma'$ arise from overlapping pairs. Let $e=(\lambda, r, 1, v)$ and $f=(u,s,1,\lambda)$ where $r_{+}v=us_{+}$ and $r_{+}$ overlap with $s_{+}$ be edges that generate $\sigma$. Similarly, $e'=(\lambda, r', 1, v)$ and $f=(u',s',1,\lambda)$ where $r'_{+}v'=u's'_{+}$ and $r'_{+}$ overlap with $s'_{+}$ be edges that generate $\sigma'$. From the condition that $\alpha. \sigma. \beta=\alpha'. \sigma. \beta'$, we have that
$$(\alpha, r, 1, v\beta)=\alpha. e. \beta= \alpha' . e'. \beta'=(\alpha', r', 1, v'\beta'),$$
from which we infer that $\alpha=\alpha'$, $r=r'$ and $v\beta=v'\beta'$. Also we have that
$$(\alpha u, s, 1, \beta)=\alpha. f. \beta=\alpha'. f'. \beta'=(\alpha'u', s', 1, \beta'),$$
which implies that $s=s'$, $\alpha u=\alpha' u'$ and $\beta=\beta'$. After cancelations we see that $u=u'$ and $v=v'$, therefore $e=e'$ and $f=f'$, hence $\sigma=\sigma'$. 
\end{proof}
\begin{corollary} \label{cunambp}
If $\alpha.\sigma.\beta$ is a translate of some $\sigma \in \mathbf{p}$ which arises from the resolution of a critical pair $(e,f)$ of positive edges with the same initial $\eta$, then $\alpha.\sigma.\beta$ is uniquely determined by the pair of edges $(\alpha.e.\beta, \alpha.f.\beta)$.
\end{corollary}
\begin{proof}
If $\alpha'.\sigma'.\beta'$ is the translate of another $\sigma' \in \mathbf{p}$ which arises from the resolution of the critical pair $(e',f')$ such that $\alpha.e.\beta=\alpha'.e'.\beta'$ and $\alpha.f.\beta=\alpha'.f'.\beta'$, then from the proof of Lemma \ref{unambp} we see that $\alpha=\alpha'$, $\beta=\beta'$, $e=e'$ and $f=f'$. This shows that $\alpha.\sigma.\beta=\alpha'.\sigma'.\beta'$.
\end{proof}
\begin{notation} \label{nsi}
In \cite{SES} it is introduced the following notation. If $\alpha.\sigma.\beta$ is a translate of some $\sigma \in \mathbf{p}$ which arises from the resolution of a critical pair $(e,f)$ of positive edges with the same initial $\eta$, then the 2-cell $\alpha.\sigma.\beta$ is denoted by $[\alpha \eta \beta, (\alpha.e.\beta, \alpha.f.\beta)]$. The above corollary implies that this notation is unambiguous.
\end{notation}
Finally we observe that no 2-cell can be at the same time square and arising from another set $\mathbf{p}$ of closed paths as the following shows.
\begin{lemma} \label{sqneqp}
\begin{itemize}
\item[(i)] The set of square 2-cells and $F.\mathbf{p}.F$ are disjoint whenever $\mathbf{p}$ arises from the resolutions of a set $\mathfrak{p}$ of critical pairs.
\item[(ii)] The set of square 2-cells and $F.\mathbf{p}.F$ are disjoint whenever $\mathbf{p}$ is a set of closed paths as in section \ref{gc}
\end{itemize}
\end{lemma}
\begin{proof}
(i) Assume that there is some $u.\sigma.v \in F.\mathbf{p}.F$ which is square. If $\sigma$ is obtained by resolving a critical pair $(e,f)$, then by Lemma \ref{cunambp} the pair $(u.e.v, u.f.v)$ uniquely determines $u.\sigma.v$. Since this 2-cell is also square, let say $[g,h]$, then from Corollary \ref{cunamb} it is uniquely determined by the pair $(g.\iota(h), \iota(g).h)$ where the edges transform disjoint subwords of $\iota(g)\iota(h)$ unlike the edges of $(u.e.v, u.f.v)$. This is a contradiction.\\
(ii) Assume that for some $p \in \mathbf{p}$ and $u,v \in F$, $[u,p,v]=[e,f]$ where $e,f$ are positive edges. Letting $\delta p= e^{\varepsilon_{1}}_{1} \circ \dots \circ e^{\varepsilon_{n}}_{n}$, it follows that
$$u. e^{\varepsilon_{1}}_{1}.v \circ \dots \circ u.e^{\varepsilon_{n}}_{n}.v= e.\iota f \circ \tau e. f \circ e^{-1}.\tau f \circ \iota e. f^{-1},$$
consequently $n=4$ and 
\begin{align*}
u. e^{\varepsilon_{1}}_{1}.v&=e.\iota f\\
u. e^{\varepsilon_{2}}_{2}.v&=\tau e. f\\
u. e^{\varepsilon_{3}}_{3}.v&=e^{-1}.\tau f\\
u. e^{\varepsilon_{4}}_{4}.v&=\iota e. f^{-1}.
\end{align*}
The first two equation imply that there are positive edges $e'$ and $f'$ such that $e=u.e'$ and $f=f'.v$, hence 
\begin{align*}
u. e^{\varepsilon_{1}}_{1}.v&=u.e'.\iota f'.v\\
u. e^{\varepsilon_{2}}_{2}.v&=u.\tau e'. f'.v\\
u. e^{\varepsilon_{3}}_{3}.v&=u.e'^{-1}.\tau f'.v\\
u. e^{\varepsilon_{4}}_{4}.v&=u.\iota e'. f'^{-1}.v.
\end{align*}
It follows from these that $\delta p=\delta[e',f']$ which contradicts the assumption on $\mathbf{p}$ made in section \ref{gc}.
\end{proof}

We conclude this section by stating that the attachments to $\mathcal{D}(\cP)$ of 2-cells $[w,\sigma,w']$ with $\sigma$ a 2-cell obtained by resolving a critical pair from $\mathfrak{p}$ is just a special case of the attachments of general 2-cells discussed is section \ref{gc}.

\section{Adding 3-cells to $\mathcal{D}(\cP) \sqcup F. \mathbf{p}.F$} \label{3c}

The idea of extending the 2-complex $\mathcal{D}(\mathcal{P}) \sqcup F.\mathbf{p}.F$ of section \ref{gc} to a 3-complex $(\mathcal{D}(\mathcal{P}),\mathbf{p})$ by attaching 3-cells is adapted from \cite{EP-phd} (see also \cite{EP}); however, we have modified and streamlined several proofs to enhance clarity and rigour.

We write for short $\mathcal{D}(\mathcal{P}) \sqcup F.\mathbf{p}.F=\Delta$. Under the assumption that the presentation $\mathcal{P}$ is numerable, it follows that $\Delta$ is a numerable CW-complex, hence from Theorem \ref{numcw} the cartesian product $\Delta \times \Delta$ is a CW-complex whose cells are pairs $(a_{1},a_{2})$ where $a_{1},a_{2}$ are cells in $\Delta$. The characteristic map $\Phi_{(a_{1}, a_{2})}$ is the product $\Phi_{a_{1}} \times \Phi_{a_{2}}$ of the corresponding characteristic maps.

Let $\mathcal{K}_{3}$ be the 3-skeleton of $\Delta \times \Delta$ which is made of pairs $(x_{1}, x_{2})$ such that there are cells $c_{1}, c_{2} \in \Delta$ such that $\text{dim}c_{1}+\text{dim}c_{2}\leq 3$ and $x_{1} \in c_{1}$, $x_{2} \in c_{2}$. We will define an equivalence relation $\sim$ on $\mathcal{K}_{3}$ and prove that $\sim$ is a cellular equivalence with the property that $\mathcal{K}_{3}/\sim$ is a CW-complex of dimension 3 whose 2-skeleton is $\Delta$. To achieve this we need to do some preparatory work. We introduce first the following notation. We define for every $\sigma_{1}, \sigma_{2} \in \Delta$ with $0\leq \text{dim}\sigma_{1}+ \text{dim}\sigma_{2}\leq 2$
$$\sigma_{1} \otimes \sigma_{2}=\begin{cases}
	uv & \text{ if } \sigma_{1}=u \in F, \sigma_{2}=v \in F \\
	u. \sigma_{2} & \text{ if }  \sigma_{1}=u\in F, \text{dim}\sigma_{2}\geq 1\\
	\sigma_{1}.v & \text{ if } \text{dim}\sigma_{1}\geq 1, \sigma_{2}=v \in F\\
	[e,f] & \text{ if } \sigma_{1}=e, \sigma_{2}=f
\end{cases}$$

\begin{lemma} \label{ascotens}
For every $\sigma_{1}, \sigma_{2}, \sigma_{3} \in \mathcal{D}(\cP) \sqcup F. \mathbf{p}.F$ such that $\text{dim}\sigma_{1}+\text{dim}\sigma_{2}+\text{dim}\sigma_{3} \leq 2$ we have that $(\sigma_{1}\otimes \sigma_{2})\otimes \sigma_{3}= \sigma_{1}\otimes (\sigma_{2} \otimes \sigma_{3})$. In particular, for every $\sigma_{1}, \sigma_{2} \in \mathcal{D}(\cP) \sqcup F. \mathbf{p}.F$ such that $\text{dim}\sigma_{1}+\text{dim}\sigma_{2} \leq 2$ and for every $u,v \in F$, $(\sigma_{1} \otimes u) \otimes (v \otimes \sigma_{2})= \sigma_{1} \otimes (uv) \otimes \sigma_{2}$.
\end{lemma}
\begin{proof}
If each $\sigma_{i} \in F$, then the equality comes from the associativity of the product in $F$. When only one of the $\sigma_{i}$'s is of dimension 1 and the remaining two are of dimension 0, the proof is easy. We will present here only one the the three cases, when $\sigma_{2}=e$ is a 1-cell and $\sigma_{1}=u, \sigma_{3}=v \in F$. In this case
\begin{align*}
(\sigma_{1}\otimes \sigma_{2})\otimes \sigma_{3}&= (u \otimes e) \otimes v\\
&= (u.e).v\\
&= u.(e.v)\\
&= u \otimes (e \otimes v)\\
&=\sigma_{1}\otimes (\sigma_{2} \otimes \sigma_{3}).
\end{align*}
Now we check the case when $\text{dim}\sigma_{1}+\text{dim}\sigma_{2}+\text{dim}\sigma_{3} = 2$. There are six cases possible here. In the first three cases, only one of the cells is of dimension 2, and the remaining two are 0-cells. In the rest of the cases, two of the cells are of dimension 1 and the other one is a 0-cell. For convenience we will give here the proofs for one case in each group. Assume first that $\sigma_{2}$ is any 2-cell and $\sigma_{1}=u, \sigma_{3}=v \in F$, then
\begin{align*}
(\sigma_{1}\otimes \sigma_{2})\otimes \sigma_{3}&= (u \otimes \sigma_{2})\otimes v\\
&= (u.\sigma_{2}).v\\
&= u. (\sigma_{2}.v)\\
&=\sigma_{1}\otimes (\sigma_{2} \otimes \sigma_{3}).
\end{align*}
Finally let us consider the case when $\sigma_{1}=e$, $\sigma_{2} = u \in F$ and $\sigma_{3}=f$. In this case we have
\begin{align*}
(\sigma_{1}\otimes \sigma_{2})\otimes \sigma_{3}&= (e \otimes u)\otimes f\\
&= e.u \otimes f\\
&= [e.u, f]\\
&= [e,u.f] \tag{Lemma \ref{unamb}}\\
&= e \otimes (u \otimes f)\\
&=\sigma_{1}\otimes (\sigma_{2} \otimes \sigma_{3}).
\end{align*}
The second assertion of the Lemma follows easily from the first one. Indeed,
\begin{align*}
(\sigma_{1} \otimes u) \otimes (v \otimes \sigma_{2})&= \sigma_{1} \otimes (u \otimes (v \otimes \sigma_{2}))\\
&= \sigma_{1} \otimes ((u \otimes v) \otimes \sigma_{2})\\
&= \sigma_{1} \otimes (uv \otimes \sigma_{2})\\
&=\sigma_{1} \otimes (uv) \otimes \sigma_{2}.
\end{align*}
\end{proof}

\begin{definition} \label{prl}
We say that two cells $(a_{1},a_{2}),  (b_{1},b_{2}) \in \mathcal{K}_{3}$ are parallel, denoted by $(a_{1},a_{2}) \parallel  (b_{1},b_{2})$, if and only if there are cells $\sigma_{1}, \sigma_{2}, \sigma_{3} \in \Delta$ such that $0\leq \text{dim}\sigma_{1}+ \text{dim}\sigma_{2}\leq 2$,  $0\leq \text{dim}\sigma_{2}+ \text{dim}\sigma_{3}\leq 2$, $0\leq \text{dim}\sigma_{1}+ \text{dim}\sigma_{2}+ \text{dim}\sigma_{3}\leq 3$ and
$$(a_{1},a_{2})=(\sigma_{1}\otimes \sigma_{2}, \sigma_{3}) \text{ and } (b_{1}, b_{2})=(\sigma_{1}, \sigma_{2} \otimes \sigma_{3}),$$
or
$$(a_{1}, a_{2})=(\sigma_{1}, \sigma_{2} \otimes \sigma_{3}) \text{ and } (b_{1},b_{2})=(\sigma_{1}\otimes \sigma_{2}, \sigma_{3}).$$
\end{definition}
\begin{remark} \label{d3}
Since our main interest is the nature of the 3-cells that we are adding to the complex, it is useful to describe parallel 3-cells of $\mathcal{K}_{3}$. If $(f, \sigma')$ is a 3-cell with $f$ a positive edge and $\sigma'$ a 2-cell which is expressed as $\sigma'=v.\sigma$ where $v \in F$ and $\sigma$ a 2-cell, then $(f,v.\sigma)\parallel (f.v, \sigma)$. Similarly, $(\sigma.v, f) \parallel (\sigma, v.f)$ for every positive edge $f$ and every 2-cell $\sigma$. Another instance of parallel cells of dimension 3 is when each $\sigma_{i}$ of Definition \ref{prl} is a positive edge. So $([e,f], g) \parallel (e, [f,g])$ for every $e,f,g$ positive edges.
\end{remark}
\begin{remark} \label{d2}
If $(a_{1},a_{2})$ and $(b_{1},b_{2})$ are parallel cells of dimension at most 2, then it follows from Lemma \ref{ascotens} that $a_{1}\otimes a_{2}=b_{1} \otimes b_{2}$.
\end{remark}
\begin{lemma} \label{pl}
The relation $\parallel$ on the cell decomposition $E(\mathcal{K}_{3})$ of $\mathcal{K}_{3}$ is an equivalence relation.
\end{lemma}
\begin{proof}
To prove that $\parallel$ is reflexive, we let $(\sigma_{1}, \sigma_{3})$ be any cell in $\mathcal{K}_{3}$. We can represent the cell in two forms
$$(\sigma_{1}\otimes \lambda, \sigma_{3})=(\sigma_{1}, \sigma_{3})=(\sigma_{1}, \lambda \otimes \sigma_{3}),$$
and then taking $\sigma_{2}=\lambda$ we see that $(\sigma_{1}, \sigma_{3}) \parallel (\sigma_{1}, \sigma_{3})$. That $\parallel$ is symmetric follows directly from the definition. Now we prove transitivity. Let $(a_{1},a_{2}) \parallel (b_{1}, b_{2})$ and $(b_{1},b_{2}) \parallel (c_{1},c_{2})$. We will make the proof when $(a_{1},a_{2})=(\sigma_{1} \otimes \sigma_{2}, \sigma_{3})$ and $(b_{1}, b_{2})=(\sigma_{1}, \sigma_{2} \otimes \sigma_{3})$. The proof when $(a_{1},a_{2})$ and $(b_{1},b_{2})$ switch roles is similar and is omitted. There are three possibilities for $(b_{1},b_{2})$ and $(c_{1},c_{2})$ to be parallel.
\begin{itemize}
\item[(i)] $(b_{1},b_{2})=(c_{1},c_{2})$. In this case the transitivity is obvious.
\item[(ii)] $(b_{1},b_{2})\neq (c_{1},c_{2})$ where $(b_{1},b_{2})=(\alpha_{1} \otimes \alpha_{2}, \alpha_{3})$ and $(c_{1},c_{2})=(\alpha_{1} , \alpha_{2} \otimes  \alpha_{3})$ with $\alpha_{2}\neq \lambda$. It follows that $\sigma_{1}=\alpha_{1} \otimes \alpha_{2}$ and $\alpha_{3}=\sigma_{2} \otimes \sigma_{3}$. Since $\text{dim}(\sigma_{1} \otimes \sigma_{2}) \leq 2$, it follows that $\text{dim}((\alpha_{1} \otimes \alpha_{2})\otimes \sigma_{2}) \leq 2$ consequently 
\begin{equation} \label{asoc1}
(\alpha_{1} \otimes \alpha_{2})\otimes \sigma_{2}= \alpha_{1} \otimes (\alpha_{2} \otimes \sigma_{2}).
\end{equation}
Also, since $\text{dim}(\alpha_{2} \otimes \alpha_{3})\leq 2$, it follows that $\text{dim}(\alpha_{2} \otimes (\sigma_{2} \otimes \sigma_{3}))\leq 2$ and then
\begin{equation} \label{asoc2}
(\alpha_{2} \otimes \sigma_{2})\otimes \sigma_{3}= \alpha_{2} \otimes (\sigma_{2} \otimes \sigma_{3}).
\end{equation}
Then we have
\begin{align*}
(a_{1},a_{2})&=(\sigma_{1} \otimes \sigma_{2}, \sigma_{3})\\
&= ((\alpha_{1} \otimes \alpha_{2}) \otimes \sigma_{2}, \sigma_{3}) \\
& = (\alpha_{1} \otimes (\alpha_{2} \otimes \sigma_{2}) , \sigma_{3}) \tag{by (\ref{asoc1})}\\
& \parallel (\alpha_{1}, (\alpha_{2} \otimes \sigma_{2})  \otimes \sigma_{3}) \\
&= (\alpha_{1}, \alpha_{2} \otimes (\sigma_{2}  \otimes \sigma_{3}) \tag{by (\ref{asoc2})}\\
&= (\alpha_{1}   ,\alpha_{2} \otimes \alpha_{3}) \\
&=(c_{1}, c_{2}).
\end{align*}
\item[(iii)] $(b_{1},b_{2})\neq (c_{1},c_{2})$ and $(b_{1},b_{2})=(\alpha_{1}, \alpha_{2} \otimes \alpha_{3})$, hence $(c_{1},c_{2})=(\alpha_{1} \otimes \alpha_{2}, \alpha_{3})$. It follows that $\alpha_{1}=\sigma_{1}$ and that $\sigma_{2} \otimes \sigma_{3}=\alpha_{2} \otimes \alpha_{3}$. There are three cases for $\text{dim}(\sigma_{2} \otimes \sigma_{3})$.
\begin{itemize}
\item[(1)] $\text{dim}(\sigma_{2} \otimes \sigma_{3})=0$. In this case $\sigma_{2}=u_{2}, \sigma_{3}=u_{3}, \alpha_{2}=v_{2}, \alpha_{3}=v_{3} \in F$. So we have
$$u_{2}u_{3}=v_{2}v_{3}.$$ If $u_{2}$ is a prefix of $v_{2}$, then $v_{2}=u_{2}w$ for some $w \in F$. It follows from this that $u_{3}=wv_{3}$. Now we see that
\begin{align*}
(c_{1},c_{2})&=(\alpha_{1} \otimes \alpha_{2}, \alpha_{3})\\
&= (\sigma_{1} \otimes u_{2}\otimes w, v_{3})\\
& \parallel (\sigma_{1} \otimes u_{2}, w \otimes v_{3})\\
&= (\sigma_{1} \otimes \sigma_{2}, \sigma_{3})\\
&= (a_{1}, a_{2}).
\end{align*}
If $v_{2}$ is a prefix of $u_{2}$, then for some $w \in F$ we have that $u_{2}=v_{2}w$ and $wu_{3}=v_{3}$. It follows that
\begin{align*}
(a_{1}, a_{2})&=(\sigma_{1} \otimes \sigma_{2}, \sigma_{3})\\
&=(\alpha_{1} \otimes v_{2} \otimes w, u_{3})\\
& \parallel (\alpha_{1} \otimes \alpha_{2}, wu_{3})\\
&= (\alpha_{1} \otimes \alpha_{2}, \alpha_{3})\\
&= (c_{1},c_{2}).
\end{align*}
\item[(2)] $\text{dim}(\sigma_{2} \otimes \sigma_{3})=1$. Assume that $\text{dim}\sigma_{2}=\text{dim}\alpha_{2}=1$, then $\sigma_{2}=u_{2}.e.u'_{2}$ where $e$ is a positive edge which is left and right reduced and $u_{2}, u'_{2} \in F$, and similarly $\alpha_{2}=v_{2}.f. v'_{2}$ where $f$ is a left and right reduced positive edge and $v_{2},v'_{2} \in F$. Also we have that $\sigma_{3}=u_{3} \in F$ and $\alpha_{3}=v_{3} \in F$. Our condition now writes
$$u_{2}.e.u'_{2} u_{3}= v_{2}.f.v'_{2}v_{3}$$
which implies that $u_{2}=v_{2}$, $e=f$ and that $u'_{2}u_{3}=v'_{2}v_{3}$. If $v'_{2}$ is a prefix of $u'_{2}$ then for some $w \in F$, $u'_{2}=v'_{2}w$ and $wu_{3}=v_{3}$. We then have
\begin{align*}
(a_{1}, a_{2})&=(\sigma_{1} \otimes \sigma_{2}, \sigma_{3})\\
&= (\alpha_{1} \otimes u_{2}.e.u'_{2}, u_{3})\\
&= (\alpha_{1} \otimes u_{2}.e.v'_{2}w, u_{3})\\
& \parallel (\alpha_{1} \otimes u_{2}.e.v'_{2}, wu_{3})\\
&= (\alpha_{1} \otimes v_{2}.f.v'_{2}, wu_{3})\\
&= (\alpha_{1} \otimes \alpha_{2}, \alpha_{3})\\
&=(c_{1},c_{2}).
\end{align*}
If $u'_{2}$ is a prefix of $v'_{2}$, then for some $w \in F$, $u'_{2}w=v'_{2}$ and $wv_{3}=u_{3}$. We have that
\begin{align*}
(c_{1},c_{2})&= (\alpha_{1} \otimes \alpha_{2}, \alpha_{3})\\
&= (\sigma_{1} \otimes v_{2}.f.v'_{2}, v_{3})\\
&= (\sigma_{1} \otimes v_{2}.f.u'_{2}w, v_{3})\\
&\parallel (\sigma_{1} \otimes v_{2}.f.u'_{2}, wv_{3})\\
&= (\sigma_{1} \otimes u_{2}.e.u'_{2}, wv_{3})\\
&= (\sigma_{1} \otimes \sigma_{2}, \sigma_{3})\\
&= (a_{1},a_{2}).
\end{align*}
Assume now that $\text{dim}\sigma_{2}=\text{dim}\alpha_{3}=1$, hence $\alpha_{2}, \sigma_{3} \in F$. We can now write $\sigma_{2}=u_{2}.e.u'_{2}$, $\sigma_{3}=u_{3} \in F$ and $\alpha_{3}=v_{3}.f.v'_{3}$. Here edges $e,f$ are left and right reduced. With these notations we have
$$u_{2}.e.u'_{2}u_{3}=\alpha_{2} v_{3}.f. v'_{3}.$$
It follows that $e=f$, $u'_{2}u_{3}=v'_{3}$ and $u_{2}=\alpha_{2}v_{3}$, consequently
\begin{align*}
(a_{1}, a_{2})&=(\sigma_{1} \otimes \sigma_{2}, \sigma_{3})\\
&= (\alpha_{1} \otimes u_{2}.e.u'_{2}, u_{3})\\
&= (\alpha_{1} \otimes \alpha_{2} v_{3} . f. u'_{2}, u_{3})\\
&\parallel (\alpha_{1} \otimes \alpha_{2}, v_{3}. f. u'_{2}u_{3})\\
&= (\alpha_{1} \otimes \alpha_{2}, v_{3}. f. v'_{3})\\
&= (\alpha_{1} \otimes \alpha_{2}, \alpha_{3})\\
&=(c_{1},c_{2}).
\end{align*}
Assume that $\text{dim}\sigma_{3}=\text{dim}\alpha_{2}=1$, hence $\alpha_{3}, \sigma_{2} \in F$. Let $\sigma_{3}=u_{3}.e.u'_{3}$ and $\alpha_{2}=v_{2}.f.v'_{2}$ with $e,f$ being left and right reduced. So in this case we have that
$$\sigma_{2}u_{3}.e.u'_{3}=v_{2}.f.v'_{2}\alpha_{3},$$
from which it follows that
$$\sigma_{2}u_{3}=v_{2}, u'_{3}=v'_{2}\alpha_{3}, e=f.$$ 
Now we have that
\begin{align*}
(c_{1},c_{2})&=(\alpha_{1} \otimes \alpha_{2}, \alpha_{3})\\
&= (\sigma_{1} \otimes v_{2}.f.v'_{2}, \alpha_{3})\\
&= (\sigma_{1} \otimes \sigma_{2} u_{3}.e. v'_{2}, \alpha_{3})\\
& \parallel (\sigma_{1} \otimes \sigma_{2} , u_{3}.e. v'_{2}\alpha_{3})\\
&= (\sigma_{1} \otimes \sigma_{2} , u_{3}.e. u'_{3})\\
&= (\sigma_{1} \otimes \sigma_{2} , \sigma_{3})\\
&= (a_{1},a_{2}).
\end{align*}
Finally assume that $\text{dim}\sigma_{3}=\text{dim}\alpha_{3}=1$, hence $\alpha_{2}, \sigma_{2} \in F$. Let $\sigma_{3}= u_{3}.e.u'_{3}$ and $\alpha_{3}=v_{3}.f.v'_{3}$ where $e$ and $f$ are left and right reduced, then
$$\sigma_{2}u_{3}.e.u'_{3}= \alpha_{2}v_{3}.f.v'_{3},$$
which implies that
$$e=f, \sigma_{2}u_{3}=\alpha_{2}v_{3}, u'_{3}=v'_{3}.$$
We check here the case when $\alpha_{2}$ is a prefix of $\sigma_{2}$ because the dual case has a similar proof. Under this assumption it follows that for some $w \in F$, $\sigma_{2}=\alpha_{2}w$ and $wu_{3}=v_{3}$. These imply that
\begin{align*}
(a_{1}, a_{2})&=(\sigma_{1} \otimes \sigma_{2}, \sigma_{3})\\
&= (\alpha_{1} \otimes \alpha_{2}w, \sigma_{3})\\
& \parallel (\alpha_{1} \otimes \alpha_{2}, w u_{3}.e.u'_{3})\\
&= (\alpha_{1} \otimes \alpha_{2}, v_{3}.f.v'_{3})\\
&= (\alpha_{1} \otimes \alpha_{2}, \alpha_{3})\\
&= (c_{1},c_{2}).
\end{align*}
\item[(3)] $\text{dim}(\sigma_{2} \otimes \sigma_{3})=2=\text{dim}(\alpha_{2} \otimes \alpha_{3})$. Assume in the first case that 
$$\sigma_{2}=u_{2}.e_{2}.u'_{2}, \sigma_{3}=u_{3}.e_{3}.u'_{3}, \alpha_{2}=v_{2}.f_{2}.v'_{2}, \alpha_{3}=v_{3}.f_{3}.v'_{3},$$
where as before the edges here are left and right reduced. Since $\sigma_{2} \otimes \sigma_{3}=\alpha_{2} \otimes \alpha_{3}$, then
$$u_{2}.e_{2}.u'_{2} \otimes u_{3}.e_{3}.u'_{3}=v_{2}.f_{2}.v'_{2}\otimes v_{3}.f_{3}.v'_{3}.$$
Lemma \ref{uf} implies that
$$e_{2}=f_{2}, e_{3}=f_{3}, u_{2}=v_{2}, u'_{3}=v'_{3}, u'_{2}u_{3}=v'_{2}v_{3}.$$
If $v'_{2}$ is a prefix of $u'_{2}$, then for some $w \in F$, $u'_{2}=v'_{2}w$ and $wu_{3}=v_{3}$. Now we have that
\begin{align*}
(a_{1}, a_{2})&=(\sigma_{1} \otimes \sigma_{2}, \sigma_{3})\\
&= (\alpha_{1} \otimes u_{2}.e_{2}.u'_{2}, u_{3}.e_{3}.u'_{3})\\
&= (\alpha_{1} \otimes v_{2}.f_{2}.v'_{2}w, u_{3}.e_{3}.u'_{3})\\
& \parallel (\alpha_{1} \otimes v_{2}.f_{2}.v'_{2}, wu_{3}.e_{3}.u'_{3})\\
&= (\alpha_{1} \otimes \alpha_{2}, \alpha_{3})\\
&=(c_{1},c_{2}).
\end{align*}
In the case when $u'_{2}$ is a prefix of $v'_{2}$, then $v'_{2}=u'_{2}w$ for some $w \in F$, and $wv_{3}=u_{3}$. It follows that
\begin{align*}
(c_{1},c_{2})&=(\alpha_{1} \otimes \alpha_{2}, \alpha_{3})\\
&= (\sigma_{1} \otimes v_{2}.f_{2}.v'_{2}, v_{3}.f_{3}.v'_{3})\\
&= (\sigma_{1} \otimes u_{2}.e_{2}.u'_{2}w, v_{3}.f_{3}.v'_{3})\\
& \parallel (\sigma_{1} \otimes u_{2}.e_{2}.u'_{2}, wv_{3}.f_{3}.v'_{3})\\
&= (\sigma_{1} \otimes \sigma_{2}, \sigma_{3})\\
&=(a_{1},a_{2}).
\end{align*}
Now we see the case when $\text{dim}\sigma_{2}=\text{dim}\sigma_{3}=1$, $\text{dim}\alpha_{2}=0$, and $\text{dim}\alpha_{3}=2$. Under these conditions, due to Lemma \ref{sqneqp} we may assume that 
$$\sigma_{2}=u_{2}.e_{2}.u'_{2}, \sigma_{3}=u_{3}.e_{3}.u'_{3}, \alpha_{3}= v_{3}.f_{2}.v'_{3} \otimes w_{3}.f_{3}. w'_{3},$$
where all edges here are left and right reduced. It follows that
$$u_{2}.e_{2}.u'_{2} \otimes u_{3}.e_{3}.u'_{3}= \alpha_{2} \otimes v_{3}.f_{2}.v'_{3} \otimes w_{3}.f_{3}. w'_{3},$$
therefore from Lemma \ref{uf} we have
$$e_{2}=f_{2}, e_{3}=f_{3}, u_{2}=\alpha_{2}v_{3}, u'_{2}u_{3}=v'_{3}w_{3}, u'_{3}=w'_{3}.$$
Now we see that
\begin{align*}
(a_{1},a_{2})&=(\sigma_{1} \otimes \sigma_{2}, \sigma_{3})\\
&= (\alpha_{1} \otimes u_{2}.e_{2}.u'_{2}, u_{3}.e_{3}.u'_{3})\\
&= (\alpha_{1} \otimes \alpha_{2}v_{3}.f_{2}.u'_{2}, u_{3}.e_{3}.u'_{3})\\
& \parallel (\alpha_{1} \otimes \alpha_{2}, v_{3}f_{2}.u'_{2} \otimes u_{3}.e_{3}.u'_{3})\\
&= (\alpha_{1} \otimes \alpha_{2}, v_{3}f_{2}.v'_{3} \otimes w_{3}.f_{3}.w'_{3})\\
&= (\alpha_{1} \otimes \alpha_{2}, \alpha_{3})\\
&=(c_{1},c_{2}).
\end{align*}
The case when $\text{dim}\alpha_{2}=\text{dim}\alpha_{3}=1$, $\text{dim}\sigma_{2}=0$, $\text{dim}\sigma_{3}=2$ is dual to the above and is omitted. 

If $\text{dim}\sigma_{2}=\text{dim}\alpha_{2}=2$ where both are square 2-cells, and $\sigma_{3}, \alpha_{3} \in F$, then 
$$\sigma_{2}= u_{2}.e_{2}.u'_{2} \otimes u_{3}.e_{3}.u'_{3}, \alpha_{2}=v_{2}.f_{2}.v'_{2} \otimes v_{3}.f_{3}.v'_{3},$$
where the 1-cells here are as always left and right reduced. Therefore we have
$$u_{2}.e_{2}.u'_{2} \otimes u_{3}.e_{3}.u'_{3} \otimes \sigma_{3}= v_{2}.f_{2}.v'_{2} \otimes v_{3}.f_{3}.v'_{3} \otimes \alpha_{3},$$
from which we get 
$$u_{2}=v_{2}, e_{2}=f_{2}, u'_{2}u_{3}=v'_{2}v_{3}, e_{3}=f_{3}, u'_{3}\sigma_{3}=v'_{3}\alpha_{3}.$$
If $v'_{3}$ is a prefix of $u'_{3}$, then for some $w \in F$, $u'_{3}=v'_{3}w$ and $w \sigma_{3}=\alpha_{3}$. It follows from these that
\begin{align*}
(a_{1},a_{2})&=(\sigma_{1} \otimes \sigma_{2}, \sigma_{3})\\
&= (\alpha_{1} \otimes u_{2}.e_{2}.u'_{2} \otimes u_{3}.e_{3}.u'_{3}, \sigma_{3})\\
&= (\alpha_{1} \otimes v_{2}.f_{2}.v'_{2} \otimes v_{3}.f_{3}.u'_{3}, \sigma_{3})\\
&= (\alpha_{1} \otimes v_{2}.f_{2}.v'_{2} \otimes v_{3}.f_{3}.v'_{3}w, \sigma_{3})\\
& \parallel (\alpha_{1} \otimes v_{2}.f_{2}.v'_{2} \otimes v_{3}.f_{3}.v'_{3}, w\sigma_{3})\\
&= (\alpha_{1} \otimes \alpha_{2}, \alpha_{3})\\
&=(c_{1},c_{2}).
\end{align*}
If $u'_{3}$ is a prefix of $v'_{3}$, then for some $w \in F$, $v'_{3}=u'_{3}w$ and $w \alpha_{3}=\sigma_{3}$. In this case we have that
\begin{align*}
(c_{1},c_{2})&= (\alpha_{1} \otimes \alpha_{2}, \alpha_{3})\\
&= (\alpha_{1} \otimes v_{2}.f_{2}.v'_{2} \otimes v_{3}.f_{3}.v'_{3}, \alpha_{3})\\
&= (\sigma_{1} \otimes u_{2}.e_{2}.u'_{2} \otimes u_{3}.e_{3}.u'_{3}w, \alpha_{3} )\\
& \parallel (\sigma_{1} \otimes u_{2}.e_{2}.u'_{2} \otimes u_{3}.e_{3}.u'_{3}, w\alpha_{3} )\\
&=(\sigma_{1} \otimes \sigma_{2}, \sigma_{3})\\
&=(a_{1},a_{2}).
\end{align*}
The dual case to the above is when $\text{dim}\sigma_{3}=\text{dim}\alpha_{3}=2$ and $\sigma_{2}, \alpha_{2} \in F$ and it is omitted. Now we see the case when $\text{dim}\sigma_{2}=2=\text{dim}\alpha_{3}$ where the 2-cells are square and $\alpha_{2}, \sigma_{3} \in F$. Let us assume that
$$\sigma_{2}=u_{2}.e_{2}.u'_{2} \otimes u_{2}.e_{3}.e'_{3}, \alpha_{3}= v_{2}.f_{2}.v'_{2} \otimes v_{3}.f_{3}.v'_{3},$$
with all edges being left and right reduced, then
$$u_{2}.e_{2}.u'_{2} \otimes u_{2}.e_{3}.e'_{3} \otimes \sigma_{3}=\alpha_{2} \otimes v_{2}.f_{2}.v'_{2} \otimes v_{3}.f_{3}.v'_{3},$$
which implies that
$$u_{2}=\alpha_{2}v_{2}, e_{2}=f_{2}, u'_{2}\otimes u_{3}=v'_{2} \otimes v_{3}, e_{3}=f_{3}, u'_{3}\otimes \sigma_{3}=v'_{3}.$$
From this we obtain
\begin{align*}
(a_{1},a_{2})&= (\sigma_{1} \otimes \sigma_{2}, \sigma_{3})\\
&=(\alpha_{1} \otimes u_{2}.e_{2}.u'_{2} \otimes u_{3}.e_{3}.u'_{3}, \sigma_{3})\\
&= (\alpha_{1} \otimes \alpha_{2} \otimes v_{2}. f_{2}. v'_{2} \otimes v_{3}.f_{3}.u'_{3}, \sigma_{3})\\
& \parallel (\alpha_{1} \otimes \alpha_{2}, v_{2}. f_{2}. v'_{2} \otimes v_{3}.f_{3}.u'_{3} \otimes \sigma_{3})\\
&= (\alpha_{1} \otimes \alpha_{2}, v_{2}. f_{2}. v'_{2} \otimes v_{3}.f_{3}.v_{3} )\\
&= (\alpha_{1} \otimes \alpha_{2}, \alpha_{3})\\
&=(c_{1},c_{2}).
\end{align*}
The case when $\text{dim}\alpha_{2}=2=\text{dim}\sigma_{3}$ where the 2-cells are square and $\sigma_{2}, \alpha_{3} \in F$, is dual to the above.

Assume now that $\sigma_{2}=u.\sigma.v\in F. \mathbf{p}.F$ and $\sigma_{3}\in F$. It follows from Lemma \ref{sqneqp} that $\alpha_{2} \otimes \alpha_{3}$ cannot be square. If $\alpha_{2}=u'.\sigma'.v' \in F. \mathbf{p}.F$, then $ \alpha_{3} \in F$ and
$$u.\sigma.v \sigma_{3}=\sigma_{2} \otimes \sigma_{3}=\alpha_{2} \otimes \alpha_{3}= u'.\sigma'.v'  \alpha_{3}.$$ 
Lemma \ref{unambp} and more generally (\ref{unqexp}) imply that
$$u=u', \sigma=\sigma', v\sigma_{3}=v'\alpha_{3}.$$
If $v'$ is a prefix of $v$, then $v=v'w$ for some $w \in F$ and $w \sigma_{3}=\alpha_{3}$. It follows that
\begin{align*}
(a_{1},a_{2})&= (\sigma_{1} \otimes \sigma_{2}, \sigma_{3})\\
&= (\alpha_{1} \otimes \sigma_{2}, \sigma_{3})\\
&= (\alpha_{1} \otimes u.\sigma. v, \sigma_{3})\\
&=(\alpha_{1} \otimes u'.\sigma'. v, \sigma_{3})\\
&=(\alpha_{1} \otimes u'.\sigma'. v'w, \sigma_{3})\\
& \parallel (\alpha_{1} \otimes u'.\sigma'. v', w\sigma_{3})\\
&= (\alpha_{1} \otimes \alpha_{2}, \alpha_{3})\\
&=(c_{1},c_{2}).
\end{align*}
If $v$ is a prefix of $v'$, then $v'=vw$ for some $w \in F$ and $w \alpha_{3}=\sigma_{3}$. It follows that
\begin{align*}
(c_{1},c_{2})&=  (\alpha_{1} \otimes \alpha_{2}, \alpha_{3})\\
&= (\sigma_{1} \otimes u'.\sigma'.v', \alpha_{3})\\
&= (\sigma_{1} \otimes u.\sigma.vw, \alpha_{3})\\
& \parallel (\sigma_{1} \otimes u.\sigma.v, w\alpha_{3})\\
&= (\sigma_{1} \otimes \sigma_{2}, \sigma_{3})\\
&= (a_{1},a_{2}).
\end{align*}
The other possible case is when $\sigma_{2}=u.\sigma. v$, $\sigma_{3} \in F$, $\alpha_{2} \in F$ and $\alpha_{3}=u'.\sigma'.v'$. It follows that
$$u=\alpha_{2}u', \sigma=\sigma', v\sigma_{3}=v'.$$
Therefore we have
\begin{align*}
(a_{1},a_{2})&= (\sigma_{1} \otimes \sigma_{2}, \sigma_{3})\\
&= (\alpha_{1} \otimes u.\sigma.v , \sigma_{3})\\
&= (\alpha_{1} \otimes \alpha_{2}u'.\sigma'.v , \sigma_{3})\\
& \parallel (\alpha_{1} \otimes \alpha_{2}, u'\sigma'. v \sigma_{3})\\
&=(\alpha_{1} \otimes \alpha_{2}, u'\sigma'. v ')\\
&= (\alpha_{1} \otimes \alpha_{2}, \alpha_{3})\\
&=(c_{1},c_{2}).
\end{align*}
The two remaining cases when $\sigma_{2} \in F$, $\sigma_{3} \in F.\mathbf{p}.F$ and either $\alpha_{2} \in F$ and $\alpha_{3} \in F. \mathbf{p}.F$, or $\alpha_{2} \in F. \mathbf{p}.F$ and $\alpha_{3} \in F$, are dual to the above and the corresponding proofs are omitted.
\end{itemize}

\end{itemize}
\end{proof}

From now and on, the characteristic map of a cell $a$ of $\mathcal{D}(\cP) \sqcup F. \mathbf{p}.F$ will be denoted by $\Phi_{a}$. The standard definition of the characteristic map 
$$\Phi_{(a_{1}, a_{2})}: I^{\text{dim}a_{1}} \times I^{\text{dim}a_{2}} \rightarrow \mathcal{K}_{3} $$ 
of a cell $(a_{1}, a_{2})$ in $\mathcal{K}_{3}$ is the product $\Phi_{a_{1}} \times \Phi_{a_{2}}$ where
$$\Phi_{a_{i}}: I_{a_{i}}^{\text{dim}a_{i}} \rightarrow \mathcal{D}(\cP) \sqcup F. \mathbf{p}.F \text{ with } i=1,2,$$
is the characteristic map of $a_{i}$. However, we can redefine the characteristic map by altering its domain to $I^{\text{dim}(a_{1}, a_{2})}$, which is homeomorphic to $I^{\text{dim}a_{1}} \times I^{\text{dim}a_{2}}$, but without affecting its values and the continuity in the following fashion. We first let $I^{0}=\{\ast\}$ where $\ast$ is just a symbol.
\begin{itemize}
\item[(i)] If $\text{dim} (a_{1}, a_{2})=0$, which means that both $a_{1}$ and $a_{2}$ are 0-cells, then we define
$$\Phi_{(a_{1},a_{2})}: I^{0} \rightarrow \mathcal{K}_{3} $$
by
$$\Phi_{(a_{1},a_{2})}(\ast)=(a_{1}, a_{2}).$$
\item[(ii)] If $\text{dim} (a_{1}, a_{2})=1$ there are two possible cases.
\begin{itemize}
\item[(ii-1)]
$a_{1}=e$ is a 1-cell and $a_{2} \in F$, then define
$$\Phi_{(e,a_{2})}: I^{1} \rightarrow \mathcal{K}_{3} $$
by
$$\Phi_{(e,a_{2})}(t)=(\Phi_{e}(t), a_{2}).$$
\item[(ii-2)] If $a_{1} \in F$ and $a_{2}=f$ is a 1-cell, then define
$$\Phi_{(a_{1},f)}: I^{1} \rightarrow \mathcal{K}_{3} $$
by
$$\Phi_{(a_{1},f)}(t)=(a_{1}, \Phi_{f}(t)).$$
\end{itemize}
\item[(iii)] If $\text{dim} (a_{1}, a_{2})=2$, there are three possible cases. 
\begin{itemize}
\item[(iii-1)] Both $a_{1}=e$ and $a_{2}=f$ are 1-cells. In this case we define
$$\Phi_{(e,f)}: I^{2} \rightarrow \mathcal{K}_{3} $$
by
$$\Phi_{(e,f)}((t_{1},t_{2}))=(\Phi_{e}(t_{1}), \Phi_{f}(t_{2})).$$
\item[(iii-2)] $a_{1}=\sigma$ is a 2-cell and $a_{2}$ is a 0-cell. In this case we define
$$\Phi_{(\sigma,a_{2})}: I^{2} \rightarrow \mathcal{K}_{3} $$
by
$$\Phi_{(\sigma,a_{2})}((t_{1},t_{2}))=(\Phi_{\sigma}(t_{1},t_{2}), a_{2}).$$
\item[(iii-3)] $a_{1}$ is a 0-cell and $a_{2}=\sigma$ is a 2-cell. Then we define
$$\Phi_{(a_{1},\sigma)}: I^{2} \rightarrow \mathcal{K}_{3} $$
by
$$\Phi_{(a_{1},\sigma)}((t_{1},t_{2}))=(a_{1},\Phi_{\sigma}(t_{1},t_{2})).$$
\end{itemize}
\item[(iv)] If $\text{dim}(a_{1},a_{2})=3$. There are two possible cases here.
\begin{itemize}
\item[(iv-1)] $a_{1}=f$ is a 1-cell and $a_{2}=\sigma$ is a 2-cell. In this case we define
$$\Phi_{(f, \sigma)}: I^{3} \rightarrow \mathcal{K}_{3}$$
by
$$\Phi_{(f, \sigma)}(t_{1},t_{2},t_{3})=(\Phi_{f}(t_{1}), \Phi_{\sigma}(t_{2},t_{3})).$$
\item[(iv-2)] $a_{1}=\sigma$ is a 2-cell and $a_{2}=f$ is a 1-cell. In this case we define
$$\Phi_{(\sigma,f)}: I^{3} \rightarrow \mathcal{K}_{3}$$
by
$$\Phi_{(\sigma, f)}(t_{1},t_{2},t_{3})=(\Phi_{\sigma}(t_{1},t_{2}), \Phi_{f}(t_{3})).$$
\end{itemize}
\end{itemize}
For every cell $(a_{1},a_{2}) \in \mathcal{K}_{3}$ with characteristic map $\Phi_{(a_{1},a_{2})}$ we will change the notation for $(a_{1}, a_{2})^{\circ}$ to $ \partial (a_{1}, a_{2})$ and then
$$(a_{1}, a_{2}) \setminus \partial (a_{1}, a_{2})= \Phi_{(a_{1},a_{2})} (I^{\text{dim}a_{1}+ \text{dim}a_{2}} \setminus \partial I^{\text{dim}a_{1}+ \text{dim}a_{2}})$$
will be the corresponding "open" cell. We note that if some pair $(x_{1},x_{2})$ is a point in $(a_{1}, a_{2}) \setminus \partial (a_{1}, a_{2})$, then there is a unique $u$ in the interior of $I^{\text{dim}a_{1}+ \text{dim}a_{2}}$ such that $\Phi_{(a_{1},a_{2})}(u)=(x_{1},x_{2})$. The element $u$ is written as $\Phi^{-1}_{(a_{1},a_{2})}(x_{1},x_{2})$.

\begin{definition} \label{relK}
Let $\sim$ be the relation on $\mathcal{K}_{3}$ containing all the pairs $((x_{1}, x_{2}), (y_{1},y_{2}))$ with the property that there are cells 
$(a_{1}, a_{2}), (b_{1},b_{2}) \in \mathcal{K}_{3}$ such that:
\begin{itemize}
\item[(i)] $(a_{1}, a_{2}) \parallel (b_{1},b_{2}).$
\item[(ii)] $(x_{1}, x_{2}) \in (a_{1}, a_{2}) \setminus \partial (a_{1}, a_{2})$, $(y_{1}, y_{2}) \in (b_{1},b_{2}) \setminus \partial (b_{1},b_{2})$,
\item[(iii)] $(y_{1},y_{2})=\Phi_{(b_{1},b_{2})}(\Phi^{-1}_{(a_{1},a_{2})}(x_{1},x_{2}))$.
\end{itemize}
\end{definition}

\begin{lemma}
The relation $\sim$ is an equivalence relation on $\mathcal{K}_{3}$.
\end{lemma}
\begin{proof}
The relation $\sim$ is obviously reflexive and symmetric. To prove that $\sim$ is transitive we assume that $(x_{1}, x_{2}) \sim (y_{1},y_{2})$ and $(y_{1}, y_{2}) \sim (z_{1},z_{2})$, therefore there are cells $(a_{1}, a_{2}), (b_{1}, b_{2}), (c_{1}, c_{2})$ such that
$(a_{1}, a_{2}) \parallel (b_{1},b_{2})$, $(b_{1}, b_{2}) \parallel (c_{1},c_{2})$, $(x_{1}, x_{2}) \in (a_{1}, a_{2}) \setminus \partial (a_{1}, a_{2})$, $(y_{1}, y_{2}) \in (b_{1},b_{2}) \setminus \partial (b_{1},b_{2})$, $(z_{1}, z_{2}) \in (c_{1}, c_{2}) \setminus \partial (c_{1}, c_{2})$ and that
$$(y_{1},y_{2})=\Phi_{(b_{1},b_{2})}(\Phi^{-1}_{(a_{1},a_{2})}(x_{1},x_{2})), (z_{1},z_{2})=\Phi_{(c_{1},c_{2})}(\Phi^{-1}_{(b_{1},b_{2})}(y_{1},y_{2})).$$
From Lemma \ref{pl}, $(a_{1}, a_{2}) \parallel (c_{1},c_{2})$, so it remains to prove the condition (iii) of the definition. Indeed,
\begin{align*}
(z_{1},z_{2})&=\Phi_{(c_{1},c_{2})}(\Phi^{-1}_{(b_{1},b_{2})}(y_{1},y_{2}))\\
&=\Phi_{(c_{1},c_{2})}(\Phi^{-1}_{(b_{1},b_{2})}(\Phi_{(b_{1},b_{2})}(\Phi^{-1}_{(a_{1},a_{2})}(x_{1},x_{2})))\\
&= \Phi_{(c_{1},c_{2})}(\Phi^{-1}_{(a_{1},a_{2})}(x_{1},x_{2})).
\end{align*}
\end{proof}

\begin{lemma} \label{p-inj}
If $(x_{1},x_{2}), (y_{1},y_{2}) \in (a_{1}, a_{2}) \setminus \partial (a_{1}, a_{2})$ and $(x_{1},x_{2}) \sim (y_{1},y_{2})$, then $(x_{1},x_{2})=(y_{1},y_{2})$.
\end{lemma}
\begin{proof}
If $(c_{1}, c_{2}) \setminus \partial (c_{1}, c_{2})$ and $(d_{1}, d_{2}) \setminus \partial (d_{1}, d_{2})$ are the open cells of Definition \ref{relK} such that $(x_{1},x_{2})\in(c_{1}, c_{2}) \setminus \partial (c_{1}, c_{2})$ and $(y_{1},y_{2})\in(d_{1}, d_{2}) \setminus \partial (d_{1}, d_{2})$, then since each point in $\mathcal{D}(\cP) \sqcup F. \mathbf{p}.F$ belongs to a unique open cell, it follows that for $i=1,2$,
$$c_{i}\setminus \partial c_{i}=a_{i} \setminus \partial a_{i}= d_{i} \setminus \partial d_{i},$$
which in turn implies that,
$$c_{i}=a_{i}=d_{i}.$$
Now we see that
\begin{align*}
(y_{1},y_{2})&=\Phi_{(d_{1},d_{2})}(\Phi^{-1}_{(c_{1},c_{2})}(x_{1},x_{2}))\\
&=\Phi_{(a_{1},a_{2})}(\Phi^{-1}_{(a_{1},a_{2})}(x_{1},x_{2}))\\
&=(x_{1},x_{2}),
\end{align*}
hence $(x_{1},x_{2})=(y_{1},y_{2})$.
\end{proof}

\begin{definition}
Denote by $(\mathcal{D}(\cP), \mathbf{p})$ the quotient space $\mathcal{K}_{3}/ \sim$ and let $p: \mathcal{K}_{3} \rightarrow \mathcal{K}_{3}/\sim$ be the quotient map. For every cell $(a_{1}, a_{2}) \in \mathcal{K}_{3}$, define $p((a_{1}, a_{2}) \setminus \partial (a_{1}, a_{2}))=\{p((x_{1},x_{2}))| (x_{1}, x_{2}) \in (a_{1}, a_{2}) \setminus \partial (a_{1}, a_{2})\}$.
\end{definition}

\begin{lemma} \label{p-par}
For every two cells $(a_{1}, a_{2}),  (b_{1}, b_{2})\in \mathcal{K}_{3}$, $p((a_{1}, a_{2}) \setminus \partial (a_{1}, a_{2}))=p((b_{1}, b_{2}) \setminus\partial (b_{1}, b_{2}))$ if and only if $(a_{1}, a_{2}) \parallel (b_{1}, b_{2})$.
\end{lemma}
\begin{proof}
To prove that $(a_{1}, a_{2}) \parallel (b_{1}, b_{2})$ it is enough to prove that there are $(x_{1},x_{2}) \in (a_{1}, a_{2}) \setminus \partial (a_{1}, a_{2})$ and $(y_{1},y_{2}) \in (b_{1}, b_{2}) \setminus \partial (b_{1}, b_{2})$ such that $(x_{1},x_{2}) \sim (y_{1}, y_{2})$. For every element $(x_{1},x_{2}) \in (a_{1}, a_{2}) \setminus \partial (a_{1}, a_{2})$, since $p((x_{1},x_{2})) \in p((a_{1}, a_{2}) \setminus \partial (a_{1}, a_{2}))$, it follows from the given equality that there is some $p((y_{1},y_{2})) \in p((b_{1}, b_{2}) \setminus \partial (b_{1}, b_{2}))$ such that $p((x_{1},x_{2}))=p((y_{1},y_{2}))$, therefore $(x_{1},x_{2}) \sim (y_{1}, y_{2})$.

Conversely assume that 
$$a=(\sigma_{1}\otimes \sigma_{2}, \sigma_{3})\parallel (\sigma_{1}, \sigma_{2}\otimes \sigma_{3})=b.$$ 
The proof of the claim when $a$ and $b$ switch roles is analogous to the following and is omitted. Let $p(x_{1},x_{2}) \in p ((\sigma_{1}\otimes \sigma_{2}, \sigma_{3}) \setminus \partial (\sigma_{1}\otimes \sigma_{2}, \sigma_{3}))$ for some 
$$(x_{1},x_{2}) \in (\sigma_{1}\otimes \sigma_{2}, \sigma_{3}) \setminus \partial (\sigma_{1}\otimes \sigma_{2}, \sigma_{3})$$ 
and let 
$$u \in I^{\text{dim}a}   \setminus \partial I^{\text{dim}a}$$ 
such that $\Phi_{(\sigma_{1}\otimes \sigma_{2},\sigma_{3})}(u)=(x_{1},x_{2})$. Let 
$$(y_{1},y_{2})=\Phi_{(\sigma_{1},\sigma_{2}\otimes \sigma_{3})} (u) \in (\sigma_{1}, \sigma_{2}\otimes \sigma_{3}) \setminus \partial (\sigma_{1}, \sigma_{2}\otimes \sigma_{3}).$$ 
It follows from Definition \ref{relK} that $(x_{1},x_{2}) \sim (y_{1},y_{2})$, hence 
$$p(x_{1},x_{2})=p(y_{1},y_{2}) \in p((\sigma_{1}, \sigma_{2}\otimes \sigma_{3}) \setminus \partial (\sigma_{1}, \sigma_{2}\otimes \sigma_{3}))$$ 
which proves that 
$$p ((\sigma_{1}\otimes \sigma_{2}, \sigma_{3}) \setminus \partial (\sigma_{1}\otimes \sigma_{2}, \sigma_{3}))\subseteq p((\sigma_{1}, \sigma_{2}\otimes \sigma_{3}) \setminus \partial (\sigma_{1}, \sigma_{2}\otimes \sigma_{3})).$$ The reverse inclusion is proved in a symmetric fashion.
\end{proof}

\begin{lemma} \label{p-par-b-1}
\begin{itemize}
\item[(d-1)] For every positive edge (1-cell) $e \in \mathcal{D}(\cP) \sqcup F.\mathbf{p}.F$ and every $u \in F$, $p(u,e)=p(\lambda,u.e)$ and similarly $p(e, u)=p(e.u, \lambda)$. Also $p(e,\lambda)=p(\lambda,e)$.
\item[(d-2)] For every cell $\sigma \in \mathcal{D}(\cP) \sqcup F.\mathbf{p}.F$ of dimension 2 and every $u \in F$, $p(u,\sigma)=p(\lambda,u.\sigma)$ and similarly $p(\sigma, u)=p(\sigma.u, \lambda)$. Also $p(\sigma, \lambda)=p(\lambda,\sigma)$.
\item[(d-3)] For every pair $e,f$ of positive edges, $p(e,f)=p(\lambda,[e,f])=p([e,f],\lambda)$.
\item[(d-4)] For every pair $e,f$ of positive edges and every $u \in F$, $p(e.u,f)=p(e,u.f)$.
\item[(d-5)] For every $u,v,',v' \in F$ such that $uv=u'v'$, $p(u,v)=p(u',v')$.
\end{itemize}
\end{lemma}
\begin{proof}
First we note that for every cell $(a_{1},a_{2}) \in \mathcal{K}_{3}$,
$$p((a_{1},a_{2}) \setminus \partial (a_{1},a_{2}))=p(a_{1},a_{2}) \setminus p(\partial (a_{1},a_{2})).$$
This holds true since 
$$p((a_{1},a_{2}) \setminus \partial (a_{1},a_{2})) \cap p(\partial (a_{1},a_{2})) = \emptyset,$$
which follows from the fact that no points in the interior of a cell are identified with points from its boundary. The definition of $\sim$ allows identifications of points from open (parallel) cells of the same dimension. \\
(d-1) We prove the second claim and the third claim. Since $(e,u) \parallel (e.u,\lambda)$, then Lemma \ref{p-par} implies that $p((e,u)\setminus \partial (e,u))=p((e.u,\lambda) \setminus (e.u, \lambda))$. Let us prove that $p(\partial (e,u))=p(\partial (e.u,\lambda))$. This holds true since
\begin{align*}
p(\partial (e,u))&= p(\partial e \times u)\\
&= p\{(\iota(e),u), (\tau(e),u)\}\\
&= \{p(\iota(e),u), p(\tau(e),u)\}\\
&= \{ p(\lambda, \iota(e)u), p(\lambda, \tau(e)u)\} \tag{Lemma \ref{p-par}}\\
&=p \{ (\lambda, \iota(e)u), (\lambda, \tau(e)u)\}\\
&= p \{(\iota(e)u,\lambda), (\tau(e)u,\lambda)\} \tag{Lemma \ref{p-par}}\\
&= p (\partial (e.u,\lambda)).
\end{align*}
To check the third claim, we need to prove that $p(\partial (e,\lambda))=p(\partial (\lambda,e))$.
\begin{align*}
p(\partial (e,\lambda))&= p \{(\iota(e),\lambda), (\tau(e),\lambda)\}\\
&=  \{p(\iota(e),\lambda), p(\tau(e),\lambda)\}\\
&= \{p(\lambda, \iota(e)), p(\lambda, \tau(e))\} \tag{Lemma \ref{p-par}}\\
&= p\{(\lambda, \iota(e)), (\lambda, \tau(e))\}\\
&= p (\partial (\lambda,e)).
\end{align*}
(d-2) We prove the first claim and the third claim. Since $(u, \sigma) \parallel (\lambda,u.\sigma)$, then from Lemma \ref{p-par} we have that $p((u, \sigma) \setminus \partial (u, \sigma) )=p((\lambda,u.\sigma)\setminus \partial (\lambda,u.\sigma))$. It remains to prove that $p(\partial (u,\sigma))=p(\partial (\lambda,u.\sigma)$. This is true since
\begin{align*}
p(\partial (u,\sigma))&= p(\underset{e \in \partial \sigma}{\cup}(u,e))\\
&= \underset{e \in \partial \sigma}{\cup}p(u,e)\\
&=\underset{e \in \partial \sigma}{\cup}p(\lambda,u.e) \tag{by (d-1)}\\
&=\underset{u.e \in \partial u.\sigma}{\cup}p(\lambda,u.e)\\
&=p\left(\underset{u.e \in \partial u.\sigma}{\cup}(\lambda,u.e)\right)\\
&=p(\partial (\lambda,u.\sigma).
\end{align*}
Now we prove that $p(\partial (\sigma, \lambda))=p(\partial (\lambda, \sigma))$.
\begin{align*}
p(\partial (\sigma, \lambda))&=p ( \underset{e \in \partial \sigma}{\cup}(e,\lambda))\\
&= \underset{e \in \partial \sigma}{\cup}p(e,\lambda)\\
&=\underset{e \in \partial \sigma}{\cup}p(\lambda, e) \tag{by (d-1)}\\
&=p(\underset{e \in \partial \sigma}{\cup}(\lambda, e))\\
&= p (\partial (\lambda,\sigma)).
\end{align*}
(d-3) We prove the first equality. As before we need to check that $p(\partial (e,f))=p(\partial (\lambda,[e,f]))$. We see that
\begin{align*}
p(\partial (e,f))&=p ((\iota(e), f) \cup (\tau(e), f) \cup (e, \iota(f)) \cup (e,\tau(f)))\\
&= p (\iota(e), f) \cup p (\tau(e), f) \cup p (e, \iota(f)) \cup p (e,\tau(f)) \\
&=p (\lambda, \iota(e). f) \cup p (\lambda, \tau(e).f) \cup p (e.\iota(f),\lambda)\cup p (e.\tau(f),\lambda) \tag{by (d-1)}\\
&= p (\lambda, \iota(e). f) \cup p (\lambda, \tau(e).f) \cup p (\lambda, e.\iota(f)) \cup p (\lambda, e.\tau(f)) \tag{by (d-1)}\\
&= p( (\lambda, \iota(e). f) \cup (\lambda, \tau(e).f) \cup  (\lambda, e.\iota(f)) \cup  (\lambda, e.\tau(f)))\\
&= p (\partial (\lambda,[e,f])).
\end{align*}
(d-4) We first note that $(e.u,f) \parallel (e,u.f)$ and then Lemma \ref{p-par} implies that $p((e.u,f) \setminus \partial (e.u,f) )=p((e,u.f) \setminus \partial (e,u.f))$. Further we see that
\begin{align*}
p(\partial (e.u,f))&= p (\iota (e) u, f) \cup p(\tau (e) u, f) \cup p(e.u, \iota(f)) \cup p(e.u, \tau(f))\\
&= p(\iota(e),u.f) \cup p(\tau (e) , u.f) \cup p(e, u.\iota(f)) \cup p(e, u.\tau(f)) \tag{by (d-1)}\\
&= p(\partial (e,u.f)).
\end{align*}
(d-5) It is enough to see that $(u,v) \sim (u',v')$. Since $uv=u'v'$, then it follows easily that $(u,v) \parallel (u',v')$. The second condition of Definition \ref{relK} is obvious since 0-cells have no boundary, and the third conditions is also true since
$$\Phi_{(u',v')}(\Phi^{-1}_{(u,v)}(u,v))=\Phi_{(u',v')}(\ast)=(u',v').$$
\end{proof}

\begin{lemma} \label{p-par-b}
For every two cells $(a_{1}, a_{2}),  (b_{1}, b_{2})\in \mathcal{K}_{3}$, if $(a_{1}, a_{2}) \parallel (b_{1}, b_{2})$, then $p(a_{1}, a_{2})=p(b_{1}, b_{2})$.
\end{lemma}
\begin{proof}
Let assume that $(a_{1}, a_{2}), (b_{1}, b_{2})$ are of dimension 3. Since from Lemma \ref{p-par} we have that $p((a_{1},a_{2})\setminus \partial (a_{1}, a_{2}))=p((a_{1},a_{2})\setminus \partial (b_{1}, b_{2}))$ it remains to prove that $p(\partial (a_{1}, a_{2}))=p(\partial (b_{1}, b_{2}))$. If $(a_{1}, a_{2})=(f,u.\sigma)$ where $f$ is a positive edge, $\sigma$ a 2-cell and $u \in F$, then $(b_{1}, b_{2})=(f.u,\sigma)$. We note first that $\partial (u.\sigma)=\underset{e \in \partial \sigma}{\cup} u.e$. It follows that
\begin{align*}
p(\partial (f, u.\sigma))&= p(\partial f \times u.\sigma \cup f \times \partial (u.\sigma)) \\
&= p(\iota(f), u.\sigma) \cup p(\tau(f), u.\sigma) \cup ( \underset{e_{i} \in \partial \sigma}{\cup} p((f,u.e_{i})))\\
&=p(\lambda, \iota(f).u.\sigma)) \cup p(\lambda, \tau(f) .u.\sigma) \cup ( \underset{e_{i} \in \partial \sigma}{\cup} p(\lambda,[f,u.e_{i}])) \tag{Lemma \ref{p-par-b-1}}\\
&=p(\lambda, \iota(f.u).\sigma)) \cup p(\lambda, \tau(f.u).\sigma) \cup ( \underset{e_{i} \in \partial \sigma}{\cup} p(\lambda,[f.u,e_{i}])) \tag{Lemma \ref{unamb}}\\
&= p(\iota(f.u),\sigma)) \cup p(\tau(f.u) ,\sigma) \cup ( \underset{e_{i} \in \partial \sigma}{\cup} p((f.u,e_{i}))) \tag{Lemma \ref{p-par-b-1}} \\
&=p(\partial (f.u) \times \sigma \cup f.u \times \partial \sigma)\\
&=p(\partial (f.u,\sigma)).
\end{align*}
The second case when $(a_{1}, a_{2})=(\sigma, u.f)$ and $(b_{1},b_{2})=(\sigma.u, f)$ if proved in a similar fashion. Now we prove the case when $(a_{1},a_{2})=([e,f],g)$ and $(b_{1},b_{2})=(e,[f,g])$. Indeed, from Lemma \ref{p-par-b-1} we have that

{\small
\begin{align*}
p(\partial ([e,f]&,g))= p(e.\iota(f), g) \cup p(\tau(e).f,g) \cup p(e.\tau(f),g) \cup p(\iota(e).f,g) \cup p([e,f], \iota(g)) \cup p([e,f], \tau(g))\\
&= p(e, \iota(f).g) \cup p(\tau(e), [f,g]) \cup p(e, \tau(f).g) \cup p(\iota(e), [f,g]) \cup p([e,f.\iota(g)],\lambda) \cup p([e, f.\tau{g}], \lambda)\\
&= p(e, \iota(f).g) \cup p(e, \tau(f).g) \cup p(e,f.\iota(g)) \cup p(e, f.\tau{g}) \cup p(\tau(e), [f,g]) \cup p(\iota(e), [f,g]) \\
&= p(\partial ([e,[f,g]])).
\end{align*}
}
Now let $(a_{1}, a_{2}), (b_{1}, b_{2})$ be parallel cells of dimension $\leq 2$, therefore $a_{1}\otimes a_{2}=b_{1} \otimes b_{2}$. There are three possible cases for $(a_{1},a_{2})$.\\
(i) $(a_{1},a_{2})=(u, \alpha)$ with $u \in F$ and $\alpha$ a cell of dimension $\leq 2$. There are three cases here for $(b_{1}, b_{2})$.\\
(i-1) $(b_{1}, b_{2})=(v, \beta)$ with $v \in F$ and $\beta$ a cell of dimension $\leq 2$. It follows that
\begin{align*}
p(u,\alpha)&= p(\lambda, u.\alpha) \tag{by Lemma \ref{p-par-b-1}}\\
&= p(\lambda,v.\beta) \tag{since $u.\alpha=u \otimes a=v \otimes b=v.\beta$}\\
&=p(v, \beta). \tag{by Lemma \ref{p-par-b-1}}
\end{align*} 
(i-2) $(b_{1}, b_{2})=(\beta, v)$ with $v \in F$ and $\beta$ a cell of dimension $\leq 2$. It follows that
\begin{align*}
p(u,\alpha)&= p(\lambda, u.\alpha) \tag{by Lemma \ref{p-par-b-1}}\\
&= p(\lambda,\beta.v) \tag{since $u.\alpha=u \otimes a=b \otimes v=\beta.v$}\\
&=p(\beta,v). \tag{by Lemma \ref{p-par-b-1}}
\end{align*}
(i-3) $(b_{1}, b_{2})=(e, f)$ where $e,f$ are positive edges. In this case we have
\begin{align*}
p(u,\alpha)&= p(\lambda, u.\alpha) \tag{by Lemma \ref{p-par-b-1}}\\
&= p(\lambda,[e,f]) \tag{since $u.\alpha=u \otimes \alpha=e \otimes f=[e,f]$}\\
&=p(e,f). \tag{by Lemma \ref{p-par-b-1}}
\end{align*}
(ii) $(a_{1},a_{2})=(\alpha, u)$ with $u \in F$ and $\alpha$ a cell of dimension $\leq 2$. The proof in this case is dual to that of (i).\\
(iii) $(a_{1},a_{2})=(e,f)$ where $e,f$ are positive edges. In this case the cell $(b_{1},b_{2})$ has to be of dimension 2, therefore there are three possible cases for $(b_{1},b_{2})$.\\
(iii-1) $(b_{1},b_{2})=(g,h)$ where $g,h$ are positive edges. It follows that
\begin{align*}
p(e,f)&= p (\lambda,[e,f]) \tag{by Lemma \ref{p-par-b-1}}\\
&= p(\lambda,[g,h]) \tag{since $[e,f]=e \otimes f=g \otimes h=[g,h]$}\\
&= p(g,h). \tag{by Lemma \ref{p-par-b-1}}
\end{align*}
(iii-2) $(b_{1},b_{2})=(u,\sigma)$ for some $u \in F$ and $\sigma$ a 2-cell. It follows that
\begin{align*}
p(e,f)&= p (\lambda,[e,f]) \tag{by Lemma \ref{p-par-b-1}}\\
&= p(\lambda,u.\sigma) \tag{since $[e,f]=e \otimes f=u \otimes \sigma=u.\sigma$}\\
&= p(u,\sigma). \tag{by Lemma \ref{p-par-b-1}}
\end{align*}
(iii-3) $(b_{1},b_{2})=(\sigma,u)$ for some $u \in F$ and $\sigma$ a 2-cell. This case is dual to (iii-2).
\end{proof}

\begin{proposition} \label{cer}
\begin{itemize}
\item[(a)] The relation $\sim$ is a cellular equivalence relation.
\item[(b)] The quotient space $\mathcal{K}_{3}/\sim$ is a cell complex whose cells are $p(a_{1},a_{2})$ for every $(a_{1},a_{2}) \in \mathcal{K}_{3}$.
\end{itemize}
\end{proposition}
\begin{proof}
(a) We will verify the three requirements for $\sim$ to be a cellular equivalence.

(i) Let $(a_{1},a_{2}) \in \mathcal{K}_{3}$. We prove that
\begin{equation} \label{preim}
p^{-1} p((a_{1}, a_{2}) \setminus \partial (a_{1}, a_{2}))= \underset{(b_{1},b_{2})\parallel (a_{1},a_{2})}{\bigcup} (b_{1}, b_{2}) \setminus \partial (b_{1}, b_{2}),
\end{equation}
which implies that $p^{-1} p((a_{1}, a_{2}) \setminus \partial (a_{1}, a_{2}))$ is a union of open cells which in our case have the same dimension as $(a_{1}, a_{2}) \setminus \partial (a_{1}, a_{2})$. Let $(x_{1},x_{2}) \in p^{-1} p((a_{1}, a_{2}) \setminus \partial (a_{1}, a_{2}))$ be any element. The condition on $(x_{1},x_{2})$ implies that $p(x_{1},x_{2}) \in p((a_{1}, a_{2}) \setminus \partial (a_{1}, a_{2}))$, hence there is some $(y_{1},y_{2}) \in (a_{1}, a_{2}) \setminus \partial (a_{1}, a_{2})$ such that $p(x_{1},x_{2})=p(y_{1},y_{2})$. It follows that $(x_{1},x_{2}) \sim (y_{1},y_{2})$ which means in particular that there is a cell $(b_{1},b_{2})$ such that $(b_{1},b_{2}) \parallel (a_{1},a_{2})$ and $(x_{1},x_{2}) \in (b_{1},b_{2}) \setminus \partial (b_{1},b_{2})$. Conversely, let $(x_{1},x_{2}) \in (b_{1}, b_{2}) \setminus \partial (b_{1}, b_{2})$ where $(b_{1},b_{2}) \parallel (a_{1},a_{2})$ which from Lemma \ref{p-par} is equivalent to $p((b_{1}, b_{2}) \setminus \partial (b_{1}, b_{2})) = p((a_{1}, a_{2}) \setminus \partial (a_{1}, a_{2}))$. It follows that $p(x_{1},x_{2}) \in p((a_{1}, a_{2}) \setminus \partial (a_{1}, a_{2}))$, or equivalently that $(x_{1},x_{2}) \in p^{-1} p((a_{1}, a_{2}) \setminus \partial (a_{1}, a_{2}))$ proving the equality.

(ii) Since all open cells in (\ref{preim}) have the same (minimal) dimension, it remains to prove that for every $(b_{1},b_{2})\parallel (a_{1},a_{2})$ the restriction of $p$ on $(b_{1}, b_{2}) \setminus \partial (b_{1}, b_{2})$ is a bijection onto $p((a_{1}, a_{2}) \setminus \partial (a_{1}, a_{2}))$. The map is clearly onto from the equality $p((b_{1}, b_{2}) \setminus \partial (b_{1}, b_{2})) = p((a_{1}, a_{2}) \setminus \partial (a_{1}, a_{2}))$. To see that it is 1-1 we let $(x_{1},x_{2}), (y_{1},y_{2}) \in (b_{1}, b_{2}) \setminus \partial (b_{1}, b_{2})$ such that $p(x_{1},x_{2})=p(y_{1},y_{2})$. This implies that $(x_{1},x_{2})$ and $(y_{1},y_{2})$ are equivalent pairs in the same open cell $(b_{1}, b_{2}) \setminus \partial (b_{1}, b_{2})$. Lemma \ref{p-inj} implies that $(x_{1},x_{2})=(y_{1},y_{2})$. Finally, Lemma \ref{p-par-b} implies that whenever $(a_{1},a_{2})\parallel (b_{1},b_{2})$, then $p(a_{1},a_{2})=p(b_{1},b_{2})$.  

(iii) Let now $a=(a_{1},a_{2}), b=(b_{1},b_{2})$ be two $\sim$-minimal cells with the corresponding characteristic maps $\Phi_{(a_{1},a_{2})}: I^{\text{dim}a} \rightarrow \mathcal{K}_{3}$ and $\Phi_{(b_{1},b_{2})}: I^{\text{dim}b} \rightarrow \mathcal{K}_{3}$, and let $h: I^{\text{dim}a} \rightarrow I^{\text{dim}b}$ be the idenitiy homeomorphism. We prove that
\begin{equation} \label{crutial}
p \circ \Phi_{(b_{1}, b_{2})} \circ h= p \circ \Phi_{(a_{1}, a_{2})}.
\end{equation}
If $u \in I^{\text{dim}a} \setminus \partial I^{\text{dim}a}$, then 
$$(x_{1},x_{2})=\Phi_{(a_{1}, a_{2})}(u) \in (a_{1},a_{2}) \setminus \partial (a_{1},a_{2})$$ 
and 
$$h(u) \in I^{\text{dim}b}\setminus \partial I^{\text{dim}b}.$$ 
It follows that 
$$(y_{1},y_{2})=\Phi_{(b_{1},b_{2})}(h(u)) \in (b_{1},b_{2}) \setminus \partial (b_{1},b_{2}).$$ 
From the convention proceeding Definition \ref{relK} we have that 
$$u=\Phi^{-1}_{(a_{1}, a_{2})}(x_{1}, x_{2})$$
and
$$h(u)=\Phi^{-1}_{(b_{1}, b_{2})}(y_{1},y_{2}),$$
consequently
$$\Phi^{-1}_{(b_{1}, b_{2})}(y_{1},y_{2})=(h \circ \Phi^{-1}_{(a_{1}, a_{2})})(x_{1},x_{2})=\Phi^{-1}_{(a_{1}, a_{2})}(x_{1},x_{2}),$$
and then 
$$(y_{1},y_{2})=\Phi_{(b_{1}, b_{2})}(\Phi^{-1}_{(a_{1}, a_{2})}(x_{1},x_{2})).$$
So we now have that $(x_{1},x_{2}) \sim (y_{1},y_{2})$, therefore
\begin{align*}
(p \circ \Phi_{(a_{1},a_{2})})(u) &= p(x_{1},x_{2})\\
&=p(y_{1},y_{2})\\
&= (p\circ \Phi_{(b_{1}, b_{2})} \circ h)(u).
\end{align*}
Now we prove the equality when $u \in \partial I^{\text{dim}a}$. There are four possible cases for $\text{dim}a$.

\begin{itemize}
\item[(0)] $\text{dim}a=0$. In this case $a_{1},a_{2},b_{1}, b_{2} \in F$ and then the result follows from (d-5) of Lemma \ref{p-par-b-1}.
\item[(1)] $\text{dim}a=1$. Then there are 1-cells $e$ and $f$ and 0-cells $u,v$ such that $a=(e,u)$ and $b=(f,v)$ or $b=(v,f)$, or $a=(u,e)$ and $b=(f,v)$ or $b=(v,f)$. We will discuss only one of these cases since the other ones are treated in a similar fashion. Let for instance $a=(e,u)$ and $b=(v,f)$, therefore from Remark \ref{d2} we have
$$e.u=e \otimes u=v \otimes f=v.f,$$
consequently
$$\iota(e)u=v \iota(f) \text{ and } \tau(e)u=v\tau(f).$$
If $t \in \partial I^{1}$, then $t=0$ or $t=1$. In the first case, 
\begin{align*}
(p \circ \Phi_{(e,u)})(0)&= p((\iota(e), u))\\
&= p((v, \iota(f))) \tag{by Lemma \ref{p-par-b-1}}\\
&= (p \circ \Phi_{(v,f)} \circ h)(0).
\end{align*}
In the second case when $t=1$, 
\begin{align*}
(p \circ \Phi_{(e,u)})(1)&= p((\tau(e), u))\\
&= p((v, \tau(f))) \tag{by Lemma \ref{p-par-b-1}}\\
&= (p \circ \Phi_{(v,f)} \circ h)(1).
\end{align*}
\item[(2)] $\text{dim}a=2$. The proof will be based on what we proved in (1) that whenever $(a_{1},a_{2})$ and $(b_{1},b_{2})$ are parallel cells of dimension 1, then for every $t \in I$, 
\begin{equation} \label{cr-1}
(p \circ \Phi_{(a_{1}, a_{2})})(t)=(p \circ \Phi_{(b_{1}, b_{2})})(t).
\end{equation}
There are a few subcases here.
\begin{itemize}
\item[(2-1)] $a=(\delta, w)$ with $\delta \in F.\mathbf{p}.F$ and $w \in F$. Then $b=(\delta',w')$ or $b=(w', \delta')$ with $w' \in F$ and $\delta' \in F.\mathbf{p}.F$, and $\delta.w=\delta'.w'$ or $\delta.w=w'.\delta'$ respectively. We will discuss here the case when $a=(\delta, w)$ and $b=(w', \delta')$. Assume that $\delta=u. \sigma. v$ where $\sigma \in \mathbf{p}$ and $u,v \in F$, and that $\delta'=u'. \sigma'. v'$ where $\sigma' \in \mathbf{p}$ and $u',v' \in F$. From the assumption that $(u.\sigma.v, w) \parallel (w',u'.\sigma'.v')$ we derive from Remark \ref{d2} that $u.\sigma.vw=w'u'.\sigma'.v'$ and then $u=w'u'$, $\sigma=\sigma'$, $vw=v'$. Assume that $\sigma$ is a 2-cell of type (iii) which is attached along the loop
$$e^{\varepsilon_{1}}_{1} \circ \dots \circ e^{\varepsilon_{s}}_{s} \circ e^{\varepsilon_{s+1}}_{s+1} \circ \dots \circ e_{n}^{\varepsilon_{n}},$$
therefore $u.\sigma.v$ is attached along the loop
$$u.e^{\varepsilon_{1}}_{1}.v \circ \dots \circ u.e^{\varepsilon_{s}}_{s}.v \circ u.e^{\varepsilon_{s+1}}_{s+1}.v \circ \dots \circ u.e_{n}^{\varepsilon_{n}}.v,$$
and $u'.\sigma'.v'=u'.\sigma.v'$ is attached along the loop
$$u'.e^{\varepsilon_{1}}_{1}.v' \circ \dots \circ u'.e^{\varepsilon_{s}}_{s}.v' \circ u'.e^{\varepsilon_{s+1}}_{s+1}.v' \circ \dots \circ u'.e_{n}^{\varepsilon_{n}}.v'.$$
Let now $(t_{1}, t_{2}) \in \partial I^{2}$ be arbitrary. We will distinguish between cases when both coordinates are $0$ or $1$, and the cases when one is either $0$ or $1$ and the other is in $]0,1[$. Since the proofs when both coordinates are $0$ or $1$ are similar to each other, we will discuss here only one of them. Let assume for instance that $(t_{1}, t_{2})=(1,0)$, then
\begin{align*}
(p \circ \Phi_{(u.\sigma.v,w)})(1,0)&= p ((\Phi_{u.\sigma.v}(1,0), w))\\
&=p (u.\tau e^{\varepsilon_{1}}_{1}.v,w)\\
&= p(w', u'.\tau e^{\varepsilon_{1}}_{1}.v') \tag{since $u=w'u'$ and $vw=v'$}\\
&= p(w', \Phi_{u'.\sigma.v'}(1,0))\\
&=p(\Phi_{(w', u'.\sigma.v')}(1,0))\\
&= (p \circ \Phi_{(w', u'.\sigma.v')} \circ h)(1,0).
\end{align*}
Let’s now see the case when $t_{1}=0$ and $t_{2} \in ]0,1[$. We have that
\begin{align*}
(p \circ \Phi_{(u.\sigma.v, w)})(0,t_{2})&= p(\Phi_{u.\sigma.v}(0,t_{2}), w)\\
&= p((\nu_{\mathcal{D}} \circ j_{1} \circ \Phi^{\varepsilon_{n}}_{u.e_{n}.v} \circ h_{3})(0,t_{2}), w)\\
&= p(\nu_{\mathcal{D}}(j_{1}(\Phi^{\varepsilon_{n}}_{u.e_{n}.v}(1-t_{2}))), w)\\
&=p(\nu_{\mathcal{D}}(j_{1}(\Phi_{u.e_{n}.v}((1-\varepsilon_{n})/2+\varepsilon_{n}\cdot (1-t_{2})))), w)\\
&=p (\Phi_{u.e_{n}.v}((1-\varepsilon_{n})/2+\varepsilon_{n}\cdot (1-t_{2}))), w)\\
&=(p \circ \Phi_{(u.e_{n}.v, w)})((1-\varepsilon_{n})/2+\varepsilon_{n}\cdot (1-t_{2}))\\
&= (p \circ  \Phi_{(w',u'.e_{n}.v')})((1-\varepsilon_{n})/2+\varepsilon_{n}\cdot (1-t_{2})) \tag{by (\ref{cr-1})}\\
&= p(w', \Phi_{u'.e_{n}.v'}((1-\varepsilon_{n})/2+\varepsilon_{n}\cdot (1-t_{2})))\\
&= p(w', \Phi^{\varepsilon_{n}}_{u'.e_{n}.v'}(1-t_{2}))\\
&=p(w', \nu_{\mathcal{D}}(j_{1}(\Phi^{\varepsilon_{n}}_{u'.e_{n}.v'}(1-t_{2})))\\
&=p(w',(\nu_{\mathcal{D}} \circ j_{1} \circ \Phi^{\varepsilon_{n}}_{u'.e_{n}.v'}\circ h_{3})(0,t_{2}))\\
&=p(w', \Phi_{u'.\sigma'.v'}(0,t_{2}))\\
&=(p \circ \Phi_{(w', u'.\sigma'.v')} \circ h)(0,t_{2}).
\end{align*}
The case when $t_{1} \in ]0,1[$ and $t_{2}=0$ is symmetric to the above. Let now see the case when $t_{1}=1$ and $t_{2} \in ]0,1[$. Then there is some $i=2,...,s$ such that $t \in [\dfrac{i-2}{s-1}, \dfrac{i-1}{s-1}]$. It follows that

\resizebox{0.9\textwidth}{!}
{\begin{minipage}{\textwidth}
\begin{align*}
(p \circ \Phi_{(u.\sigma.v, w)})(1,t_{2})&= p(\Phi_{u.\sigma.v}(1,t_{2}), w)\\
&= p ((\nu_{\mathcal{D}} \circ j_{1} \circ (\Phi^{\varepsilon_{2}}_{u.e_{2}.v} \ast \dots \ast \Phi^{\varepsilon_{s}}_{u.e_{s}.v}) \circ h_{2})(1,t_{2}), w)\\
&=p((\nu_{\mathcal{D}} \circ j_{1}  \circ (\Phi^{\varepsilon_{2}}_{u.e_{2}.v} \ast \dots \ast \Phi^{\varepsilon_{s}}_{u.e_{s}.v}))(t_{2}), w)\\
&=p(\nu_{\mathcal{D}}(j_{1}(\Phi^{\varepsilon_{i}}_{u.e_{i}.v}((s-1)t_{2}-(i-2)))) , w) \tag{by (\ref{C2})}\\
&=p(\nu_{\mathcal{D}}(j_{1}(\Phi_{u.e_{i}.v}((1-\varepsilon_{i})/2+ \varepsilon_{i}((s-1)t_{2}-(i-2))))) , w)\\
&= p(\Phi_{u.e_{i}.v}((1-\varepsilon_{i})/2+ \varepsilon_{i}((s-1)t_{2}-(i-2))) , w)\\
&= (p \circ \Phi_{(u.e_{i}.v, w)})((1-\varepsilon_{i})/2+ \varepsilon_{i}((s-1)t_{2}-(i-2)))\\
&= (p \circ \Phi_{(w', u'.e_{i}.v')})((1-\varepsilon_{i})/2+ \varepsilon_{i}((s-1)t_{2}-(i-2))) \tag{by (\ref{cr-1})}\\
&=p(w', \Phi_{u'.e_{i}.v'}((1-\varepsilon_{i})/2+ \varepsilon_{i}((s-1)t_{2}-(i-2))) \\
&=p(w', \Phi^{\varepsilon_{i}}_{u'.e_{i}.v'}((s-1)t_{2}-(i-2))) \\
&= p (w', \nu_{\mathcal{D}}(j_{1}(\Phi^{\varepsilon_{i}}_{u'.e_{i}.v'}((s-1)t_{2}-(i-2))))\\
&= p(w', (\nu_{\mathcal{D}}\circ j_{1}  \circ (\Phi^{\varepsilon_{2}}_{u'.e_{2}.v'} \ast \dots \ast \Phi^{\varepsilon_{s}}_{u'.e_{s}.v'}))(t_{2}))\\
&= p (w', (\nu_{\mathcal{D}} \circ j_{1}  \circ (\Phi^{\varepsilon_{2}}_{u'.e_{2}.v'} \ast \dots \ast \Phi^{\varepsilon_{s}}_{u'.e_{s}.v'}) \circ h_{2})(1,t_{2}))\\
&=p(w', \Phi_{u'.\sigma.v'}(1,t_{2}))\\
&= (p \circ \Phi_{(w', u'.\sigma'.v')} \circ h)(1,t_{2}).
\end{align*}
\end{minipage}
}
The case when $t_{1} \in ]0,1[$ and $t_{2} =1$ is symmetric to the above. We will present its proof in full for convenience.

\resizebox{0.9\textwidth}{!}
{\begin{minipage}{\textwidth}
\begin{align*}
(p \circ \Phi_{(u.\sigma.v, w)})(t_{1},1)&= p(\Phi_{u.\sigma.v}(t_{1},1), w)\\
&=p((\nu_{\mathcal{D}} \circ j_{1} \circ (\Phi^{\varepsilon_{s+1}}_{e_{s+1}} \ast \dots \ast \Phi^{\varepsilon_{n-1}}_{e_{n-1}}) \circ h_{3})(t_{1},1), w)\\
&= p((\nu_{\mathcal{D}} \circ j_{1} \circ (\Phi^{\varepsilon_{s+1}}_{e_{s+1}} \ast \dots \ast \Phi^{\varepsilon_{n-1}}_{e_{n-1}}))(1-t_{1}),w)\\
&=p(\nu_{\mathcal{D}} (j_{1} (\Phi^{\varepsilon_{s+1+i}}_{u.e_{s+1+i}.v}(r(1-t_{1})-i))), w)) \tag{by (\ref{C4})}\\
&= p(\nu_{\mathcal{D}} (j_{1} (\Phi_{u.e_{s+1+i}.v}((1-\varepsilon_{s+1+i})/2+ \varepsilon_{s+1+i}(r(1-t_{1})-i)))), w))\\
&= p (\Phi_{u.e_{s+1+i}.v}((1-\varepsilon_{s+1+i})/2+ \varepsilon_{s+1+i}(r(1-t_{1})-i)), w))\\
&= (p \circ \Phi_{(u.e_{s+1+i}.v, w)}((1-\varepsilon_{s+1+i})/2+ \varepsilon_{s+1+i}(r(1-t_{1})-i))\\
&=(p \circ \Phi_{(w', u'.e_{s+1+i}.v')}((1-\varepsilon_{s+1+i})/2+ \varepsilon_{s+1+i}(r(1-t_{1})-i)) \tag{by (\ref{cr-1})}\\
&=p(w', \Phi_{u'.e_{s+1+i}.v'}((1-\varepsilon_{s+1+i})/2+ \varepsilon_{s+1+i}(r(1-t_{1})-i)))\\
&= p(w', \Phi^{\varepsilon_{s+1+i}}_{u'.e_{s+1+i}.v'}(r(1-t_{1})-i))\\
&=p(w', \nu_{\mathcal{D}}(j_{1}(\Phi^{\varepsilon_{s+1+i}}_{u'.e_{s+1+i}.v'}(r(1-t_{1})-i)))\\
&=p(w', (\nu_{\mathcal{D}}\circ j_{1} \circ (\Phi^{\varepsilon_{s+1}}_{u'.e_{s+1}.v'} \ast \dots \ast \Phi^{\varepsilon_{n-1}}_{u'.e_{n-1}.v'})\circ h_{4})(t_{1},1)))\\
&=(p \circ \Phi_{(w', u'.\sigma.v')})(t_{1},1)\\
&= (p \circ \Phi_{(w', u'.\sigma.v')} \circ h)(t_{1},1).
\end{align*}
\end{minipage}
}

Now we discuss the case when $\sigma$ is of type (ii). Let assume that $\sigma$ is attached along
$$e^{\varepsilon_{1}}_{1} \circ e^{\varepsilon_{2}}_{2} \circ e_{3}^{\varepsilon_{3}},$$
therefore $u.\sigma.v$ is attached along 
$$u.e^{\varepsilon_{1}}_{1}.v \circ u.e^{\varepsilon_{2}}_{2}.v \circ u.e_{3}^{\varepsilon_{3}}.v,$$
and $u.\sigma'.v=u.\sigma. v$ is attached along 
$$u'.e^{\varepsilon_{1}}_{1}.v' \circ u'.e^{\varepsilon_{2}}_{2}.v' \circ u'.e_{3}^{\varepsilon_{3}}.v'.$$
We discuss here the most delicate case when $t_{2}=1$ and $t_{1} \in I$ is arbitrary. The discussions of other cases are identical to the corresponding ones for the cell of type (iii). 
\begin{align*}
(p \circ \Phi_{(u.\sigma.v, w)})(t_{1},1)&= p(\Phi_{u.\sigma.v}(t_{1},1), w)\\
&=p(\Phi^{\varepsilon_{2}}_{u.e_{2}.v}(1), w) \tag{by (\ref{B4})}\\
&=p(\Phi_{u.e_{2}.v}((1-\varepsilon_{2})/2+ \varepsilon_{2}\cdot 1), w)\\
&=(p \circ \Phi_{(u.e_{2}.v, w)})((1-\varepsilon_{2})/2+ \varepsilon_{2}\cdot 1)\\
&=(p \circ \Phi_{(w',u'.e_{2}.v')})((1-\varepsilon_{2})/2+ \varepsilon_{2}\cdot 1) \tag{by (\ref{cr-1})}\\
&= p(w', \Phi_{u'.e_{2}.v'}((1-\varepsilon_{2})/2+ \varepsilon_{2}\cdot 1))\\
&= p(w', \Phi^{\varepsilon_{2}}_{u'.e_{2}.v'}(1))\\
&= p(w', \nu_{\mathcal{D}}j_{1}(\tau (u'.e^{\varepsilon_{2}}_{2}.v')))\\
&= p(w', (c_{(\nu_{\mathcal{D}} \circ j_{1})(\tau (u'.e^{\varepsilon_{2}}_{2}.v')} \circ h_{4})(t_{1},1))\\
&= p(w', \Phi_{u'.\sigma.v'}(t_{1},1))\\
&=(p \circ \Phi_{(w',u'.\sigma'.v')})(t_{1},1).
\end{align*}
When the 2-cell $\sigma$ is of type (i) and attached along $e^{\varepsilon_{1}}_{1} \circ e^{\varepsilon_{2}}_{2}$, then $u.\sigma. v$ is attached along $u.e^{\varepsilon_{1}}_{1}.v \circ u.e^{\varepsilon_{2}}_{2}.v$ and $u'.\sigma.v'$ is attached along $u'.e^{\varepsilon_{1}}_{1}.v' \circ u'.e^{\varepsilon_{2}}_{2}.v'$. We discuss here for convenience the case when $t_{1} \in I$ and $t_{2}=1$. Its dual case when $t_{1}=1$ and $t_{2} \in I$ has a similar proof and is omitted. The other cases are proved similarly to their corresponding ones for the cell of type (iii).
\begin{align*}
(p \circ \Phi_{(u.\sigma.v, w)})(t_{1},1)&= p(\Phi_{u.\sigma.v}(t_{1},1), w)\\
&=p(\Phi^{\varepsilon_{2}}_{u.e_{2}.v}(0), w) \tag{by (\ref{A4})}\\
&=p(\Phi_{u.e_{2}.v}((1-\varepsilon_{2})/2+\varepsilon_{2} \cdot 0), w)\\
&=(p \circ \Phi_{(u.e_{2}.v, w)})((1-\varepsilon_{2})/2+\varepsilon_{2} \cdot 0)\\
&=(p \circ \Phi_{(w',u'.e_{2}.v')})((1-\varepsilon_{2})/2+\varepsilon_{2} \cdot 0) \tag{by (\ref{cr-1})}\\
&= p(w', \Phi_{u'.e_{2}.v'}((1-\varepsilon_{2})/2+\varepsilon_{2} \cdot 0)))\\
&= p(w', \Phi^{\varepsilon_{2}}_{u'.e_{2}.v'}(0))\\
&= p(w', \nu_{\mathcal{D}}j_{1}(\tau (u'.e^{\varepsilon_{2}}_{2}.v')))\\
&= p(w', (c_{(\nu_{\mathcal{D}} \circ j_{1})(\tau (u'.e^{\varepsilon_{2}}_{2}.v')} \circ h_{4})(t_{1},1))\\
&= p(w', \Phi_{u'.\sigma.v'}(t_{1},1))\\
&=(p \circ \Phi_{(w',u'.\sigma'.v')})(t_{1},1).
\end{align*}

\item[(2-2)] The case when $a=(w,\delta)$ with $w \in F$ and $\delta \in F.\mathbf{p}.F$ is dual to case (2-1).
\item[(2-3)] $a=([e,f],w)$ where $e,f$ are 1-cells and $w \in F$. There are three possible cases for $b$ here. In the first case $b=([g,h], w')$ with $g,h$ being 1-cells and $w' \in F$. In the second case $b=(w', [g,h])$ with $g,h$ being 1-cells and $w' \in F$, and in the third case $b=(g,h)$ with $g,h$ being 1-cells. The proofs for the first two cases are similar to those of (2-2) and are omitted. So let assume that $b=(g,h)$, hence from the assumption we have from Remark \ref{d2} that $[e,f.w]=[e,f].w=[g,h]$. It follows from Lemma \ref{unamb} that there is some $u \in F$ such that $e.u=g$ and $u.h=f.w$, or that $g.u=e$ and $u.(f.w)=h$. We make the proof under this last assumption. Let $(t_{1}, t_{2}) \in \partial I^{2}$.   As in (2-2) we discuss separately cases when both $t_{1},t_{2} \in \{0,1\}$, and when one of them is in $\{0,1\}$ and the other one is from $]0,1[$. Again, in order to avoid the repetition of similar arguments, we pick only one case from each group. Assume for example that $(t_{1}, t_{2})=(0,1)$. We see that
\begin{align*}
(p \circ \Phi_{(g,h)})(0,1)&= p(\Phi_{(g,h)}(0,1))\\
&= p(\Phi_{g}(0), \Phi_{h}(1))\\
&= p(\iota(g), \tau(h))\\
&= p(\iota(g), u \tau(f)w)\\
&=p(\iota(g)u\tau(f),w) \tag{by Lemma \ref{p-par-b-1}}\\
&= p(\iota(e)\tau(f), w)\\
&= p(\Phi_{[e,f]}(0,1), w) \tag{by (\ref{S2})}\\
&= (p \circ \Phi_{([e,f],w)} \circ h)(0,1).
\end{align*}
From the second group of cases we discuss here the case when $t_{1} \in ]0,1[$ and $t_{2}=0$. In this case we have
\begin{align*}
(p \circ \Phi_{(g,h)})(t_{1},0)&= p(\Phi_{g}(t_{1}), \Phi_{h}(0))\\
&= p(\Phi_{g}(t_{1}), \iota(h))\\
&= p(\Phi_{g}(t_{1}), u \iota(f)w)\\
&= (p \circ \Phi_{(g, u \iota(f)w)})(t_{1})\\
&= (p \circ \Phi_{(g.u \iota(f),w)})(t_{1}) \tag{by (\ref{cr-1})}\\
&= p(\Phi_{g.u \iota(f)} (t_{1}), w) \\
&= p(\Phi_{e.\iota(f)}(t_{1}), w)\\
&= p(\Phi_{[e,f]}(t_{1},0), w) \tag{by (\ref{S1})}\\
&= p(\Phi_{([e,f],w)} (t_{1},0))\\
&=(p \circ \Phi_{([e,f],w)} \circ h)(t_{1},0).
\end{align*}
\end{itemize}
\item[(3)] Both cells are of dimension 3. The proof will be based on what we proved in (2) that whenever $(a_{1},a_{2})$ and $(b_{1},b_{2})$ are parallel cells of dimension 2, then for every $(t_{1},t_{2}) \in I^{2}$, 
\begin{equation} \label{cr-2}
(p \circ \Phi_{(a_{1}, a_{2})})(t_{1},t_{2})=(p \circ \Phi_{(b_{1}, b_{2})})(t_{1},t_{2}).
\end{equation}
There are three cases for our parallel 3-cells of $\mathcal{K}_{3}$.
\begin{itemize}
\item[(3-1)] The first case is when $(a_{1},a_{2})=([e,f],g)$ and $(b_{1},b_{2})=(e, [f,g])$. We note that $(t_{1}, t_{2}, t_{3}) \in \partial I^{3}$,
if and only if at least one of $t_{i}$'s is either $0$ or $1$. If $t_{3}=0$, then on the one hand we have that
\begin{align*}
\Phi_{([e,f],g)} (t_{1}, t_{2}, 0)&=(\Phi_{[e,f]}(t_{1}, t_{2}), \Phi_{g}(0))\\
&=(\Phi_{[e,f]}(t_{1}, t_{2}),\iota(g))\\
&= \Phi_{([e,f], \iota (g))}(t_{1},t_{2}),
\end{align*}
and on the other hand
\begin{align*}
\Phi_{(e,[f,g])} (t_{1}, t_{2}, 0)&=(\Phi_{e}(t_{1}), \Phi_{[f,g]}(t_{2},0))\\
&= (\Phi_{e}(t_{1}), \Phi_{f.\iota(g)}(t_{2})). \tag{by (\ref{S1})}\\
&= \Phi_{(e, f. \iota (g))}(t_{1},t_{2}).
\end{align*}
Observe now that
$$\Phi_{([e,f], \iota (g))}(t_{1},t_{2}) \in ([e,f], \iota(g))$$ \text{ and } 
$$\Phi_{(e, f. \iota (g))}(t_{1},t_{2}) \in (e, f.\iota(g)).$$
But from Lemma \ref{p-par-b-1}
$$([e,f], \iota(g)) \parallel (e, f.\iota(g)),$$
then from (\ref{cr-2}) it follows that
\begin{align*}
(p\circ \Phi_{([e,f],g)} )(t_{1}, t_{2}, 0)&= (p \circ \Phi_{(e,[f,g])})(t_{1},t_{2},0)\\
&=(p \circ \Phi_{(e,[f,g])} \circ h) (t_{1}, t_{2},0).
\end{align*}
If $t_{3}=1$, then on the one hand we have that
\begin{align*}
\Phi_{([e,f],g)} (t_{1}, t_{2}, 1)&=(\Phi_{[e,f]}(t_{1}, t_{2}), \Phi_{g}(1))\\
&=(\Phi_{[e,f]}(t_{1}, t_{2}),\tau(g))\\
&= \Phi_{([e,f], \tau (g))}(t_{1},t_{2}),
\end{align*}
and on the other hand
\begin{align*}
\Phi_{(e,[f,g])} (t_{1}, t_{2}, 1)&=(\Phi_{e}(t_{1}), \Phi_{[f,g]}(t_{2},1))\\
&= (\Phi_{e}(t_{1}), \Phi_{f.\tau(g)}(t_{2})). \tag{by (\ref{S2})}\\
&=\Phi_{(e, f. \tau(g))}(t_{1},t_{2}).
\end{align*}
Observe now that
$$(\Phi_{[e,f]}(t_{1}, t_{2}),\tau(g)) \in ([e,f], \tau(g))$$ \text{ and } 
$$(\Phi_{e}(t_{1}), \Phi_{f.\tau(g)}(t_{2})) \in (e, f.\tau(g)).$$
But from Lemma \ref{p-par-b-1}
$$([e,f], \tau(g)) \parallel (e, f.\tau(g)),$$
then from (\ref{cr-2}) it follows that
\begin{align*}
(p\circ \Phi_{([e,f],g)} )(t_{1}, t_{2}, 1)&= p (\Phi_{e}(t_{1}), \Phi_{[f,g]}(t_{2},1))\\
&=(p \circ \Phi_{(e,[f,g])} \circ h) (t_{1}, t_{2},1).
\end{align*}
Now we discuss cases when $t_{2}=0$ and $t_{2}=1$. When $t_{2}=0$ we have that
\begin{align*}
\Phi_{([e,f],g)}(t_{1},0,t_{3})&= (\Phi_{[e,f]}(t_{1},0), \Phi_{g}(t_{3}))\\
&= (\Phi_{e.\iota(f)}(t_{1}), \Phi_{g}(t_{3}))\\
&= \Phi_{(e.\iota (f), g)}(t_{1},t_{3}),
\end{align*}
and
\begin{align*}
\Phi_{(e,[f,g])}(t_{1},0,t_{3})&= (\Phi_{e}(t_{1}), \Phi_{[f,g]}(0,t_{3}))\\
&= (\Phi_{e}(t_{1}), \Phi_{\iota(f).g}(t_{3})) \tag{by (\ref{S2})}\\
&= \Phi_{(e, \iota(f).g)}(t_{1},t_{3}).
\end{align*}
Observe that 
$$\Phi_{(e.\iota (f), g)}(t_{1},t_{3}) \in (e.\iota(f), g) \text{ and } \Phi_{(e, \iota(f).g)}(t_{1},t_{3}) \in (e, \iota(f).g),$$
and since 
$$(e.\iota(f), g) \parallel (e, \iota(f).g)$$
it follows from (\ref{cr-2}) that
\begin{align*}
(p\circ \Phi_{([e,f],g)} )(t_{1}, 0,t_{3})&= (p \circ \Phi_{(e,[f,g])})(t_{1},0,t_{3})\\
&= (p \circ \Phi_{(e,[f,g])} \circ h)(t_{1},0,t_{3}).
\end{align*}
In the case when $t_{2}=1$ we have that
\begin{align*}
\Phi_{([e,f],g)}(t_{1},1,t_{3})&= (\Phi_{[e,f]}(t_{1},1), \Phi_{g}(t_{3}))\\
&= (\Phi_{e.\tau(f)}(t_{1}), \Phi_{g}(t_{3}))\\
&=\Phi_{(e.\tau (f),g)}(t_{1},1,t_{3}),
\end{align*}
and
\begin{align*}
\Phi_{(e,[f,g])}(t_{1},1,t_{3})&= (\Phi_{e}(t_{1}), \Phi_{[f,g]}(1,t_{3}))\\
&= (\Phi_{e}(t_{1}), \Phi_{\tau(f).g}(t_{3})) \tag{by (\ref{S1})}\\
&= \Phi_{(e, \tau (f).g)}(t_{1},t_{3}).
\end{align*}
Observe that 
$$\Phi_{(e.\tau (f),g)}(t_{1},1,t_{3}) \in (e.\tau(f), g) \text{ and } \Phi_{(e, \tau (f).g)}(t_{1},t_{3}) \in (e, \tau(f).g),$$
and since 
$$(e.\tau(f), g) \parallel (e, \tau(f).g)$$
it follows from (\ref{cr-2}) that
\begin{align*}
(p\circ \Phi_{([e,f],g)} )(t_{1}, 1,t_{3})&= (p \circ \Phi_{(e,[f,g])} )(t_{1},1,t_{3}).\\
&= (p \circ \Phi_{(e,[f,g])} \circ h)(t_{1},1,t_{3}).
\end{align*}
Now we see the case when $t_{1}=0$ or $t_{1}=1$. In the first case we have that 
\begin{align*}
\Phi_{([e,f],g)}(0, t_{2},t_{3})&= (\Phi_{[e,f]}(0, t_{2}), \Phi_{g}(t_{3}))\\
&= (\Phi_{\iota(e).f}(t_{2}), \Phi_{g}(t_{3})) \tag{by (\ref{S2})}\\
&= \Phi_{(\iota(e).f, g)}(t_{2},t_{3}),
\end{align*}
and
\begin{align*}
\Phi_{(e,[f,g])}(0,t_{2},t_{3})&= (\Phi_{e}(0), \Phi_{[f,g]}(t_{2},t_{3}))\\
&= (\iota(e), \Phi_{[f,g]}(t_{2},t_{3}))\\
&=\Phi_{(\iota e, [f,g])}(t_{2},t_{3}).
\end{align*}
Observe that 
$$\Phi_{(\iota(e).f, g)}(t_{2},t_{3}) \in (\iota(e).f, g) \text{ and } \Phi_{(\iota e, [f,g])}(t_{2},t_{3}) \in (\iota(e), [f,g]),$$
and since 
$$(\iota(e).f, g) \parallel (\iota(e), [f,g])$$
it follows from (\ref{cr-2}) that
\begin{align*}
(p\circ \Phi_{([e,f],g)} )(0,t_{2}, t_{3})&= (p \circ \Phi_{(e,[f,g])})(0,t_{2},t_{3})\\
&= (p \circ \Phi_{(e,[f,g])} \circ h)(0,t_{2},t_{3}).
\end{align*}
In the second case when $t_{1}=1$ we have that 
\begin{align*}
\Phi_{([e,f],g)}(1, t_{2},t_{3})&= (\Phi_{[e,f]}(1, t_{2}), \Phi_{g}(t_{3}))\\
&= (\Phi_{\tau(e).f}(t_{2}), \Phi_{g}(t_{3}))\tag{by (\ref{S1})}\\
&=\Phi_{(\tau(e).f, g)}(t_{2},t_{3}),
\end{align*}
and
\begin{align*}
\Phi_{(e,[f,g])}(1,t_{2},t_{3})&= (\Phi_{e}(1), \Phi_{[f,g]}(t_{2},t_{3}))\\
&= (\tau(e), \Phi_{[f,g]}(t_{2},t_{3}))\\
&=\Phi_{(\tau(e), [f,g])}(t_{2},t_{3}).
\end{align*}
Observe that 
$$\Phi_{(\tau(e).f, g)}(t_{2},t_{3}) \in (\tau(e).f, g) \text{ and } \Phi_{(\tau(e), [f,g])}(t_{2},t_{3}) \in (\tau(e), [f,g]),$$
and since 
$$(\tau(e).f, g) \parallel (\tau(e), [f,g])$$
it follows from (\ref{cr-2}) that
\begin{align*}
(p\circ \Phi_{([e,f],g)} )(1,t_{2}, t_{3})&= (p \circ \Phi_{(e,[f,g])})(1,t_{2},t_{3})\\
&= (p \circ \Phi_{(e,[f,g])} \circ h)(1,t_{2},t_{3}).
\end{align*}
\item[(3-2)] The second case for the 3-cells is when $(a_{1},a_{2})=(\sigma.u, f)$ and $(b_{1},b_{2})=(\sigma, u.f)$ where $\sigma$ is a 2-cell from $F.\mathbf{p}.F$, $f$ is a 1-cell and $u \in F$. In the first case we assume that $\sigma$ arises from a (iii) type cell and that it is attached along the loop
$$e^{\varepsilon_{1}}_{1} \circ \dots \circ e^{\varepsilon_{s}}_{s} \circ e^{\varepsilon_{s+1}}_{s+1} \circ \dots \circ e_{n}^{\varepsilon_{n}}.$$
When $t_{1}=0$ we see that
\begin{align*}
\Phi_{(\sigma.u,f)}(0,t_{2},t_{3})&= (\Phi_{\sigma.u}(0,t_{2}), \Phi_{f}(t_{3}))\\
&=(\Phi^{\varepsilon_{n}}_{e_{n}.u}(1-t_{2}), \Phi_{f}(t_{3})) \tag{by (\ref{C3})}\\
&= (\Phi_{e_{n}.u}((1-\varepsilon_{n})/2+\varepsilon_{n}\cdot (1-t_{2})), \Phi_{f}(t_{3}))\\
&= \Phi_{(e_{n}.u, f)} ((1-\varepsilon_{n})/2+\varepsilon_{n}\cdot (1-t_{2}),t_{3})
\end{align*}
and that
$$\Phi_{(e_{n}.u, f)} ((1-\varepsilon_{n})/2+\varepsilon_{n}\cdot (1-t_{2}),t_{3}) \in (e_{n}.u, f).$$
Also we see that
\begin{align*}
\Phi_{(\sigma, u.f)}(0,t_{2},t_{3})&= (\Phi_{\sigma}(0,t_{2}), \Phi_{u.f}(t_{3}))\\
&=(\Phi^{\varepsilon_{n}}_{e_{n}}(1-t_{2}), \Phi_{u.f}(t_{3})) \tag{by (\ref{C3})}\\
&= (\Phi_{e_{n}}((1-\varepsilon_{n})/2+\varepsilon_{n}\cdot (1-t_{2})), \Phi_{u.f}(t_{3}))\\
&=\Phi_{(e_{n}, u.f)} ((1-\varepsilon_{n})/2+\varepsilon_{n}\cdot (1-t_{2}),t_{3}),
\end{align*}
and that
$$\Phi_{(e_{n}, u.f)} ((1-\varepsilon_{n})/2+\varepsilon_{n}\cdot (1-t_{2}),t_{3}) \in (e_{n}, u.f).$$
But 
$$(e_{n}.u, f) \parallel (e_{n}, u.f)$$
and then (\ref{cr-2}) implies that
\begin{align*}
(p \circ \Phi_{(\sigma.u,f)})(0,t_{2},t_{3})&= (p \circ \Phi_{(\sigma, u.f)})(0,t_{2},t_{3})\\
&= (p \circ \Phi_{(\sigma, u.f)} \circ h)(0,t_{2},t_{3}).
\end{align*}
Now we discuss the case when $t_{1}=1$. Assuming that there is some $i=2,...,s$ such that $t_{1} \in [\dfrac{i-2}{s-1}, \dfrac{i-1}{s-1}]$, it follows that 
\begin{align*}
\Phi_{\sigma.u}(1,t_{2})&= \Phi^{\varepsilon_{i}}_{e_{i}.u}((s-1)t_{2}-(i-2)) \tag{by (\ref{C2})}\\
&= \Phi_{e_{i}.u}((1-\varepsilon_{i})/2+\varepsilon_{i}((s-1)t_{2}-(i-2))).
\end{align*}
Using this we derive that
\begin{align*}
\Phi_{(\sigma.u,f)}(1,t_{2},t_{3})&= (\Phi_{\sigma.u}(1,t_{2}), \Phi_{f}(t_{3}))\\
&=( \Phi_{e_{i}.u}((1-\varepsilon_{i})/2+\varepsilon_{i}((s-1)t_{2}-(i-2))), \Phi_{f}(t_{3}))\\
&=\Phi_{(e_{i}.u,f)}((1-\varepsilon_{i})/2+\varepsilon_{i}((s-1)t_{2}-(i-2)), t_{3}).
\end{align*}
Observe that
$$\Phi_{(e_{i}.u,f)}((1-\varepsilon_{i})/2+\varepsilon_{i}((s-1)t_{2}-(i-2)), t_{3}) \in (e_{i}.u, f).$$
In the same way as above we see that
\begin{align*}
\Phi_{(\sigma, u.f)}(1,t_{2},t_{3})&= (\Phi_{\sigma}(1,t_{2}), \Phi_{u.f}(t_{3}))\\
&= (\Phi_{e_{i}}((1-\varepsilon_{i})/2+\varepsilon_{i} ((s-1)t_{2}+(i-2))), \Phi_{u.f}(t_{3}))\\
&= \Phi_{(e_{i}, u.f)}((1-\varepsilon_{i})/2+\varepsilon_{i} ((s-1)t_{2}+(i-2)), t_{3}),
\end{align*}
and that
$$ \Phi_{(e_{i}, u.f)}((1-\varepsilon_{i})/2+\varepsilon_{i} ((s-1)t_{2}+(i-2)), t_{3}) \in (g_{i}, u.f).$$
But 
$$(g_{i}.u, f) \parallel (g_{i}, u.f)$$
and then (\ref{cr-2}) implies that
\begin{align*}
(p \circ \Phi_{(\sigma.u,f)})(1,t_{2},t_{3})&= (p \circ \Phi_{(\sigma, u.f)})(1,t_{2},t_{3})\\
&= (p \circ \Phi_{(\sigma, u.f)} \circ h)(1,t_{2},t_{3}).
\end{align*}
The cases when $t_{2}=0$ or $t_{2}=1$ are proved similarly to the above two cases and are omitted here. It remains to make the proofs for the cases when $t_{3}=0$ or $t_{3}=1$. When $t_{3}=0$ we see that
\begin{align*}
\Phi_{(\sigma.u,f)}(t_{1},t_{2},0)&=(\Phi_{\sigma.u}(t_{1},t_{2}), \Phi_{f}(0))\\
&= (\Phi_{\sigma.u}(t_{1},t_{2}), \iota(f))\\
&=\Phi_{(\sigma.u, \iota(f))}(t_{1},t_{2}),
\end{align*}
and that
$$\Phi_{(\sigma.u, \iota(f))}(t_{1},t_{2}) \in (\sigma.u, \iota(f)).$$
Also we see that
\begin{align*}
\Phi_{(\sigma, u.f)}(t_{1},t_{2},0)&= (\Phi_{\sigma}(t_{1},t_{2}), \Phi_{u.f}(0))\\
&= (\Phi_{\sigma}(t_{1},t_{2}), u\iota(f))\\
&=\Phi_{(\sigma, u.f)}(t_{2},t_{3}),
\end{align*}
and that
$$\Phi_{(\sigma, u.f)}(t_{2},t_{3}) \in (\sigma, u\iota(f)).$$
But
$$(\sigma.u, \iota(f)) \parallel (\sigma, u\iota(f)),$$
and then (\ref{cr-2}) implies that
\begin{align*}
(p \circ \Phi_{(\sigma.u,f)})(t_{1},t_{2},0)&= (p \circ \Phi_{(\sigma, u.f)})(t_{1},t_{2},0)\\
&= (p \circ \Phi_{(\sigma, u.f)} \circ h)(t_{1},t_{2},0).
\end{align*}
The case when $t_{3}=1$ is dual to the above and is omitted. 

Now we check the case when $\sigma$ arises from a type (ii) cell. Assume that it is attached along the loop
$$e^{\varepsilon_{1}}_{1} \circ e^{\varepsilon_{2}}_{2} \circ e_{3}^{\varepsilon_{3}}.$$
The only case when the proof differs from the respective proofs for other types of 2-cells is when $t_{2}=1$ while $t_{1}, t_{3} \in I$ are arbitrary. In this case we have that
\begin{align*}
(p \circ \Phi_{(\sigma.u,f)})(t_{1},1,t_{3})&= p(\Phi_{(\sigma.u,f)}(t_{1},1,t_{3}))\\
&=p(\Phi_{\sigma.u}(t_{1},1), \Phi_{f}(t_{3}))\\
&=p(\Phi^{\varepsilon_{2}}_{e_{2}.u}(1), \Phi_{f}(t_{3})) \tag{by (\ref{B4})}\\
&= p(\Phi_{e_{2}.u}((1-\varepsilon_{2})/2+\varepsilon_{2}), \Phi_{f}(t_{3}))\\
&= (p \circ \Phi_{(e_{2}.u, f)})((1-\varepsilon_{2})/2+\varepsilon_{2}, t_{3})\\
&= (p \circ \Phi_{(e_{2}, u.f)})((1-\varepsilon_{2})/2+\varepsilon_{2}, t_{3})\tag{by (\ref{cr-2})}\\
&= p(\Phi_{e_{2}}((1-\varepsilon_{2})/2+\varepsilon_{2}), \Phi_{u.f}(t_{3})) \\
&=p(\Phi^{\varepsilon_{2}}_{e_{2}}(1), \Phi_{u.f}(t_{3}))\\
&=p(\Phi_{\sigma}(t_{1},1), \Phi_{u.f}(t_{3})) \tag{by (\ref{B4})}\\
&= (p \circ \Phi_{(\sigma,u.f)} \circ h)(t_{1},1,t_{3}).
\end{align*}
Now we check the case when $\sigma$ arises from a type (i) 2-cell therefore is attached along $e^{\varepsilon_{1}}_{1}\circ e_{2}^{\varepsilon_{2}}$. We discuss here the case when $t_{1}=1$ and $t_{2}, t_{3} \in I$ are arbitrary. The proof for the case when $t_{2}=1$ and $t_{1}, t_{3} \in I$ are arbitrary is similar and is omitted. The other possible cases have proofs identical to their corresponding ones for other type of 2-cells. In this case we have.
\begin{align*}
(p \circ \Phi_{(\sigma.u,f)})(1,t_{2}, t_{3})&= p(\Phi_{(\sigma.u,f)}(1,t_{2},t_{3}))\\
&=p(\Phi_{\sigma.u}(1,t_{2}), \Phi_{f}(t_{3}))\\
&=p(\Phi^{\varepsilon_{1}}_{e_{1}.u}(1), \Phi_{f}(t_{3})) \tag{by (\ref{A2})}\\
&=p(\Phi_{e_{1}.u}((1-\varepsilon_{1})/2+\varepsilon_{1}), \Phi_{f}(t_{3}))\\
&= (p \circ \Phi_{(e_{1}.u, f)})((1-\varepsilon_{1})/2+\varepsilon_{1}, t_{3})\\
&= (p \circ \Phi_{(e_{1}, u.f)})((1-\varepsilon_{1})/2+\varepsilon_{1}, t_{3}) \tag{by (\ref{cr-2})}\\
&=p(\Phi_{e_{1}}((1-\varepsilon_{1})/2+\varepsilon_{1}), \Phi_{u.f}(t_{3}))\\
&= p(\Phi^{\varepsilon_{1}}_{e_{1}}(1), \Phi_{u.f}(t_{3}))\\
&=p(\Phi_{\sigma}(t_{1},1), \Phi_{u.f}(t_{3})) \tag{by (\ref{A2})}\\
&= (p \circ \Phi_{(\sigma,u.f)} \circ h)(t_{1},1,t_{3}).
\end{align*}

\item[(3-3)] The third case for the 3-cells is when $(a_{1},a_{2})=(f, u.\sigma)$ and $(b_{1},b_{2})=(f.u, \sigma)$ where $\sigma$ is a 2-cell from $F.\mathbf{p}.F$, $f$ is a 1-cell and $u \in F$. The proof in this case is dual to that of (3-2) and is omitted. 
\end{itemize}
\end{itemize}
(b) That $\mathcal{K}_{3}/\sim$ is a cell complex it follows from Proposition \ref{cc}. Since all $\sim$-minimal cells are of the same dimension, then it follows that the cells of $\mathcal{K}_{3}/\sim$ are $p(a_{1},a_{2})$ with $(a_{1},a_{2})$ varying in the cell decomposition $E(\mathcal{K}_{3})$ of $\mathcal{K}_{3}$ and that their characteristic maps are $p \circ \Phi_{(a_{1},a_{2})}$. 
\end{proof}

\begin{definition}
We denote the cell complex $\mathcal{K}_{3}/\sim$ by $\mathcal{D}(\cP,\mathbf{p})$.
\end{definition}

\begin{theorem} \label{embd}
The cell complex $\mathcal{D}(\cP,\mathbf{p})$ is a CW-complex whose 2-skeleton is $\mathcal{D}(\cP) \sqcup F. \mathbf{p}. F$.
\end{theorem}
\begin{proof}
To prove that $\mathcal{D}(\cP,\mathbf{p})$ is a CW-complex it is enough from Proposition \ref{c2cw} to see that $\mathcal{D}(\cP,\mathbf{p})$ is Hausdorff. This will be done in several stages. First, consider the composition of maps
$$\xymatrix{ \mathcal{D}(\cP) \sqcup F. \mathbf{p}.  F \ar[r]^(0.70){\iota_{1}} &  \mathcal{K}_{3} \ar[r]^(0.35){p} & \mathcal{D}(\cP,\mathbf{p})}
$$
where $\iota_{1}(x)=(x,\lambda)$ and $p$ is the quotient map. We denote this composition by $\Psi$. The map $\Psi$ is continuous since $\iota_{1}$ is the homeomorphism that identifies $\mathcal{D}(\cP) \sqcup F. \mathbf{p}.  F$ with the sub complex $(\mathcal{D}(\cP) \sqcup F. \mathbf{p}.  F) \times \{\lambda\}$ of $\mathcal{K}_{3}$, and $p$ is the quotient map which is continuous by definition. Now we prove that $\Psi$ is an injective map. Let $x,y \in \mathcal{D}(\cP) \sqcup F. \mathbf{p}.  F$ such that $p(x,\lambda)=p(y, \lambda)$, hence $(x,\lambda) \sim (y, \lambda)$. Let $a\setminus \partial a$ and $b\setminus \partial b$ be the unique open cells such that $x \in a\setminus \partial a$ and $y \in b\setminus \partial b$. Also we note that the only open cell which contain $\lambda$ is $\lambda$ itself so the unique open cells in $\mathcal{K}_{3}$ which contain $(x,\lambda)$ and $(y, \lambda)$ are respectively $(a,\lambda) \setminus \partial (a, \lambda)$ and $(b,\lambda) \setminus \partial (b, \lambda)$. But from Definition \ref{relK} we must have that $(a,\lambda) \parallel (b, \lambda)$. Definition \ref{prl} implies that either $(a,\lambda)=(\sigma_{1} \otimes \sigma_{2}, \sigma_{3})$ and $(b, \lambda)=(\sigma_{1}, \sigma_{2} \otimes \sigma_{3})$, or $(a,\lambda)=(\sigma_{1}, \sigma_{2} \otimes \sigma_{3})$ and $(b, \lambda)=(\sigma_{1} \otimes \sigma_{2}, \sigma_{3})$. In the first case it follows that $\sigma_{2}=\sigma_{3}=\lambda$ and $a=\sigma_{1}=b$. Similarly in the second case we see that again $a=b$. Now from Definition \ref{relK} we have that 
\begin{align*}
(y, \lambda)&= \Phi_{(b,\lambda)}(\Phi^{-1}_{(a, \lambda)}(x,\lambda))\\
&= \Phi_{(a,\lambda)}(\Phi^{-1}_{(a, \lambda)}(x,\lambda)) \tag{since $a=b$}\\
&= (x,\lambda),
\end{align*}
therefore $(x,\lambda)=(y,\lambda)$. The fact that $\Psi$ is injective means in particular that the restriction of $p$ on $(\mathcal{D}(\cP) \sqcup F. \mathbf{p}.  F) \times \{\lambda\}$ is injective.

In the second stage we show that the subspace $\Psi(\mathcal{D}(\cP) \sqcup F. \mathbf{p}. F)$ of $\mathcal{D}(\cP,\mathbf{p})$ is Hausdorff. Let $p(x,\lambda) \neq p(y, \lambda)$ be two distinct points in $\Psi(\mathcal{D}(\cP) \sqcup F. \mathbf{p}. F)$, hence $(x,\lambda) \neq (y,\lambda)$. But $(\mathcal{D}(\cP) \sqcup F. \mathbf{p}.  F) \times \{\lambda\}$ is Hausdorff, therefore we can find closed subsets $U,V$ of $(\mathcal{D}(\cP) \sqcup F. \mathbf{p}.  F) \times \{\lambda\}$ such that $(x,\lambda) \in U$, $(y, \lambda)\in V$ and $U \cap V=\emptyset$. Consider now the subsets $p(U)$ and $p(V)$ of $\Psi(\mathcal{D}(\cP) \sqcup F. \mathbf{p}. F)$. It is obvious that $p(x,\lambda) \in p(U)$ and $p(y,\lambda) \in p(V)$. These subsets are disjoint for if $p(z, \lambda) \in p(U)\cap p(V)$, then there are $(z_{u}, \lambda) \in U$ and $(z_{v}, \lambda) \in V$ such that
$$p(z_{u}, \lambda)=p(z,\lambda)=p(z_{v}, \lambda).$$
From the first stage of the proof we see that $(z_{u}, \lambda)=(z_{v}, \lambda)$ contradicting the assumption that $U \cap V=\emptyset$. Next we prove that both $p(U)$ and $p(V)$ are closed in $\Psi(\mathcal{D}(\cP) \sqcup F. \mathbf{p}. F)$. To prove for instance that $p(U)$ is closed it suffices to prove that $p^{-1}(p(U))$ is closed in the CW- complex $\mathcal{K}_{3}$ which amounts to saying that for every cell $(a_{1},a_{2})$ in $\mathcal{K}_{3}$, the intersection 
$$\Phi^{-1}_{(a_{1},a_{2})}(p^{-1}(p(U))) \cap D^{n}$$
is  closed in $D^{n}$ where $n=\text{dim}(a_{1},a_{2})$. In the case when $n=0$, the claim is trivial since $D^{0}=\{\ast\}$ and in this case the intersection is either $\emptyset$ or $\{\ast\}$, therefore it is closed in $D^{0}$. Now assume that $0<n\leq 3$. We first prove that if $0<\text{dim}(a_{1},a_{2})=n \leq 2$, then there is some cell $(b_{1}, \lambda) \in \iota_{1}(\mathcal{D}(\cP) \sqcup F. \mathbf{p}.  F)$ such that
\begin{equation} \label{brz2}
\Phi^{-1}_{(a_{1},a_{2})}(p^{-1}(p(U))) \cap D^{n}=\Phi^{-1}_{(b_{1}, \lambda)}(U) \cap D^{n}.
\end{equation}
Before we prove (\ref{brz2}) we recall from Lemma \ref{p-par-b-1} that for every cell $(a_{1},a_{2})$ of dimension $\leq 2$ there is some cell $(b_{1}, \lambda)$ of the same dimension and parallel to the first. Let now $z \in \Phi^{-1}_{(a_{1},a_{2})}(p^{-1}(p(U))) \cap D^{n}$, then we have
\begin{align*}
z \in \Phi^{-1}_{(a_{1},a_{2})}(p^{-1}(p(U))) \cap D^{n} & \Leftrightarrow \Phi_{(a_{1},a_{2})}(z) \in p^{-1}(p(U)) \text{ and } z \in D^{n}\\
& \Leftrightarrow p( \Phi_{(a_{1},a_{2})}(z) )=p(u) \text{ where } u \in U \text{ and } z \in D^{n}.
\end{align*}
For the cell $(b_{1}, \lambda)$ and for the same $z \in D^{n}$, we have from (\ref{crutial}) that $p(\Phi_{(b_{1}, \lambda)}(z)) = p(\Phi_{(a_{1},a_{2})}(z))$. It follows now that $p(\Phi_{(b_{1}, \lambda)}(z))=p(u)$ and then the injectivity of $p$ on $\iota_{1}(\mathcal{D}(\cP) \sqcup F. \mathbf{p}.  F)$ implies that $\Phi_{(b_{1}, \lambda)}(z)=u$. This in turn implies that $z \in \Phi^{-1}_{(b_{1}, \lambda)}(U) \cap D^{n}$, proving that
$$\Phi^{-1}_{(a_{1},a_{2})}(p^{-1}(p(U))) \cap D^{n} \subseteq \Phi^{-1}_{(b_{1}, \lambda)}(U) \cap D^{n}.$$ 
To prove the reverse inclusion we observe that there are two cases possible. In the first case $\Phi^{-1}_{(b_{1}, \lambda)}(U) \cap D^{n}=\emptyset$. In this case the above inclusion immediately implies the equality (\ref{brz2}). In the second case, there is some $z \in \Phi^{-1}_{(b_{1}, \lambda)}(U) \cap D^{n}$, hence $\Phi_{(b_{1}, \lambda)}(z) \in U$. From (\ref{crutial}) again we have that $p(\Phi_{(a_{1},a_{2})}(z))=p(\Phi_{(b_{1}, \lambda)}(z))$, consequently $p(\Phi_{(a_{1},a_{2})}(z)) \in p(U)$, therefore $\Phi_{(a_{1},a_{2})}(z) \in p^{-1}(p(U))$ which implies that $z \in \Phi^{-1}_{(a_{1},a_{2})}(p^{-1}(p(U))) \cap D^{n}$ proving the reverse inclusion
$$\Phi^{-1}_{(b_{1}, \lambda)}(U) \cap D^{n} \subseteq  \Phi^{-1}_{(a_{1},a_{2})}(p^{-1}(p(U))) \cap D^{n}.$$
Since $U$ is closed, it follows that the right hand side of (\ref{brz2}) is a closed subset of $D^{n}$, hence so is the left hand side. Since for all cells $(a_{1},a_{2})$ of dimension $0\leq n\leq 2$ we have
$$\Phi^{-1}_{(a_{1},a_{2})}(p^{-1}(p(U))=\Phi^{-1}_{(a_{1},a_{2})}(p^{-1}(p(U)) \cap \mathcal{K}_{3}^{(2)}),$$
and from (\ref{brz2}) $\Phi^{-1}_{(a_{1},a_{2})}(p^{-1}(p(U))$ is closed, we derive that,
\begin{equation} \label{ind-2}
p^{-1}(p(U)) \cap \mathcal{K}_{3}^{(2)} \text{ is closed in } \mathcal{K}_{3}^{(2)}.
\end{equation}
Now we take care of the case when $\text{dim}(a_{1},a_{2})=3$. We prove first that
\begin{equation} \label{brz1}
\Phi^{-1}_{(a_{1},a_{2})}(p^{-1}(p(U))) \cap \text{Int}(D^{3})=\emptyset.
\end{equation}
Indeed, 
\begin{align*}
z \in \Phi^{-1}_{(a_{1},a_{2})}(p^{-1}(p(U))) \cap \text{Int}(D^{3}) & \Leftrightarrow \Phi_{(a_{1},a_{2})}(z) \in p^{-1}(p(U)) \text{ and } z \in \text{Int}(D^{3})\\
& \Leftrightarrow p(\Phi_{(a_{1},a_{2})}(z))= p(u) \text{ where }u \in U  \text{ and } z \in \text{Int}(D^{3})\\
& \Leftrightarrow \Phi_{(a_{1},a_{2})}(z) \sim u \text{ where }u \in U  \text{ and } z \in \text{Int}(D^{3}),
\end{align*}
but the last statement is false since $\Phi_{(a_{1},a_{2})}(z) \in (a_{1},a_{2}) \setminus \partial (a_{1},a_{2})$ while $u$ belongs to a cell of dimension $\leq 2$ and we know that no points of cells of dimension 3 can be identified with points of cells of lower dimensions. Further we have that
\begin{align*} 
\Phi^{-1}_{(a_{1},a_{2})}(p^{-1}(p(U))&=\Phi^{-1}_{(a_{1},a_{2})}(p^{-1}(p(U))) \cap D^{3} \\
&=(\Phi^{-1}_{(a_{1},a_{2})}(p^{-1}(p(U))) \cap S^{2}) \cup (\Phi^{-1}_{(a_{1},a_{2})}(p^{-1}(p(U))) \cap \text{Int}(D^{3}))\\
&= \Phi^{-1}_{(a_{1},a_{2})}(p^{-1}(p(U))) \cap S^{2} \tag{from (\ref{brz1})}\\
&=  \phi^{-1}_{(a_{1},a_{2})}(p^{-1}(p(U)) \cap \mathcal{K}_{3}^{(2)}), 
\end{align*}
where $\phi_{(a_{1},a_{2})}$ is the attaching map of $(a_{1},a_{2})$. But from (\ref{ind-2}) we have that $p^{-1}(p(U)) \cap \mathcal{K}_{3}^{(2)}$ is closed in $\mathcal{K}_{3}^{(2)}$ and then the continuity of $\phi_{(a_{1},a_{2})}$ implies that $\phi^{-1}_{(a_{1},a_{2})}(p^{-1}(p(U)) \cap \mathcal{K}_{3}^{(2)})$ is closed in $S^{2}$ consequently it is closed in $D^{3}$, hence so is $\Phi^{-1}_{(a_{1},a_{2})}(p^{-1}(p(U))$ which finally proves that $p^{-1}(p(U))$ is closed in $\mathcal{K}_{3}$.

In the third stage we observe that $\mathcal{D}(\cP,\mathbf{p})$ is obtained from the Hausdorff space $\Psi(\mathcal{D}(\cP) \sqcup F. \mathbf{p}. F)$ by attaching to it 3-cells $p(f, \sigma)$ or $p(\sigma,f)$ where $\sigma$ is a 2-dimensional cell and $f$ is a 1-cell. Proposition \ref{hd} implies that $\mathcal{D}(\cP,\mathbf{p})$ is Hausdorff.

Finally, since $\iota_{1}$ is obviously a regular map and from Proposition \ref{c2cw} $p$ is also regular, then so is $\Psi$. Proposition \ref{idsc} allows us to regard $\mathcal{D}(\cP) \sqcup F. \mathbf{p}.  F$ as the 2-dimensional skeleton of $\mathcal{D}(\cP,\mathbf{p})$ via $\Psi$. 
\end{proof}

\begin{remark} \label{remsub}
The result of Theorem \ref{embd} allows us to use the same notations for cells of $\mathcal{D}(\cP,\mathbf{p})$ of dimension at most 2 as those of $\mathcal{D}(\cP) \sqcup F. \mathbf{p}.  F$. While for cells of dimension 3 which are of the form $p(f, \sigma)$ or $p(\sigma,f)$ where $\sigma$ is a 2-dimensional cell and $f$ is a positive edge, we will use notations $[f, \sigma]$ and $[\sigma, f]$ respectively.
\end{remark}

\section{The cellular boundary maps for $\mathcal{D}(\cP,\mathbf{p})$}

For each $i=0,1,2,3$ we let $C_{i}$ be the free abelian group on the set of $i$-cells of $\mathcal{D}(\cP,\mathbf{p})$. We will define the boundary maps of the cellular chain complex
\begin{equation}\label{cchc}
\xymatrix{ C_{3} \ar[r]^{\partial_{3}} & C_{2} \ar[r]^{\partial_{2}} & C_{1} \ar[r]^{\partial_{1}} & C_{0} \ar[r] & 0.}
\end{equation}
The standard definition of $\partial_{1}$ is made by setting for every 1-cell $e$,
$$\partial_{1}(e)=\tau(e)-\iota(e).$$
This differs from the definition of $\partial_{1}$ given in \cite{LDHM1} only by the sign. Since in several papers like \cite{LDHM1}, \cite{LDHM2}, \cite{SES}, the interest of studying (\ref{cchc}) is the first homology group of $\mathcal{D}(\cP)$ which is invariant of the change of sign of $\partial_{1}$, all the results obtained there for $H_{1}(\mathcal{D}(\cP))$ remain valid if the definition of $\partial_{1}$ is the standard one given as above.

The definition of $\partial_{2}$ on 2-cells $[e,f]$ is made in \cite{Sth} (see also \cite{SES}) by
\begin{align}
\partial_{2}[e,f]&=e.(\iota (f) - \tau (f))-(\iota (e)- \tau (e)).f \notag\\
&= (e. \iota(f)+ \tau(e).f)- (\iota(e).f+ e. \tau(f)). && \label{bsq}
\end{align}
We will see below that this definition is not arbitrary but follows in the standard fashion from the way the attaching map of $[e,f]$ is defined. Assume that we have oriented the boundary of the attaching disk in (\ref{amsq}) in a counterclockwise way. To compute the incidence numbers of the edges on the boundary of $[e,f]$, we collapse the entire complex to a point apart from the open edge and see if the orientation of the edge agrees with that of the boundary of the disk. If it agrees, then the incidence number is 1 times the number of times $\partial [e,f]$ passes through that edge. If it doesn't, then it is -1 times that number. In our case the incidence numbers of boundary edges of $[e,f]$ are as follows. The edges $e. \iota(f)$ and $\tau(e).f$ have incidence numbers 1, and those $\iota(e).f$ and $e. \tau(f)$ have incidence numbers -1. That's the reason the cellular boundary formula for $[e,f]$ is the one stated above. In a similar fashion we can compute the cellular boundary formula for some 2-cell $\sigma$ from $F.\mathbf{p}.F$ which is attached along the loop $e^{\varepsilon_{1}}_{1} \circ \dots  \circ e_{n}^{\varepsilon_{n}}$ as depicted in (\ref{amp4}), and get
\begin{equation} \label{bp}
\partial_{2} (\sigma)= \varepsilon_{1}e_{1}+ \dots + \varepsilon_{n}e_{n},
\end{equation}
which is the same as the one declared in \cite{SES}. Now we compute the formula for $\partial_{3}$. Let
$$\xymatrix{C_{3}(\mathcal{K}_{3}) \ar[r]^{\delta_{3}} & C_{2}(\mathcal{K}_{3}) \ar[r]^{\delta_{2}} & C_{1}(\mathcal{K}_{3}) \ar[r]^{\delta_{1}} & C_{0}(\mathcal{K}_{3}) \ar[r] & 0}$$
be the chain complex for the 3-skeleton of $\Delta \times \Delta$. From Proposition 3B.1, p 269 of \cite{AH} we have in particular that for 3-cells $(f, \sigma)$ and $(\sigma, f)$ the following hold true.
\begin{equation} \label{x1}
\delta_{3}(f,\sigma)= \partial_{1}(f) \times \sigma + (-1)^{1} f \times \partial_{2} (\sigma)
\end{equation}
and
\begin{equation} \label{x2}
\delta_{3}(\sigma,f)= \partial_{2}(\sigma) \times f + (-1)^{2} \sigma \times \partial_{1} (f).
\end{equation}
The cellular map $p: \mathcal{K}_{3} \rightarrow \mathcal{D}(\cP,\mathbf{p})$ induces a mapping $p^{\#}$ of chain complexes
\begin{equation} \label{chm3}
\xymatrix{C_{3}(\mathcal{K}_{3}) \ar[d]_{p^{\#}_{3}} \ar[r]^{\delta_{3}} & C_{2}(\mathcal{K}_{3}) \ar[d]_{p^{\#}_{2}} \ar[r]^{\delta_{2}} & C_{1}(\mathcal{K}_{3}) \ar[d]_{p^{\#}_{1}} \ar[r]^{\delta_{1}} & C_{0}(\mathcal{K}_{3}) \ar[d]_{p^{\#}_{0}} \ar[r] & 0 \ar@{}[d]\\
C_{3} \ar[r]^{\partial_{3}} & C_{2} \ar[r]^{\partial_{2}} & C_{1} \ar[r]^{\partial_{1}} & C_{0} \ar[r] & 0}
\end{equation}
It is possible to determine the value of $\partial_{3}$ on free generators of $C_{3}$, which are 3-cells of the form $[f,\sigma]=p(f,\sigma)$ and $[\sigma, f]=p(\sigma, f)$, in terms of $p^{\#}_{3}$, $p^{\#}_{2}$ and $\delta_{3}$ alone. To see this we first note that for every cell $(a_{1},a_{2})$ of dimension $n=1,2,3$,
\begin{equation} \label{pth}
p_{n}^{\#}(a_{1},a_{2})=p(a_{1},a_{2}).
\end{equation}
Indeed, from Definition \ref{R} we have in general that
\begin{equation*}
p_{\#}(a_{1},a_{2})=\sum_{\mu \in M_{n}}[p:(a_{1},a_{2}):p(b^{\mu}_{1},b^{\mu}_{2})]p(b^{\mu}_{1},b^{\mu}_{2}),
\end{equation*}
where $M_{n}$ in the set of distinct $n$-cells of $\mathcal{D}(\cP, \mathbf{p})$, but from Corollary \ref{R2} the only $p(b^{\mu}_{1},b^{\mu}_{2})$ for which $[p:(a_{1},a_{2}):p(b^{\mu}_{1},b^{\mu}_{2})] \neq 0$ is $p(a_{1},a_{2})$, therefore the sum above has only one term which is $[p:(a_{1},a_{2}):p(a_{1},a_{2})]p(a_{1},a_{2})$. To compute the incidence number $[p:(a_{1},a_{2}):p(a_{1},a_{2})]$ we apply Proposition \ref{R1} by taking $X=\mathcal{K}_{3}$, $Y=\mathcal{D}(\cP, \mathbf{p})$, $f=p$, $\sigma=(a_{1},a_{2})$ and $\tau=p(a_{1},a_{2})$. Taking now into account that $\Phi_{p(a_{1},a_{2})}=p \circ \Phi_{(a_{1},a_{2})}$ we see that the mapping $F((a_{1},a_{2}): p(a_{1},a_{2}))$ of diagram (\ref{Rr1}) in our case is the composite map
$$\xymatrix{ D^{n}/ S^{n-1} \ar[rr]^(0.35){\bar{\Phi}_{p(a_{1},a_{2})}} && \mathcal{D}(\cP, \mathbf{p})^{(n)}/\mathcal{D}(\cP, \mathbf{p})^{(n-1)} \ar[rr]^(0.53){j_{p(a_{1},a_{2})}} && p(a_{1},a_{2})/\partial p(a_{1},a_{2}),}$$
which is orientation preserving, therefore $[p:(a_{1},a_{2}):p(a_{1},a_{2})]=1$ and then we have (\ref{pth}).

Let us assume that $\sigma$ is any 2-cell from $\mathcal{D}(\cP) \sqcup F. \mathbf{p}.F$ and that
$$\partial_{2} (\sigma)= \varepsilon_{1}e_{1}+ \dots + \varepsilon_{n}e_{n}.$$
We already know from (\ref{bsq}) and (\ref{bp}) that $\partial_{2} (\sigma)$ is always of the above form. Now for every 3-cell $p(f,\sigma)=[f,\sigma]$ we see that
\begin{align*}
\partial_{3}[f,\sigma]&= \partial_{3}(p(f,\sigma))\\
&=\partial_{3}(p^{\#}_{3}(f,\sigma)) \tag{by (\ref{pth})}\\
&=p^{\#}_{2}(\delta_{3}(f,\sigma))\\
&= p^{\#}_{2}\left( (\tau(f), \sigma)-(\iota(f),\sigma)-\sum_{i=1}^{n} \varepsilon_{i} (f, e_{i}) \right)\tag{by (\ref{x1})}\\
&= p(\tau(f), \sigma)- p(\iota(f),\sigma)-\sum_{i=1}^{n} \varepsilon_{i} p(f, e_{i}) \tag{by (\ref{pth})}\\
&= (\tau(f).\sigma - \iota(f).\sigma)-\sum_{i=1}^{n} \varepsilon_{i}[f, e_{i}]. \tag{by Remark \ref{remsub}}
\end{align*}
When the 3-cell is $p(\sigma,f)=[\sigma,f]$ we have
\begin{align*}
\partial_{3}[\sigma,f]&= \partial_{3}(p(\sigma,f))\\
&=\partial_{3}(p^{\#}_{3}(\sigma, f)) \tag{by (\ref{pth})}\\
&=p^{\#}_{2}(\delta_{3}(\sigma,f))\\
&= p^{\#}_{2}\left(\sum_{i=1}^{n} \varepsilon_{i}(e_{i},f)+( (\sigma, \tau(f))-(\sigma, \iota(f))\right) \tag{by (\ref{x2})}\\
&= \sum_{i=1}^{n} \varepsilon_{i} p(e_{i},f)+( p(\sigma, \tau(f))-p(\sigma, \iota(f)) \tag{by (\ref{pth})}\\
&= \sum_{i=1}^{n} \varepsilon_{i}[e_{i},f]+( \sigma. \tau(f)-\sigma.\iota(f)) \tag{by Remark \ref{remsub}}\\
&= ( \sigma. \tau(f)-\sigma.\iota(f))+ \sum_{i=1}^{n} \varepsilon_{i}[e_{i},f].
\end{align*}
We remark that the actual value of $\partial_{3}$ differs from the one given in \cite{SES} only by the sign, having thus no effect on the results related to $H_{2}(\mathcal{D}(\cP,\mathbf{p}))$ proved in \cite{SES}.

\end{document}